\documentclass[10pt,leqno]{amsart}
\usepackage{amssymb,amsfonts,amsmath,amsopn,amstext,amscd,extarrows,latexsym,xy,hyperref,mathrsfs,verbatim}
\usepackage{epsfig}
\input xy
\xyoption{all}

\allowdisplaybreaks[1]

\theoremstyle{remark}
\newtheorem{para}{\bf}[section]

\newtheorem{rmk}[para]{\bf Remark}
\newtheorem{assumption}[para]{\bf Assumption}

\theoremstyle{definition}
\newtheorem{exam}[para]{\bf Example}

\newtheorem{dfn}[para]{\bf Definition}

\theoremstyle{plain}
\newtheorem{thm}[para]{\bf Theorem}
\newtheorem{lemma}[para]{\bf Lemma}
\newtheorem{sublemma}[para]{\bf Sublemma}
\newtheorem{cor}[para]{\bf Corollary}
\newtheorem{prop}[para]{\bf Proposition}

\newenvironment{numequation}
{\addtocounter{enumi}{1}\begin{equation}}{\end{equation}}

\newcommand{\vep}{\varepsilon}

\newcommand{\vphi}{\varphi}

\newcommand{\cE}{{\mathcal E}}
\newcommand{\cF}{{\mathcal F}}

\newcommand{\cH}{{\mathcal H}}
\newcommand{\cI}{{\mathcal I}}
\newcommand{\cJ}{{\mathcal J}}

\newcommand{\cL}{{\mathcal L}}

\newcommand{\cO}{{\mathcal O}}

\newcommand{\cX}{{\mathcal X}}

\newcommand{\bB}{{\bf B}}

\newcommand{\bG}{{\bf G}}

\newcommand{\bL}{{\bf L}}
\newcommand{\bM}{{\bf M}}
\newcommand{\bN}{{\bf N}}

\newcommand{\bP}{{\bf P}}

\newcommand{\bT}{{\bf T}}
\newcommand{\bU}{{\bf U}}

\newcommand{\bbG}{{\mathbb G}}

\newcommand{\bbN}{{\mathbb N}}

\newcommand{\bbP}{{\mathbb P}}
\newcommand{\bbQ}{{\mathbb Q}}
\newcommand{\bbR}{{\mathbb R}}

\newcommand{\bbZ}{{\mathbb Z}}

\newcommand{\fra}{{\mathfrak a}}
\newcommand{\frb}{{\mathfrak b}}

\newcommand{\frd}{{\mathfrak d}}

\newcommand{\frg}{{\mathfrak g}}

\newcommand{\frl}{{\mathfrak l}}
\renewcommand{\frm}{{\mathfrak m}}

\newcommand{\frp}{{\mathfrak p}}
\newcommand{\frq}{{\mathfrak q}}

\newcommand{\frt}{{\mathfrak t}}
\newcommand{\fru}{{\mathfrak u}}

\newcommand{\frx}{{\mathfrak x}}
\newcommand{\fry}{{\mathfrak y}}
\newcommand{\frz}{{\mathfrak z}}

\newcommand{\GL}{{\rm GL}}
\newcommand{\Hom}{{\rm Hom}}

\newcommand{\ad}{{\rm ad}}
\newcommand{\Ad}{{\rm Ad}}
\newcommand{\bksl}{\backslash}

\newcommand{\hra}{\hookrightarrow}

\newcommand{\ind}{{\rm ind}}
\newcommand{\Ind}{{\rm Ind}}
\newcommand{\Lie}{{\rm Lie}}
\newcommand{\lra}{\longrightarrow}
\newcommand{\midc}{\;|\;}

\newcommand{\Qp}{{\bbQ}_p}
\newcommand{\ra}{\rightarrow}
\newcommand{\Rep}{{\rm Rep}}

\newcommand{\Spec}{{\rm Spec}}

\newcommand{\sub}{\subset}
\newcommand{\supp}{{\rm supp}}
\newcommand{\alg}{{\rm alg}}
\newcommand{\Oa}{{\cO_\alg}}
\newcommand\Zp{{{\bbZ}_p}}

\newcommand{\Pf}{{\it Proof. }}
\renewcommand{\qed}{{\hfill{\space} $\Box$}}

\begin{document}

\title{On Jordan-H\"older series of some locally analytic representations}
\author{Sascha Orlik}
\address{Fachbereich C - Mathematik und Naturwissenschaften, Bergische Universit\"at Wuppertal,
Gau{\ss}stra\ss{}e 20, D-42119 Wuppertal, Germany}
\email{orlik@math.uni-wuppertal.de}
\author{Matthias Strauch}
\address{Indiana University, Department of Mathematics, Rawles Hall, Bloomington, IN 47405, USA}
\email{mstrauch@indiana.edu}
\thanks{M. S. would like to acknowledge the support of the National Science Foundation (award numbers
DMS-0902103 and DMS-1202303).}
\maketitle

\normalsize

\begin{abstract}
Let $G$ be a split reductive $p$-adic group. This paper is about the Jordan-H\"older series of locally analytic $G$-representations which are induced from locally algebraic representations of a parabolic subgroup $P \subset G$. We construct for every representation $M$ of $\Lie(G)$ in the {\rm BGG}-category $\cO$, which is equipped with an algebraic $P$-action, and for every smooth $P$-representation $V$, a locally analytic representation $\cF^G_P(M,V)$ of $G$. This gives rise to a bi-functor to the category of locally analytic representations. We prove that it is exact and give a criterion for the topological irreducibility of $\cF^G_P(M,V)$ in terms of $M$ and $V$.
\end{abstract}

\tableofcontents

\section{Introduction}
Let $L$ be a finite extension of $\Qp$ and let $G = \bG(L)$ be the group of $L$-valued points of a split connected reductive algebraic group $\bG$ over $L$. This paper is about the Jordan-H\"older series of certain classes of locally analytic $G$-representations. In particular, our treatment includes those representations which are induced from locally algebraic representations of a parabolic subgroup. 

The theory of locally analytic representations was introduced by P. Schneider and J. Teitelbaum in \cite{ST1}. Such representations 
appear in the study of vector bundles on $p$-adic period domains \cite{ST3, O} and in the $p$-adic Langlands program. In \cite{S2,Te} 
the authors pose the question of determining  irreducible parabolically induced locally-analytic representations. More generally, one would like to know the Jordan-H\"older series of such an object. Here we give a partial answer to these questions. As an application we consider the $\GL_{d+1}$-equivariant filtration constructed in \cite{O} on the Fr\'echet space $H^0(\cX,\cL)$, where $\cL$ is a homogeneous line bundle on projective space $\bbP^d_K$ and $\cX \subset \bbP^d_K$ is the Drinfeld half space. It is shown here that this filtration is either a Jordan-H\"older series, or that it has a simple refinement (which we describe explicitly) that is a Jordan-H\"older series.


In order to explain our main construction and results, we introduce some notation.
Fix a Borel subgroup $\bB \sub \bG$, a split maximal torus $\bT \sub \bB$, a standard parabolic subgroup $\bP \supset \bB$, and let $\frb = \Lie(\bB)$, $\frt = \Lie(\bT)$, $\frp = \Lie(\bP)$, and $\frg = \Lie(G)$ be their Lie algebras. The corresponding groups of $L$-valued points are denoted by $B = \bB(L)$, $T = \bT(L)$, and $P = \bP(L)$. Furthermore, we fix once and for all a finite extension $K$ of $L$ which will be our field of coefficients. The base change of an $L$-vector space (or scheme over $L$) to $K$ will always be denoted by the subscript ${}_K$, e.g., $\frg_K = \frg \otimes_L K$, $\bG_K = \bG \times_{\Spec(L)} \Spec(K)$.


We will study locally analytic representations which are in the image of certain bi-functors

\vspace{-0.3cm}
$$\cF^G_P(\cdot \,,\, \cdot): \cO^\frp_{\alg} \times \Rep^{\infty,{\rm adm}}_K(L_P) \longrightarrow  \Rep_K^{\ell a}(G) \,$$

\noindent depending on the standard parabolic subgroup $P \subset G$. Here $\Rep^{\infty,{\rm adm}}_K(L_P)$ is the category of smooth 
admissible representations of the Levi subgroup $L_P \sub P$  on $K$-vector spaces. Further $\Rep_K^{\ell a}(G)$ is the category of locally analytic representations of $G$ on $K$-vector spaces.  By $\cO = \cO^\frb$ we denote the following category of $U(\frg_K)$-modules\footnote{this is an adaptation of the {\rm BGG}-category $\cO$ to our setting in which the coefficient field is not algebraically closed}: objects are $U(\frg_K)$-modules $M$ satisfying the following properties:

\begin{enumerate}
\item $M$ is finitely generated as $U(\frg_K)$-module.\smallskip
\item $M$ decomposes as a direct sum of one-dimensional $U(\frt_K)$-modules. \smallskip
\item The action of $\frb_K$ on $M$ is locally finite, i.e. for every $m \in M$, the subspace $U(\frb_K)\cdot m \subset M$ is finite-dimensional over $K$.
\end{enumerate}

\noindent For a standard parabolic subalgebra $\frp \supset \frb$ we let $\cO^\frp$ be the subcategory of those modules $M$ in $\cO$ on which $\frp_K$ acts locally finitely. In $\cO$ we consider the subcategory $\cO_\alg$ of modules $M$ which have the additional property: 


(4) All weights of $\frt_K$ on $M$ lie in the image of the natural map $X^*(\bT) = \Hom(\bT,\bbG_m) \ra \frt_K^* = \Hom_L(\frt,K)$.

\vskip4pt
\noindent We will show, cf. \ref{T-alg implies L-alg}, that modules $M$ in $\cO_\alg$ on which $\frp_K$ acts locally finitely have the property that the $\frp_K$-action lifts uniquely to a locally finite-dimensional algebraic $\bP_K$-representation, and we denote this subcategory by $\cO^\frp_\alg$, i.e., $\cO^\frp_\alg = \cO_\alg \cap \cO^\frp$. 

Given a module $M$ in $\cO^\frp_\alg$, there is a finite-dimensional $K$-subspace $W \sub M$ which is stable under $U(\frp_K)$ and which generates $M$ as $U(\frg_K)$-module. Let $\frd$ be the kernel of the canonical surjection $U(\frg_K) \otimes_{U(\frp_K)} W \ra M$,
and let $W'$ be the dual representation. By what we have said above, $W'$ lifts uniquely to an algebraic representation of $P$. Then we consider the locally analytic induced representation
$$\Ind^G_P(W') = \{f \in C^{an}(G,W') \midc \forall g \in G, p \in P: \, f(gp) = p^{-1}\cdot f(g) \, \} \,,$$

\noindent where $C^{an}(G,W')=C_L^{an}(G,W')$ is the $K$-vector space of $L$-locally analytic $W'$-valued functions on $G$. The action of $G$ on $\Ind^G_P(W')$ is given by left translation: $(g \cdot f)(x) = f(g^{-1}x)$. There is a $C^{an}(G,K)$-valued pairing

$$\begin{array}{rccc}
\langle \cdot , \cdot \rangle_{C^{an}(G,K)}: & \left(U(\frg_K) \otimes_{U(\frp_K)} W \right) \otimes_K \Ind^G_P(W') & \lra & C^{an}(G,K) \\
&&&\\
& (\fry \otimes w) \otimes f & \mapsto & \Big[ g \mapsto \big((\fry \cdot_r f)(g)\big)(w)\Big]
\end{array}$$

\vskip2pt

\noindent where, for $\frx \in \frg$ we have, by definition,
$$(\frx \cdot_r f)(g) = \frac{d}{dt}f(g\exp(t\frx))|_{t=0} \,.$$

\vskip2pt
\noindent The common kernel of all $\langle \frz, \cdot \rangle_{C^{an}(G,K)}$, with $\frz \in \frd$, is a closed subrepresentation

\vspace{-0.3cm}
$$\cF^G_P(M) = \Ind^G_P(W')^\frd = \{ f \in \Ind^G_P(W') \midc \mbox{ for all } \frz \in \frd : \langle \frz, f \rangle_{C^{an}(G,K)} = 0 \} \,.$$

\vskip2pt

\noindent More generally, for a smooth admissible $L_P$-representation $V$ we define the representation

\vspace{-0.3cm}
$$\cF^G_P(M,V) = \Ind^G_P(W' \otimes_K V)^\frd = \{ f \in \Ind^G_P(W' \otimes_K V) \midc \mbox{ for all } \frz \in \frd : \langle \frz, f \rangle_{C^{an}(G,V)} = 0 \}$$


\noindent similarly as before (here, the pairing $\langle \cdot , \cdot \rangle_{C^{an}(G,V)}$ takes values in $C^{an}(G,V)$). The representation $\cF^G_P(M,V)$ is an admissible locally analytic representation.


Let $M$ be an object of $\cO^\frp$. We call the parabolic subalgebra $\frp$ {\it maximal for $M$}, if $M$ does not lie in a subcategory $\cO^\frq$ with a parabolic subalgebra $\frq$ properly containing $\frp$. It follows from \cite[sec. 9.4]{H2} that for every object $M$ of $\cO^\frp$ there is unique parabolic subalgebra $\frq \supset \frp$ which is maximal for $M$.
The same definition applies for objects in the subcategory $\cO^\frp_{\alg}$ in which case we also say that the parabolic subgroup $P$ {\it maximal for $M$}.


The main result of this paper is the following:

\vskip5pt

\noindent {\bf Theorem.} {\it (i) $\cF^G_P$ is functorial in both arguments: contravariant in $M$ and covariant in $V$.

\vskip5pt

\noindent (ii) $\cF^G_P$ is exact in both arguments.

\vskip5pt

\noindent (iii) If $Q \supset P$ is a parabolic subgroup and $M$ is an object of $\cO^\frq_\alg$ 
(and hence, in particular, an object of $\cO^\frp_\alg$), then

\vspace{-0.3cm}
$$\cF^G_P(M,V) = \cF^G_Q(M,\ind^{L_Q}_{L_P(L_Q \cap U_P)}(V))$$


\noindent for all smooth admissible $L_P$-representations $V$. Here, $\ind^{L_Q}_{L_P(L_Q \cap U_P)}(V)=\ind^Q_P(V)$ denotes the smooth induction of the representation $V$.

\vskip2pt

\noindent (iv) $\cF^G_P(M,V)$ is topologically irreducible if and only if the following conditions are both satisfied:

\vskip5pt

\noindent - $M$ is simple,

\noindent - if $Q \supset P$ is maximal for $M$, then the smooth representation $\ind^{L_Q}_{L_P(L_Q \cap U_P)}(V)$ is irreducible 
as an $L_Q$-representation.

\vskip5pt

\noindent In (iv) we assume that the residue field characteristic $p$ of $L$ is odd, if the root system $\Phi = \Phi(\frg,\frt)$ has irreducible components of type $B$, $C$ or $F_4$, and if $\Phi$ has irreducible components of type $G_2$ we assume that $p > 3$.}

\vskip5pt

\noindent  {\it Remark.} The assumptions on the residue field characteristic $p$ of $L$ are imposed because of our method of proof, but might not be necessary. That is, we do not know of any example where the fourth statement fails because $p \le 3$.

\vskip5pt

Our main result covers a particular case of the irreducibility result shown in \cite{OS}.
In loc.cit. we consider arbitrary reductive groups and arbitrary locally analytic finite-dimensional $L_P$-representations $W$, but where, on the other
hand, no differential equations appear.
Using as an input Jordan-H\"older series for $M$ and $V$, and for possibly occurring smooth induced representations of the form $\ind^{L_Q}_{L_P(L_Q \cap U_P)}(V)$, it is a rather formal exercise to determine a Jordan-H\"older series for $\cF^G_P(M,V)$, cf. section \ref{Composition series}. Indeed, we concretize the determination of the composition factors for locally analytic representations of the form $\Ind^G_P(W')^\frd$. Finally, we study for a homogenous line bundle $\cL$ on projective space $\bbP^d_k$ and the Drinfeld space $\cX\subset \bbP^d_k$, the space of sections $H^0(\cX,\cL)$ which is a Fr\'echet space equipped with a continuous action of $G=\GL_{d+1}(K)$. More precisely, we consider the filtration
$H^0(\cX,\cL)=\cL(\cX)^0 \supset \cL(\cX)^{1} \supset \cdots \supset \cL(\cX)^{d-1} \supset \cL(\cX)^d =  H^0(\bbP^d,\cL)$ constructed in \cite{O} consisting of closed $G$-subspaces. We show that for $i\geq 1,$ the dual space $(\cL(\cX)^{i}/\cL(\cX)^{i+1})'$ is a locally analytic $G$-representation which lies in the image of the functor $\cF^{G}_{P_{(i+1,d-i)}}$. Here $P_{(i+1,d-i)}$ is the maximal standard parabolic subgroup
corresponding to the decomposition $(i+1,d-i)$ of $d+1$.
On the other hand, the dual of $\cL(\cX)^{0}/\cL(\cX)^{1}$ sits inside an extension
$$ 0 \rightarrow v^{G}_{B} \otimes_K H^{d}(\bbP^d_K,\cL)' \rightarrow (\cL(\cX)^{0}/\cL(X)^{1})' \rightarrow  \cF^{G}_{P_{(1,d)}}(M) \rightarrow 0 $$
for some simple object $M\in \cO^{\frp_{(1,d)}}.$ Here $v^G_B$ is the smooth Steinberg representation and $H^{d}(\bbP^d_K,\cL)$ is a simple finite-dimensional algebraic $G$-module.  By using our main result, we prove that all the locally analytic representations above are irreducible.


Independently of our approach, O. Jones considers in \cite{Jo} the relation between the Bernstein-Gelfand-Gelfand resolution and 
locally analytic principal series representations of $G$ (and, more generally, of subgroups of $G$ possessing an Iwahori decomposition). 
In \cite[Thm. 26]{Jo} he proves the existence of an exact sequence which coincides with the exact sequence \ref{lcBGG} in 
paragraph \ref{locanBGG}. The results about irreducibility in section \ref{irredresults} can then be used to decide if this 
sequence is actually a composition series, and, if not, how it can be refined to give such a series.

\vskip12pt

{\it Notation and conventions:} We denote by $p$ a prime number and consider fields $L \sub K$ which are both finite extensions of $\Qp$. 
Let $O_L$ and $O_K$ be the rings of integers of $L$, resp. $K$, and let $|\cdot |_K$ be the absolute value on $K$ such that $|p|_K = p^{-1}$. The field $L$ is our ''base field'', whereas we consider $K$ as our ''coefficient field''. For a locally convex $K$-vector space $V$ we denote by $V'_b$ its strong dual, i.e., the $K$-vector space of continuous linear forms equipped with the strong topology of bounded convergence. Sometimes, in particular when $V$ is finite-dimensional, we simplify notation and write $V'$ instead of $V'_b$. All finite-dimensional $K$-vector spaces are equipped with the unique Hausdorff locally convex topology.

We let $\bG_0$ be a split reductive group scheme over $O_L$ and $\bT_0 \sub \bB_0 \sub \bG_0$ a maximal split torus and a Borel subgroup scheme, respectively. We denote by $\bG$, $\bB$, $\bT$ the base change of $\bG_0$, $\bB_0$ and $\bT_0$ to $L$. By $G_0 = \bG_0(O_L)$, $B_0 = \bB_0(O_L)$, etc., and $G = \bG(L)$, $B = \bB(L)$, etc., we denote the corresponding groups of $O_L$-valued points and $L$-valued points, respectively. Standard parabolic subgroups of $\bG$ (resp. $G$) are those which contain $\bB$ (resp. $B$). For each standard parabolic subgroup $\bP$ (or $P$) we let $\bL_\bP$ (or $L_P$) be the unique Levi subgroup which contains $\bT$ (resp. $T$). Finally, Gothic letters $\frg$, $\frp$, etc., will denote the Lie algebras of $\bG$, $\bP$, etc.: $\frg = \Lie(\bG)$, $\frt = \Lie(\bT)$, $\frb = \Lie(\bB)$, $\frp = \Lie(\bP)$, $\frl_P = \Lie(\bL_\bP)$, etc.. Base change to $K$ is usually denoted by the subscript ${}_K$, for instance, $\frg_K = \frg \otimes_L K$.


We make the general convention that we denote by $U(\frg)$, $U(\frp)$, etc., the corresponding enveloping algebras, {\it after base change to $K$}, i.e., what would be usually denoted by $U(\frg) \otimes_L K$, $U(\frp) \otimes_L K$ etc.
Similarly, we use the abbreviations $D(G) = D(G,K), D(P) = D(P,K)$ etc. for the locally $L$-analytic distributions with values in $K.$

\vskip10pt

{\it Acknowledgments.} We thank O. Gabber, G. Henniart, V. Mazorchuk, T. Schmidt and B. Schraen for some helpful remarks. Furthermore, we thank the referees for their comments which helped us to improve the paper in a number of places.

\section{Preliminaries on locally analytic representations and $\frg$-modules}

\setcounter{enumi}{0}

\begin{para}\label{locanrep} {\it Locally analytic representations.}  We start with recalling some basic facts on locally analytic representations as introduced by Schneider and Teitelbaum \cite{ST1}.

For a locally $L$-analytic group $H$, let $C^{an}(H,K)$ be the locally convex vector space of locally $L$-analytic $K$-valued functions. More generally, if $V$ is a Hausdorff locally convex $K$-vector space, let $C^{an}(H,V)$ be the $K$-vector space consisting of locally analytic functions with values in $V$. It has the structure of a Hausdorff locally convex vector space, as well. The dual space $D(H) = D(H,K) = C^{an}(H,K)'$ is a $K$-algebra, and the multiplication (convolution) is separately continuous. It has the structure of a Fr\'echet-Stein algebra when $H$ is compact \cite{ST2}. Here the (convolution) product of $\delta_1,\delta_2\in D(H)$ is defined by

\vspace{-0.3cm}
\begin{numequation}\label{convolution}
\delta_1 \cdot \delta_2(f) =
(\delta_1 \otimes \, \delta_2)((h_1,h_2) \mapsto f(h_1h_2))
\,.
\end{numequation}

Let $V$ be a Hausdorff locally convex $K$-vector space. A  homomorphism $\rho: H \ra \GL_K(V)$
(or simply  $V$) is called a {\it locally analytic representation} of $H$
if the topological $K$-vector space $V$ is barrelled, the action of $H$ on $V$ is by continuous automorphisms, and the orbit maps $\rho_v: H \rightarrow V, \; h \mapsto \rho(h)(v)$, are
elements in $C^{an}(H,V)$ for all $v \in V$. We recall that a Hausdorff locally convex $K$-vector space $V$ is called of {\it compact type}, if it is an inductive limit of countably many Banach spaces with injective and compact transition maps, cf. \cite[sec. 1]{ST1}. In this case, the strong dual $V'_b$ is a nuclear Fr\'echet space (by \cite[16.10, 19.9]{S1}) and a separately continuous left $D(H)$-module. By \cite[3.3]{ST1}, the duality functor gives an anti-equivalence of categories

\begin{numequation}\begin{array}{ccc}\label{equivalence}

\left\{\begin{array}{c}
\mbox{locally analytic $H$-represen-} \\
\mbox{tations on $K$-vector spaces } \\
\mbox{of compact type with } \\
\mbox{continuous linear $H$-maps }
\end{array} \right\}

& \stackrel{\sim}{\longrightarrow} &

\left\{\begin{array}{c}
\mbox{separately continuous $D(H)$-}  \\
\mbox{modules on nuclear Fr\'echet} \\
\mbox{spaces with continuous  } \\
\mbox{$D(H)$-module maps}
\end{array} \right\}.

\end{array}
\end{numequation}

\vskip8pt

\noindent In particular, $V$ is topologically irreducible if $V'_b$ is a simple $D(H,K)$-module. We also recall that a locally analytic $H$-representation $V$ is called {\it strongly admissible} if $V$ is of compact type and if its strong dual $V'_b$ is a finitely generated $D(H_0)$-module for any (equivalently, one) compact open subgroup $H_0 \sub H$, cf. \cite[sec. 3]{ST1}.

\vskip8pt

\end{para}

\begin{para}\label{indlocanrep} {\it Induced locally analytic representations.} For any closed subgroup $H'$ of $H$ and any locally analytic
representation $V$ of $H'$, we denote by $\Ind^H_{H'}(V)$ the
induced locally analytic representation. It is defined by
$$\Ind^H_{H'}(V) := \Big\{f \in C^{an}(H,V) \midc \forall h' \in H', \forall h \in H: f(h \cdot
h') = (h')^{-1} \cdot f(h) \; \Big\} \;.$$


\noindent The group $H$ acts on this vector space by $(h \cdot f)(x) = f(h^{-1}x)$. We will frequently consider the strong dual space of $\Ind^H_{H'}(V)$. In this paper, $V$ will always be of compact type and $H/H'$ will be compact. As recalled above, cf. \ref{locanrep}, $V'_b$ is then a nuclear Fr\'echet space.

\end{para}

\setcounter{enumi}{0}

\begin{lemma}\label{finitediml} Assume that $H$ and $H'$ are compact $p$-adic Lie groups. If $V$ is a finite-dimensional locally analytic $H'$-representation, then $(\Ind^H_{H'} V)'_b$ is canonically isomorphic, as  $D(H)$-module, to $D(H) \otimes_{D(H')} V'$.
\end{lemma}

\noindent \Pf Consider the canonical map of $D(H)$-modules

\vspace{-0.3cm}
\begin{numequation}\label{dual_iso}
D(H) \otimes_{D(H')} V' \ra (\Ind^H_{H'}V)'_b \,, \,\,
\delta \otimes \vphi \mapsto \delta \cdot \vphi \,,
\end{numequation}

\vspace{-0.3cm}
\noindent where $(\delta \cdot \vphi)(f) = \delta(g \mapsto \vphi(f(g^{-1})))$. Because $H'$ is a compact $p$-adic Lie group, it is topologically finitely generated (cf. \cite[8.34]{DDMS}). Because $V$ is finite-dimensional the same arguments as in \cite{ST5}, before Lemma 6.1, show that \ref{dual_iso} is an isomorphism of topological vector spaces, if we give the left side the quotient topology of the projective tensor product topology on $D(H,K) \otimes_K V'$. (We remark that the projective and inductive tensor product topologies coincide for tensor products of Fr\'echet spaces, cf. \cite[17.6]{S1}.)  \qed

We recall that in this paper $G = \bG(L)$ denotes the group of $L$-valued points of a split reductive algebraic group over $L$. We use the notation as fixed at the end of the introduction. In particular, $G_0 = \bG_0(O_L)$ is the subgroup of $O_L$-valued points of a smooth model of $\bG$ over $O_L$.

Furthermore, every rational $\bG$-representation and every smooth $\bG(L)$-representation may be considered as a locally analytic representation, cf. \cite[\S 2]{ST4}. In the latter case the underlying vector space is the locally convex inductive limit of its finite-dimensional subspaces, and it is of compact type, if we assume that the smooth representation is admissible (because it is then of countable dimension).

\vskip8pt

\setcounter{enumi}{0}

\begin{lemma}\label{tensorprod} Let $P \sub G$ be a parabolic subgroup and put $P_0 = P \cap G_0$. Let $V$ be a smooth representation of the Levi subgroup $L_P$, and let $E$ be a finite-dimensional algebraic $P$-representation. Both $V$ and $E$ are assumed to be $K$-vector spaces, and $E'$ denotes the contragredient representation. 

(i) If $V$ is an admissible smooth representation, then $\Ind^G_P(E' \otimes_K V)$ is an admissible locally analytic representation\footnote{in the sense of \cite[sec. 6]{ST2}}.

(ii) If $V$ is of finite length, then $\Ind^G_P(E' \otimes_K V)$ is strongly admissible.

\end{lemma}

\noindent \Pf (i) Assume first that the $P$-representation $E$ comes by inflation from an $L_P$-representation (i.e., the unipotent radical $U_P$ of $P$ acts trivially on $E$). If $V$ is a smooth admissible representation, then it is admissible as a locally analytic representation by \cite[6.6]{ST2}. This implies that $E' \otimes_K V$ is an admissible $L_P$-representation by \cite[6.1.5]{Em1}, and by \cite[2.1.2]{Em2} we can conclude that $\Ind^G_P(E' \otimes_K V)$ is an admissible locally analytic representation. To treat the general case 
(when $U_P$ does not necessarily act trivially on $E$), we note first that we can find a descending filtration $E = E_0 \supset E_1 \supset \ldots \supset E_r = 0$ by $P$-stable subrepresentations $E_i$ such that $U_P$ acts trivially on the associated graded vector space $\bar{E} = \bigoplus_{i=0}^{r-1} E_i/E_{i+1}$. We have already proved the assertion in the case $r=1$. Arguing inductively, it suffices to treat the case $r=2$. 
Consider the extension $0 \ra E_1 \ra E \ra E/E_1 \ra 0$, which gives $0 \ra (E/E_1)' \ra E' \ra E_1' \ra 0$ by passing to dual spaces. 
As these are finite-dimensional, this sequence splits as topological vector spaces, and the same is true for the sequence that we get by 
tensoring with $V$:
$$0 \lra (E/E_1)' \otimes_K V \lra E' \otimes_K V \lra E_1' \otimes_K V \lra 0 \;.$$

\noindent By \cite[5.1, 5.4]{K2} this gives an exact sequence of locally analytic representations
$$0 \lra  \Ind^G_P((E/E_1)' \otimes_K V) \lra  \Ind^G_P(E' \otimes_K V) \lra \Ind^G_P(E_1' \otimes_K V)  \ra  0 \;.$$

\noindent As we pointed out above, the representations on the left and on the right are admissible. Then the representation in the middle is admissible as well, because the category of admissible representations is closed under extensions, cf. \cite[p. 304]{ST5}, \cite[Remark 3.2]{ST2}.

(ii) Any admissible representation that is an extension of two strongly admissible representations is strongly admissible as well. 
Therefore, by the same d\'evissage as above, we may assume that the action of $P$ on $E$ comes by inflation from an action of $L_P$. 
Because $V$ is now assumed to be of finite length, the dual space $V'$ is a finitely generated $D(P_0)$-module, by \cite[2.2]{ST4}. By \cite[3.3]{ST4}, this implies that also $E \otimes_K V'$ is a finitely generated $D(P_0)$-module. Therefore, $E' \otimes_K V$ is a strongly admissible $L_P$-representation. By \cite[2.1.2]{Em2}, the induced representation $\Ind^G_P(E' \otimes_K V)$ is strongly admissible.
\qed


\setcounter{enumi}{0}

\begin{para} {\it The category $\cO$ and its parabolic variants $\cO^\frp$.}
Let $\bG$ be a connected split reductive group over $L$. We use the notation as introduced at the end of the introduction. In particular, we take care to distinguish between vector spaces (and schemes) over our ''base field'' $L$ and our ''coefficient field'' $K$. The subscript ${}_K$ denotes base change to $K$. However, when dealing with universal enveloping algebras, we make the general convention to write $U(\frg)$, $U(\frb)$ etc. to denote the corresponding universal enveloping algebras {\it after base change to $K$}, i.e., what is precisely $U(\frg_K)$, $U(\frb_K)$ and so on.


The category $\cO$ in the sense of Bernstein, Gelfand, Gelfand, cf. \cite{BGG}, \cite{H1}, is defined for complex semi-simple Lie algebras. Here we consider the following variant for split reductive Lie algebras over a field of characteristic zero. Thus we let $\cO$ be the full subcategory of all $U(\frg)$-modules $M$ which satisfy the following properties:

\begin{enumerate}
\item $M$ is finitely generated as a $U(\frg)$-module. \smallskip
\item $M$ decomposes as a direct sum of one-dimensional $\frt_K$-representations. \smallskip
\item The action of $\frb_K$ on $M$ is locally finite, i.e. for every $m \in M$, the subspace $U(\frb)\cdot m \subset M$ is finite-dimensional over $K$.
\end{enumerate}

\noindent As in the classical case one shows that $\cO$ is a $K$-linear, abelian, noetherian, artinian category which is closed under submodules and quotients, cf. \cite[1.1, 1.11]{H1}. In particular, every object of $\cO$ has a Jordan-H\"older series.

In our paper we are mainly interested in a certain subcategory of $\cO$. The reason will become clear later on. Note that by property (2), we may write any object $M$ in $\cO$ as a direct sum
\begin{equation}\label{eigenspace}
M=\bigoplus_{\lambda \in \frt^\ast_K} M_\lambda
\end{equation}
where $M_\lambda=\{m\in M \mid \forall \frx \in \frt_K: \frx \cdot m = \lambda(\frx) m  \}$ is the $\lambda$-eigenspace attached to $\lambda \in \frt^\ast_K =  \Hom_K(\frt_K,K)$. Let $X^\ast(\bT) = \Hom(\bT,\bbG_m)$ be the group of characters of the torus $\bT$ which we consider via the derivative as a subgroup of $\frt^\ast_K$.
\end{para}

\begin{dfn} We denote by $\cO_\alg$ the full subcategory of $\cO$
whose objects are $U(\frg)$-modules such that all $\lambda$ appearing in \ref{eigenspace}, for which $M_\lambda \neq 0$, are contained in $X^\ast({\bf T})\subset \frt^\ast_K$.
\end{dfn}

Thus $M\in \cO$ is an object of $\cO_\alg$ if the $\frt_K$-module structure on every $M_\lambda$ lifts to an algebraic action of $\bT$.
Again, $\Oa$ is an abelian noetherian, artinian category which is closed under submodules and quotients.
The Jordan-H\"older series of a given $U(\frg)$-module lying in $\Oa$ is the same as the one considered in the category $\cO.$

\begin{exam}
For $\lambda\in \frt^\ast_K$, let $K_\lambda=K$ be the 1-dimensional $\frt_K$-module where the action is given by $\lambda$. Then 
$K_\lambda$ extends uniquely to a $\frb_K$-module.  Let
$$M(\lambda)=U(\frg) \otimes_{U(\frb)} K_\lambda \in \cO$$
\noindent be the corresponding Verma module. Denote by $L(\lambda)\in \cO$  its simple quotient.
Then $M(\lambda)$ resp. $L(\lambda)$ is an object of $\Oa$ if and only if  $\lambda \in X^\ast({\bf T}).$
\end{exam}

Let ${\bf P}$ be a standard parabolic subgroup of ${\bf G}$ with Levi decomposition ${\bf P=L_P \cdot U_P}$ where ${\bf T\subset L_P}$. We put $\fru_P = \Lie(\bU_\bP)$ and $\frl_P = \Lie(\bL_\bP)$. We define $\cO^\frp$ to be the category of $U(\frg)$-modules $M$ satisfying the following properties:

\begin{enumerate}
\item $M$ is finitely generated as a $U(\frg)$-module. \smallskip
\item Viewed as a $\frl_{P,K}$-module, $M$ is the direct sum of finite-dimensional simple modules. \smallskip
\item The action of $\fru_{P,K}$ on $M$ is locally finite.
\end{enumerate}

This is analogous to the definition over an algebraically closed field, cf. \cite[ch. 9]{H1}. Clearly, the category $\cO^\frp$ is a full subcategory of $\cO$. Furthermore, it is $K$-linear, abelian and closed under submodules and quotients, cf. \cite[9.3]{H1}. Hence the Jordan-H\"older series of every $U(\frg)$-module in $\cO^\frp\subset \cO$ lies in $\cO^\frp$ as well. If $Q$ is a standard parabolic subgroup with $Q\supset P$, then  $\cO^\frq \subset \cO^\frp.$ Finally, consider the extreme case $\frp=\frg$: the category $\cO^\frg$ consists of all finite-dimensional semi-simple $\frg_K$-modules. On the other hand, $\cO^\frb = \cO$.


Similarly as before we define a subcategory $\cO^\frp_\alg$ of $\cO^\frp$ as follows. Let ${\rm Irr}(\frl_{P,K})^{\rm fd}$ be the set of isomorphism classes of finite-dimensional irreducible $\frl_{P,K}$-modules.
Again, any object in $\cO^\frp$ has by property (2) a decomposition into $\frl_{P,K}$-modules
\begin{numequation}\label{isotypical_decom}
M= \bigoplus_{\fra \in {\rm Irr}(\frl_{P,K})^{\rm fd}} M_{\fra}
\end{numequation}
\noindent where $M_{\fra} \sub M$ is the $\fra$-isotypic part of the representation $\fra$. We denote by $\cO^\frp_\alg$ the full subcategory of $\cO^\frp$
consisting of objects $M$ of $\cO^\frp$ with the following property: if $M_\fra \neq 0$ (with the notation as in \ref{isotypical_decom}), then $\fra$ is the Lie algebra representation induced by a finite-dimensional algebraic $\bL_{\bP,K}$-representation, where $\bL_{\bP,K} = \bL_\bP \times_{\Spec(L)} \Spec(K)$.
Again, the category $\cO^\frg_\alg$ is contained in $\Oa$ and contains all finite-dimensional $\frg_K$-modules which are induced by algebraic ${\bf G}$-modules.
Every object in $\cO^\frp_\alg$ has a Jordan-H\"older series which coincides with the Jordan-H\"older series in $\Oa$. The following lemma will later play a crucial role.

\begin{lemma}\label{T-alg implies L-alg} Let $M$ be an object of $\cO^\frp$. Then $M$ is in $\cO^\frp_\alg$ if and only it is in $\cO_\alg$.
\end{lemma}

\noindent \Pf This follows from the discussion in section 1.20 of part II in \cite{Ja}. It is shown there that a finite-dimensional $U(\frl_P)$-module which is equipped with an algebraic $\bT$-action, whose induced $U(\frt)$-structure coincides with the one coming from the $U(\frl_P)$-structure, lifts (uniquely) to an algebraic $\bL_{\bP,K}$-representation. Note that the algebra ${\rm Dist}(\bL_{\bP,K})$ used in loc.cit. is in characteristic zero canonically isomorphic to the enveloping algebra $U(\frl_{P,K})$, cf. \cite[part I, 7.10]{Ja}. Apart from the condition that the induced $U(\frt)$-structure coincides with the one coming from the $U(\frl_P)$-structure, Jantzen requires the following compatibility between the $\bT$-action and $U(\frl_P)$-action:
$${\rm Ad}(t)(\frx) \cdot m  = t \cdot (\frx \cdot (t^{-1} \cdot m)) \;,$$


\noindent for all $m \in M$, $t \in \bT$ and $\frx \in \frl_{P,K}$. In our situation this is automatic: we may assume $m \in M_\lambda$ and $\frx \in \frg_\alpha$. This implies that $\frx \cdot m \in M_{\lambda+\alpha}$. Therefore ${\rm Ad}(t)(\frx) \cdot m = \alpha(t) (\frx \cdot m)$, and $t \cdot (\frx \cdot (t^{-1} \cdot m)) = (\alpha + \lambda)(t) \lambda(t^{-1}) \cdot (\frx \cdot m) = \alpha(t) (\frx \cdot m)$. \qed

\begin{rmk} The preceding lemma can also be proved by using the description of irreducible algebraic representations of $\bL_\bP$ in terms of characters of $\bT$.
\end{rmk}

\setcounter{enumi}{0}

\begin{exam}\label{Example_Verma}
Let $\Delta$ be the set of simple roots of ${\bf G}$ with respect to ${\bf T\subset B}$.
Let $\lambda \in X({\bf T})^\ast$ and set $I = \{\alpha \in \Delta \midc \langle \lambda, \alpha^\vee \rangle \in \bbZ_{\ge 0} \}.$
Let ${\bf P}={\bf P}_I$ be the standard parabolic subgroup of ${\bf G}$ attached to $I.$  Then $\lambda$ is dominant with respect to the Levi subgroup ${\bf L_P}.$ Denote by   $V_I(\lambda)$  the corresponding irreducible finite-dimensional algebraic  ${\bf L_P}$-representation, cf. \cite[Pt. 2, 2.13]{Ja}.
We consider it as a ${\bf P}$-module by letting act ${\bf U_P}$ trivially on it.
The generalized parabolic Verma module (in the sense of Lepowsky \cite{Le}) attached to the weight $\lambda$  is given by
$$M_I(\lambda)=U(\frg) \otimes_{U(\frp_I)} V_I(\lambda).$$
Then $M_I(\lambda)$ is an object of $\cO^\frp_\alg.$
Further, there is a surjective map
$$M(\lambda) \rightarrow M_I(\lambda),$$
where the kernel is given by the image of $\oplus_{\alpha \in I} M(s_\alpha\cdot \lambda) \rightarrow M(\lambda)$.
It follows that $L(\lambda)$ is an object of $\cO^\frp_\alg$, as well, cf. \cite[sec. 9.4]{H1}.
\end{exam}


\section{From $\cO_\alg$ to locally analytic representations}

\setcounter{enumi}{0}

\begin{para} {\it The representation associated to an object of $\cO_\alg$.} We fix as in the previous chapter a standard parabolic subgroup ${\bf P}$ with Levi decomposition ${\bf P=L_P \cdot U_P}$ where ${\bf T \subset L_P}$. Let $M$ be an object of $\cO^\frp_\alg$.
By the defining properties (1)-(3) for $\cO^\frp$, we may choose a finite-dimensional representation $W \sub M$ of $\frp_K$ which generates $M$ as $U(\frg)$-module. Thus we have a short exact sequence of $U(\frg)$-modules

\vspace{-0.3cm}
\begin{numequation}\label{basicsequence}
0 \ra \frd \ra U(\frg) \otimes_{U(\frp)} W \ra M \ra 0 \,,
\end{numequation}

\vspace{-0.3cm}
\noindent where, by definition, $\frd$ is the kernel of the canonical map $U(\frg) \otimes_{U(\frp)} W \ra M$. We denote the representation of $\frp_K$ on $W$ by $\rho$. It is induced by an algebraic $\bP_K$-representation by the following lemma.
\end{para}


\setcounter{enumi}{0}

\begin{lemma}\label{lemma_lift_alg} In the setting above, the representation $\rho$ lifts uniquely to an algebraic $\bP_K$-representation on $W$ (which we denote again by $\rho$).
\end{lemma}


\noindent \Pf  As $M$ is an object of $\cO^\frp$, the $U(\frp)$-module $W$, considered as a $U(\frl_\frp)$-module, decomposes into a 
direct sum of isotypic modules $W_\fra$. Each module $W_\fra$ lifts uniquely to an algebraic representation of 
$\bL_{\bP,K}$ by \ref{T-alg implies L-alg}. We denote this action of $\bL_{\bP,K}$ on $W = \bigoplus_\fra W_\fra$ by $\rho_\bL$. 
Moreover, the action of the Lie algebra $\fru_{\frp,K}$ integrates uniquely to give an algebraic action of $\bU_\bP$ on $W$ as follows. 
Given an element $u = \exp(\frx)\in {\bf U_P}(\overline{K})$, where $\overline{K}$ denotes an algebraic closure of $K$, we define $\rho(u) := \sum_{n \ge 0} \frac{\rho(\frx)^n}{n!}$, where $\rho(\frx)^n = 0$ for $n \gg 0$. These two representations are compatible in the sense that $\rho_\bL(h) \circ \rho(u) \circ \rho_\bL(h^{-1}) = \rho(\Ad(h)(u))$, for $h \in \bL_\bP(\overline{K})$, $u \in \bU_\bP(\overline{K})$. This shows that $\rho$ lifts uniquely to an algebraic representation of $\bP_K$ on $W$. \qed

\vskip8pt

The induced locally analytic representation of $P$ on the dual space $W' = \Hom_K(W,K)$ will be denoted by $\rho'$. We consider the locally analytic induced representation $\Ind_P^G(W')$.
By \ref{tensorprod} the canonical map of $D(G)$-modules \vspace{-0.3cm}

$$D(G) \otimes_{D(P)} W  \lra \left(\Ind^G_P(W')\right)'\,, \hskip10pt \delta \otimes w \mapsto [f \mapsto \delta(f)(w)] \,,$$

\noindent is an isomorphism of topological vector spaces. We thus have a pairing

\vspace{-0.3cm}
\begin{numequation}\label{pair_K}
\langle \cdot \,, \cdot \rangle: \left(D(G) \otimes_{D(P)} W\right)  \otimes_K \Ind^G_P(W') \lra K \,,
\end{numequation}

\vspace{-0.3cm}
\noindent which identifies the left hand side with the topological dual of the right hand side and vice versa. We remark that
the Iwasawa decomposition $G=G_0\cdot B$ shows that as $D(G_0)$-modules, the restriction of $D(G) \otimes_{D(P)} W$ to $D(G_0)$ is isomorphic to $D(G_0) \otimes_{D(P_0)} W$. The pairing \ref{pair_K} is obtained from the following $C^{an}(G,K)$-valued pairing by composition with the evaluation map $C^{an}(G,K) \ra K$, $f \mapsto f(1)$.

\begin{numequation}\label{pair_C(G,K)}
\begin{array}{rccc}
\langle \cdot , \cdot \rangle_{C^{an}(G,K)}: & \left(D(G) \otimes_{D(P)} W \right) \otimes_K \Ind^G_P(W') & \lra & C^{an}(G,K) \\
&&&\\
& (\delta \otimes w) \otimes f & \mapsto & \Big[ g \mapsto \big(\delta \cdot_r (f(\cdot)(w))\big)(g)\Big]
\end{array}
\end{numequation}

\noindent where, by definition, we have $\big(\delta \cdot_r (f(\cdot)(w))\big)(g) = \delta(x \mapsto f(gx)(w))$. The same arguments as in \cite[3.4.1]{OS} show 
that the obvious map

\vspace{-0.3cm}
$$U(\frg) \otimes_{U(\frp)} W \ra D(G) \otimes_{D(P)} W$$


\noindent is injective, so that we consider $U(\frg) \otimes_{U(\frp)} W$ as being contained in the dual space of $\Ind^G_P(W')$. We then denote by $\Ind^G_P(W')^\frd$ the subspace of $\Ind^G_P(W')$ annihilated by $\frd$ via the pairing $\langle \cdot , \cdot \rangle_{C^{an}(G,K)}$:

\vspace{-0.3cm}
\begin{numequation}\label{basicdfn}
\Ind^G_P(W')^\frd = \{ f \in \Ind^G_P(W') \midc \mbox{ for all } \delta \in \frd : \langle \delta, f \rangle_{C^{an}(G,K)} = 0_{C^{an}(G,K)} \} \,.
\end{numequation}

\vspace{-0.3cm}
\noindent (This is exactly as in \cite{ST3}; cf. the definition of the representation before Thm. 8.6. in \cite{ST3}.)
Here $\Ind^G_P(W')^\frd$ is clearly a closed subspace, since the action of $U(\frg)$ is continuous on $\Ind^G_P(W')$. It is $G$-invariant because the pairing $\langle \cdot , \cdot \rangle_{C^{an}(G,K)}$ uses the right translation of $G$ on itself (in the argument of $f$), but the $G$-action on the induced representation is given by left translation.  

\vskip8pt

\setcounter{enumi}{0}

\begin{prop}\label{Prop_admissible}
(i) The representation $\Ind^G_P(W')^\frd$ is a strongly admissible locally analytic $G$-represen-\linebreak tation.

\noindent (ii) The annihilator of $\Ind^G_P(W')^\frd$ in $D(G) \otimes_{D(P)} W$, i.e., the set
$$\{\psi \in D(G) \otimes_{D(P)} W \midc \mbox{ for all } f \in \Ind^G_P(W')^\frd: \langle \psi, f \rangle = 0 \} \,.$$

\noindent is equal to $D(G)\frd$. We therefore have a canonical isomorphism of coadmissible $D(G)$-modules

$$\left( \Ind^G_P(W')^\frd \right)' \cong \left(D(G) \otimes_{D(P)} W \right)/ \, D(G)\frd \,\,.$$

\end{prop}

\vskip6pt

\noindent \Pf (i) The representation $\Ind^G_P(W')$ is strongly admissible by \ref{tensorprod}. By \cite[3.5]{ST1} closed invariant subspaces of strongly admissible representations are strongly admissible again.

(ii) By  \cite[6.3]{ST2}, the admissible subrepresentations of $\Ind^G_P(W')$ are in bijection with the coadmissible quotients of $D(G) \otimes_{D(P)} W$. By \cite[3.6]{ST2} these are given by the closed submodules of $D(G) \otimes_{D(P)} W$. The bijection is given by associating to a closed submodule $J\subset D(G) \otimes_{D(P)} W$ the representation

$$\{f \in \Ind^G_P(W') \midc \mbox{ for all } \psi \in J: \langle \psi, f \rangle = 0 \} \,.$$


\noindent By definition, for $f \in \Ind^G_P(W')$ to lie in $\Ind^G_P(W')^\frd$ is equivalent to satisfy

$$\sum_{i=1}^m \Big((\fry_i \cdot_r f)(g)\Big)(w_i) = 0$$


\noindent for all $g \in G$ and all $\sum_i \fry_i \otimes w_i \in \frd$. By the definition of $\cdot_r$ and the definition of the convolution product \ref{convolution} in $D(G)$  this is equivalent to

$$\sum_{i=1}^m \Big((\delta_g \cdot \fry_i)(f)\Big)(w_i) = 0$$


\noindent for all $g \in G$ and all $\sum_i \fry_i \otimes w_i \in \frd$. Hence $f$ is in $\Ind^G_P(W')^\frd$ if and only if $\langle \delta_g \cdot \frz, f \rangle = 0$, for all $g \in G$ and all $\frz \in \frd$, where $\langle \cdot , \cdot \rangle$ is the $K$-valued pairing in \ref{pair_K}. Because the subspace generated by the Dirac distributions $\delta_g$ is dense in $D(G)$, this is equivalent to saying that $\langle \delta ,f \rangle = 0$,
for all $\delta \in D(G) \frd$. As $\frd$ is again an object of $\cO^\frp_\alg$, it is finitely generated as $U(\frg)$-module. Hence $D(G)\frd$ is finitely generated as $D(G)$-module and therefore closed by \cite[Cor. 3.4, Lemma 3.6]{ST2}. The assertion follows. \qed


\begin{para}\label{fund_constr}{\it Another description.}  We proceed by giving another description of the dual space of $\Ind_P^G(W')^\frd$. Let $M$ be an object of $\cO^\frp_\alg$ as above. Then $M$ is the union of finite-dimensional $\frp_K$-modules. Denote by $X$ one of these finite-dimensional submodules. $X$ lifts by Lemma \ref{lemma_lift_alg} uniquely to an algebraic $\bP_K$-representation. In particular, we can consider $X$ as a locally analytic $P$-representation, and as such it has a unique structure as a $D(P)$-module, where the Dirac distributions act as group elements, cf. \cite[Prop. 3.2]{ST1}, and the sentence before Lemma 3.1 in \cite{ST1}. (Note that we do not consider here the $D(P)$-module structure on the dual space $X'$.) It thus follows that there is a unique $D(P)$-module structure on $M$ which extends the $\frp_K$-module structure and such that the Dirac distributions $\delta_g$ act as group elements $g \in P$.


The next step is to consider the subring $D(\frg,P)$ generated by $U(\frg)$ and $D(P)$ inside $D(G)$. It follows from the proposition below that this ring is equal to $U(\frg)D(P)$, i.e., every element of $D(\frg,P)$ can be written as a finite sum of elements of the form $\frz \cdot \delta$ with $\frz \in U(\frg)$ and $\delta \in D(P)$.
\end{para}

\begin{prop}\label{commute} Let $H \sub G$ be a closed analytic subgroup, and let $\delta \in D(H)$. Then $\delta \cdot U(\frg) \subset U(\frg) \cdot D(H)$. In particular, the smallest subring of $D(G)$ containing $U(\frg)$ and $D(H)$ consists of finite sums $\sum_j \frz_j \cdot \delta_j$ with $\frz_j \in U(\frg)$ and $\delta_j \in D(H)$.
\end{prop}

\noindent \Pf Of course, it is enough to show that for any $\frx \in \frg$ we have $\delta \cdot \frx \in \frg \cdot D(H)$. Let $f \in C^{an}(G,K)$. For $g \in G$ we denote by $g.f$ the function defined by $(g.f)(x) = f(g^{-1}x)$.
It is easy to see that the convolution product \ref{convolution} of two distributions $\lambda_1, \lambda_2 \in D(G)$ satisfies

$$(\lambda_1 \cdot \lambda_2)(f) = \lambda_1(g \mapsto \lambda_2(g^{-1}.f)) = \lambda_1(g \mapsto \lambda_2(h \mapsto f(gh)) = \lambda_2(h \mapsto \lambda_1(g \mapsto f(gh))).$$

\vskip6pt

\noindent The last equality may be called ''Fubini's Theorem'', and this is used below in one instance. Furthermore, the image of $\frx$ in $D(G)$ is given by the formula

$$\frx(f) = \lim_{t \ra 0} \frac{1}{t}(f(\exp(t\frx)) - f(1)).$$

\vskip6pt

\noindent For $h \in H$ and $\frx$ as above write $\mbox{Ad}(h)(\frx) = \sum_{i = 1}^dc_i(h)\frx_i$, where $(\frx_i)_i$ is some basis for $\frg$, and the $c_i$ are locally analytic functions on $H$. Define distributions $\delta_i \in D(H)$ by $\delta_i(f) = \delta(h \mapsto c_i(h)f(h))$. Then we compute:
$$\begin{array}{rcl}
(\delta \cdot \frx)(f) & = & \delta(h \mapsto \frx(h^{-1}.f)) \hskip4pt = \hskip4pt \delta(h \mapsto \mbox{Ad}(h)(\frx)(g \mapsto f(gh)))\\
&&\\
& = & \delta\left(h \mapsto (\sum_{i = 1}^dc_i(h)\frx_i)(g \mapsto f(gh))\right) \hskip4 pt = \hskip4pt  \sum_{i = 1}^d \delta\Big(h \mapsto (c_i(h)\frx_i)(g \mapsto f(gh))\Big)\\
&&\\
& = & \sum_{i = 1}^d \delta(h \mapsto \frx_i(g \mapsto c_i(h)f(gh)))  \hskip5pt \stackrel{\mbox{\tiny{Fubini}}}{=} \hskip5pt \sum_{i = 1}^d \frx_i(g \mapsto \delta(h \mapsto c_i(h)f(gh)))\\
&&\\
& = & \sum_{i = 1}^d \frx_i(g \mapsto \delta(h \mapsto c_i(h)(g^{-1}.f)(h))) \hskip4 pt = \hskip4pt \sum_{i = 1}^d \frx_i(g \mapsto \delta_i(g^{-1}.f))\\
&&\\
& = & \sum_{i = 1}^d (\frx_i \cdot \delta_i)(f)\\
\end{array}$$


And this shows that $\delta \cdot \frx = \sum_{i = 1}^d \frx_i \cdot \delta_i$ is in $\frg \cdot D(H)$. \qed

\vskip8pt

\begin{cor}\label{extended structure} There is on any object $M$ of $\cO^\frp_\alg$ a unique $D(\frg,P)$-module structure with the following properties: 

(i) The action of $U(\frp)$, as a subring of $U(\frg)$, coincides with the action of $U(\frp)$ as a subring of $D(P)$. \vskip8pt

(ii) The Dirac distributions $\delta_g \in D(P)$ act like group elements $g \in P$ (the latter action being given by Lemma \ref{lemma_lift_alg}). 

Moreover, any morphism $M_1 \ra M_2$ in $\cO^\frp_\alg$ is automatically a homomorphism of $D(\frg,P)$-modules.
\end{cor}

\noindent \Pf The first assertion about the existence and uniqueness of the $D(\frg,P)$-module structure follows from the discussion in the beginning of section \ref{fund_constr} (which in turn uses \cite[3.1, 3.2]{ST1}), together with the Prop. \ref{commute} just proved. The assertion about homomorphisms $M_1 \ra M_2$ follows from the uniqueness of the $D(\frg,P)$-module structure. \qed

\vskip8pt

In the following we consider the tensor product $D(G) \otimes_{D(\frg,P)} M$ for objects $M$ of $\cO^\frp_\alg$. Because $M$ is a finitely generated $U(\frg)$-module, hence a finitely generated $D(\frg,P)$-module, the tensor product $D(G) \otimes_{D(\frg,P)} M$ is a
finitely generated $D(G)$-module. Let $D(\frg,P_0)$ be the subring of $D(G_0)$ generated by $U(\frg)$ and $D(P_0)$. By \ref{commute} we have $D(\frg,P_0) = U(\frg) \cdot D(P_0)$ and $D(G_0) \otimes_{D(\frg,P_0)} M$ is likewise a finitely generated $D(G_0)$-module. Furthermore,
it follows from the Iwasawa decomposition $G = G_0P$ that the canonical map

\vspace{-0.3cm}
$$D(G_0) \otimes_{D(\frg,P_0)} M \lra D(G) \otimes_{D(\frg,P)} M$$


\noindent is an isomorphism of $D(G_0)$-modules, cf. \cite[6.1 (i)]{ST5}. The next proposition shows that $D(G_0) \otimes_{D(\frg,P_0)} M$ is a co-admissible $D(G_0)$-module.

\setcounter{enumi}{0}

\begin{prop}\label{basiciso} The canonical map

\vspace{-0.3cm}
$$\iota: M  = \left(U(\frg) \otimes_{U(\frp)} W\right)/\frd  \lra  \left(D(G) \otimes_{D(P)} W \right)/ \, D(G)\frd$$

\vskip6pt

\noindent extends to an isomorphism of $D(G)$-modules 

$$D(G) \otimes_{D(\frg,P)} M \cong \left(D(G) \otimes_{D(P)} W \right)/ \, D(G)\frd
\,\,,$$

\vskip6pt

\noindent Together with \ref{Prop_admissible} we hence get an isomorphism of $D(G)$-modules

$$D(G) \otimes_{D(\frg,P)} M  \cong \left( \Ind^G_P(W')^\frd \right)' \,\,.$$

\vskip6pt

\noindent As the canonical map $D(G_0) \otimes_{D(\frg,P_0)} M \ra D(G) \otimes_{D(\frg,P)} M$ is an isomorphism, we also have an isomorphism of $D(G_0)$-modules

$$D(G_0) \otimes_{D(\frg,P_0)} M \cong \left( \Ind^G_P(W')^\frd \right)' \;.$$

\noindent In particular, $D(G_0) \otimes_{D(\frg,P_0)} M$ is a coadmissible $D(G_0)$-module.

\end{prop}

\noindent \Pf First we have to show that $\iota$ is $D(\frg,P)$-linear. But this follows from the $D(\frg,P)$-linearity of the natural maps $q: U(\frg)\otimes_{U(\frp)} W \ra M$ and $U(\frg) \otimes_{U(\frp)} W \lra  D(G) \otimes_{D(P)} W$ together with the fact that $\frd$ is a $D(\frg,P)$-submodule of $U(\frg) \otimes_{U(\frp)} W.$
This shows that $\iota$ extends to a $D(G)$-module homomorphism

$$\Phi: D(G) \otimes_{D(\frg,P)} M \lra \left(D(G) \otimes_{D(P)} W \right)/ \, D(G)\frd$$

\vskip6pt

\noindent by setting $\Phi(\delta \otimes v) = \delta \iota(v)$. In the other direction we define

$$\Psi: \left(D(G) \otimes_{D(P)} W \right)/ \, D(G)\frd \lra D(G) \otimes_{D(\frg,P)} M$$

\vskip6pt

\noindent by $\Psi((\delta \otimes w) + D(G)\frd) = \delta \otimes \left(w + \frd \right)$. It is immediate that $\Psi$ is well-defined and easily checked that $\Phi$ and $\Psi$ are inverse to each other. The last statement is now an immediate consequence of \ref{Prop_admissible} (i). \qed

\vskip8pt

\setcounter{enumi}{0}

\begin{para}\label{yetanother}{\it Yet another description.}
In this paragraph we present another approach to the representation $\Ind^G_P(W')^\frd$ which appeared already in \cite{O} in the case of $G=\GL_n.$ Here we correct certain group actions which were formulated in loc.cit..

Let ${\bf G_0}$ be a split reductive group model of ${\bf G}$ over $O_K.$  Let ${\bf T_0}\subset {\bf G_0}$ be a maximal torus and fix a Borel subgroup  ${\bf B_0\subset G_0}$ containing ${\bf T_0}$.
We fix  a standard parabolic subgroup ${\bf P_{0}}$ of ${\bf G_0}$. We denote by ${\bf U_{P,0}}$ its unipotent radical, by ${\bf U^-_{P,0}}$ its opposite unipotent radical and by ${\bf L_{P,0}}$ its Levi component containing ${\bf T_0}$.

Let $\pi\in O_L$ be a uniformizer. For any positive integer $n \in \bN$, we consider the reduction map
\begin{numequation}\label{reduction_map}
p_n:{\bf G_0}(O_L)\rightarrow {\bf G_0}(O_L/{(\pi^n)})
\end{numequation}
Set
\begin{eqnarray*}
 P^n & = &  p_n^{-1}\big({\bf P_0}(O_L/{(\pi^n)})\big), \; U_P^{n} =    p_n^{-1}\big({\bf U_{P,0}}(O_L/{(\pi^n)})\big), \; L_P^{n}  =    
 p_n^{-1}\big({\bf L_{P,0}}(O_L/{(\pi^n)})\big)
\end{eqnarray*}
and
$$  U_P^{-,n}   =  {\rm ker}\,\big({\bf U_{P,0}^{-} }(O_L)\rightarrow {\bf U_{P,0}^{-}}(O_L/{(\pi^n)})\big)
, U_P^{+,n}   =  {\rm ker}\,\big({\bf U_{P,0}}(O_L)\rightarrow {\bf U_{P,0}}(O_L/{(\pi^n)})\big).$$
These are compact open subgroups of $G_0={\bf G}(O_L).$ The Levi decomposition on ${\bf P_0}(O_L/(\pi^n))$ induces a decomposition
\begin{numequation}\label{Levi}
P^n=L_P^n\cdot U_P^n.
\end{numequation}
Further, we have equalities
\begin{eqnarray}\label{commutativity}
\nonumber P^n & = & U_P^{-,n}\cdot P_{0}=P_0\cdot U_P^{-,n}, \smallskip\\
U_P^n & = & U_P^{-,n}\cdot U_{P,0}=U_{P,0}\cdot U_P^{-,n}, \smallskip \\
\nonumber L_P^n & = & U_P^{+,n} \cdot U_P^{-,n}\cdot L_{P,0}=L_{P,0}\cdot U_P^{+,n}\cdot U_P^{-,n} \;.
\end{eqnarray}

We may interpret $U_P^{-,n} \subset U_{P,0}^-$ as the $L$-valued points of  an open $L$-affinoid polydisc, since  all non-diagonal entries $x$ in $U_P^{-,n}$ have norm $|x| \leq |\pi^n|.$ Then the ring of $K$-valued rigid-analytic functions $\cO(U_P^{-,n}) $ is a $K$-Banach algebra equipped with the following  norm.
Let $$\Phi_{U_P^{-}}=\{\beta_1,\ldots,\beta_r\}$$ be the set of roots appearing   in $\fru_P^-.$
We consider  the ring $\cO({\bf U_P^-})$ of $K$-valued algebraic functions on ${\bf U_P^-}$  as  the polynomial $K$-algebra in the
indeterminates $X_{\beta_1},\ldots,X_{\beta_r}.$ For $n\in \bN$, set $\epsilon_n=|\pi|^n.$  Then

\begin{numequation}\label{norm_a}
 \Big|\sum\nolimits_{(i_1,\ldots,i_r) \in \bN^r_0} a_{i_1,\ldots,i_r} X^{i_1}_{\beta_1}\cdots X^{i_r}_{\beta_r} \Big|_n:= \sup\nolimits_{(i_1,\ldots,i_r) \in \bN^r_0} |a_{i_1,\ldots,i_r}|_K\epsilon_n^{i_1+\cdots + i_r}.
\end{numequation}

\noindent defines a norm on the $K$-algebra $\cO({\bf U_P^-})$ so that $\cO(U_P^{-,n})$ becomes the completion of it.
This  $K$-Banach space is contained in the larger ring of bounded functions on $U_P^{-,n}$

\vspace{-0.3cm}
\begin{eqnarray*}\cO_b(U_P^{-,n}):=\Big\{\sum\nolimits_{(i_1,\ldots,i_r)} a_{i_1,\ldots,i_r}\cdot  X^{i_1}_{\beta_1}\cdots X^{i_r}_{\beta_r} \mid a_{i_1,\ldots,i_r} \in K,\; \sup\nolimits_{(i_1,\ldots,i_r)}\!\! |a_{i_1,\ldots,i_r}|_K\epsilon_n^{i_1+\cdots + i_r} < \infty \Big\}
\end{eqnarray*}

\noindent which is a $K$-Banach space with the same norm, as well.


Let $W$ be a finite-dimensional locally analytic $P$-representation. For any $n\in \bN,$ set 
$$V_n=\{\mbox{rigid-analytic maps } f:P^n \to W' \mid f(gp)= p^{-1}\cdot f(g) \}.$$
This is a locally analytic $P^n$-representation where $P^n$ acts via left translation.
There is a natural identification
\begin{numequation}\label{rigid_analytic}
 V_n\cong \cO(U_P^{-,n}) \otimes_K W'.
\end{numequation}

The action of $P^n=U_P^{-,n}\cdot L_{P,0}\cdot U_{P,0}$ on the latter object translates as follows. The subgroup $L_{P,0}$ acts 
via conjugation on $U_P^{-n}$ on the first factor and on $W'$ by the given one. The subgroup $U_P^{-,n}$ acts by translations on the 
first factor and trivially on $W'$. We omit the description of the action of $U_{P,0}$ since we will not use it explicitly.

\begin{lemma}\label{Lemma_ind_proj}
 There is a canonical isomorphism
\begin{eqnarray*}
\Ind^{G_0}_{P_{0}}(W') & \cong & \varinjlim\nolimits_{n\in \bbN}  {\rm Ind}^{G_0}_{P^n}(V_n),
\end{eqnarray*}
where $\Ind^{G_0}_{P^n}$ denotes the ordinary  induction of group representations. 
\end{lemma}

\noindent \Pf Let $f\in  {\rm Ind}^{G_0}_{P^n}(V_n)$ and denote by $R_n\subset G_0$ a set of
representatives for $G_0 /P^n.$ We may consider $f$ as a function on $G_0= \bigcup_{g \in R} gP^n$
such that the restriction to each open subset $gP^n$ is rigid-analytic. Hence it gives rise to
an locally analytic function on $G_0$ which lies in $\Ind^{G_0}_{P_{0}}(W').$ This construction is compatible
with varying the integer $n.$

On the other hand, if $f\in \Ind^{G_0}_{P_{0}}(W')$, then there is a disjoint union $G_0=\bigcup_i U_i$ by $P_0$-stable subsets (with respect to
the action from the right)  such that $f$ restricted to $U_i$ is rigid-analytic. The open subgroups $P^n \subset G, n\in \bN,$ form a cofinal system of 
$P_0$-stable neighborhoods of the identity. Hence by choosing a refinement of the covering we may assume that there is some $n$ such that 
it is of the shape $G_0= \bigcup_{g \in R} gP^n$ as above. Thus we get an element in ${\rm Ind}^{G_0}_{P^n}(V_n).$

Hence we have constructed two maps which are clearly inverse to each other. It follows from the definition of the space of $W'$-valued locally analytic functions that these maps are also topological isomorphisms.
\qed

Next we consider $W$ as a Lie algebra representation of $\frp.$
Since the universal enveloping algebra of $\frg$ splits (by the PBW-theorem) into a tensor product $U(\frg)=U(\fru_P^-)\otimes_K U(\frp),$ we get  $U(\frg) \otimes_{U(\frp)} W'=U(\fru_P^-)\otimes_K W'.$
Similarly as above, there is an action of $P_0$ on this space.
On elements $x\otimes w \in U(\frg) \otimes_{U(\frp)} W'$, a given $p\in P_0$ acts by
$$p\cdot (x\otimes w)={\rm Ad}(p)(x)\otimes pw.$$
Here ${\rm Ad}$ denotes the adjoint action.
Under the identification $U(\frg) \otimes_{U(\frp)} W' = U(\fru^{-}) \otimes_K W'$
the action of the Levi factor $L_P$ is again  via the adjoint action on $U_P^{-,n}$ and the natural one on $W'.$
We stress that there is no natural action of $U_P^{-,n}$ on  the above space since it is not complete in a suitable sense, cf. 
Step 2 in the proof of Thm. \ref{irredH}.

Let $L_{\beta_1},\ldots , L_{\beta_r}$ be a basis of the vector space $\fru_P^-$. We consider the $L_P$-equivariant pairing
\setcounter{enumi}{4}
\begin{eqnarray}\label{pairing1}  \cO({\bf U_P^-}) \times  U(\fru_P^-) &\rightarrow& K \\
\nonumber (f,\frz) &\mapsto & (\frz \cdot f)(1).
\end{eqnarray}
This is a non-degenerate pairing and induces therefore a $K$-linear $L_P$-equivariant injection
$$\cO({\bf U_P^-}) \hookrightarrow {\rm Hom}_K(U(\fru_P^-),K) . $$
\noindent given by
$$X_{\beta_1}^{i_1}\cdots X_{\beta_r}^{i_r} \mapsto (i_1)! \cdots (i_r)! \cdot (L^{i_1}_{\beta_1}\cdots L^{i_r}_{\beta_r})^\ast.$$

\noindent Here $\big\{ (L^{i_1}_{\beta_1}\cdots L^{i_r}_{\beta_r})^\ast \mid
(i_1,\ldots,i_r)\in \bN_0^r \big\}$ is the dual basis of $\{
L^{i_1}_{\beta_1}\cdots L^{i_r}_{\beta_r}\mid(i_1,\ldots,i_r)\in
\bN_0^r \}.$
The pairing \ref{pairing1} extends by continuity to a non-degenerate $(L_P)_{0}$-equivariant pairing
\begin{numequation}\label{pairing2}
\cO(U_P^{-,n}) \times  U(\fru_P^-) \rightarrow K .
\end{numequation}
which induces a $P_0$-equivariant pairing
\begin{numequation}\label{pairing3}
(\;,\;):  (\cO(U_P^{-,n}) \otimes_K W')  \times  (U(\fru_P^-) \otimes_K W) \rightarrow K
\end{numequation}
\vspace{-0.3cm}
$$ (f\otimes n,\frz\otimes \phi) \mapsto \phi(n)\cdot (\frz \cdot f)(1).$$

\vskip5pt
For $\epsilon \in |\overline{K^\ast}|,$ we consider the  norm $|\;\;|_\epsilon$ on $U(\fru_P^-)$ given by

\vspace{-0.3cm}
\begin{numequation}\label{norm}
\left|\sum_{(i_1,\ldots,i_r) \in \bN^r_0} a_{i_1,\ldots,i_r} L^{i_1}_{\beta_1}\cdots L^{i_r}_{\beta_r} \right|_\epsilon:= \sup_{(i_1,\ldots,i_r) \in \bN^r_0} |(i_1)! \cdots (i_r)! \cdot  a_{i_1,\ldots,i_r}|\epsilon^{i_1+\cdots + i_r}.
\end{numequation}

\noindent The completion of $U(\fru_P^-)$ with respect to this norm yields the $K$-Banach space

\vspace{-0.3cm}
\begin{eqnarray*}
U(\fru_P^-)_{\epsilon}  := & \Big\{& \sum\nolimits_{(i_1,\ldots,i_r) \in \bN^r_0} a_{i_1,\ldots,i_r} L^{i_1}_{\beta_1}\cdots L^{i_r}_{\beta_r} \mid a_{i_1,\ldots,i_r} \in K,\; \\
& & |(i_1)! \cdots (i_r)! \cdot  a_{i_1,\ldots,i_r}|\epsilon^{i_1+\cdots + i_r} \rightarrow 0, i_1+\cdots + i_r \rightarrow \infty \Big\}.
\end{eqnarray*}

\noindent We abbreviate $U(\fru_P^-)_n :=U(\fru_P^-)_{\frac{1}{\epsilon_n}}.$
The pairing \ref{pairing2}
extends to
$$\cO_b(U_P^{-,n}) \times  U(\fru_P^-)_n \rightarrow K $$
such that
$\cO_b(U_P^{-,n})$ becomes the topological dual of
$U(\fru_P^-)_n$ (cf. the Example in \cite[ch. I, \S 3]{S1}). It follows that $\cO_b(U_P^{-,n})\otimes_K W'$ is the topological dual of $U(\fru_P^-)_n \otimes_K W$. In particular, we get an action of $U^{-,n}_P$ and hence of $P^n$ on
$U(\fru_P^-)_n \otimes_K W.$

The following statement has been proved for ${\bf G=\GL_n}$ in \cite{O}.

\end{para}

\begin{prop}\label{Duality} There is an isomorphism of (Hausdorff) locally convex $K$-vector spaces
$$\varinjlim\nolimits_{n\in \bN} \cO(U_P^{-,n})\otimes_K W' \stackrel{\sim}{\longrightarrow}  \Big(\varprojlim\nolimits_{n \in \bN} U(\fru_P^{-})_n \otimes_K W\Big)' $$

\vskip8pt

\noindent compatible with the action of
$\varprojlim_n P^n =  P_{0}$.
\end{prop}

\proof The proof is the same as in loc.cit. and proceeds as follows. There are identifications
$$\varinjlim\nolimits_n \cO_b(U_P^{-,n}) = \varinjlim\nolimits_n \cO(U_P^{-,n})$$


\noindent resp.
$$\varinjlim\nolimits_n \cO_b(U_P^{-,n})\otimes_K W'  = \varinjlim\nolimits_n \cO(U_P^{-,n})\otimes_K W'$$


\noindent of locally convex $K$-vector spaces.  As we saw above, the space $\cO_b(U_P^{-,n})\otimes_K W'$ is the topological dual of $U(\fru_P^-)_n \otimes_K W$. The claim follows now from \cite[Thm. 3.4]{Mo1} respectively \cite[Prop. 16.10]{S1} on the duality of projective limits of $K$-Fr\'echet spaces and injective limits of $K$-Banach spaces with compact
transition maps. \qed

\bigskip
Now we consider more generally a surjective map
$$\phi: U(\frg)\otimes_{U(\frp)} W \rightarrow M$$
\noindent of $U(\frg)$-modules
as in \ref{basicsequence}. Let $\frd=\ker \phi$ be its kernel. We may consider $\frd$ by PBW as a submodule of $U(\fru_P^-) \otimes_K W$. 
We denote by
$$\frd_n\subset U(\fru_P^{-})_n \otimes_K W$$


\noindent its topological closure in $U(\fru_P^{-})_n \otimes_K W$. Finally, we put
$$(\cO(U_P^{-,n})\otimes_K W')^\frd=\{f\in \cO(U_P^{-,n})\otimes_K W' \mid \langle \frz,f \rangle =0 \;\forall \;\frz\in \frd\} \;.$$


\noindent Then $U(\fru_P^{-})_n\otimes_K W'/\frd_n$ and $(\cO(U_P^{-,n})\otimes_K W')^\frd$ are $K$-Banach spaces equipped with a 
natural $P^n$-action, as well. The following statement generalizes Proposition \ref{Duality}.

\begin{prop}\label{DualityII} There is an isomorphism of (Hausdorff) locally convex $K$-vector spaces

$$\varinjlim\nolimits_{n\in \bN} (\cO(U_P^{-,n})\otimes_K W')^\frd \stackrel{\sim}{\longrightarrow}  \Big(\varprojlim\nolimits_{n \in \bN} U(\fru_P^{-})_n\otimes_K W/\frd_n\Big)' $$

\vskip8pt

\noindent compatible with the action of $\varprojlim_n P^n =  P_{0}$.
\end{prop}

\noindent \Pf We define  similarly $(\cO(U_P^{-,n})\otimes_K W')^{\frd_n}, (\cO_b(U_P^{-,n})\otimes_K W')^\frd$ and $(\cO_b(U_P^{-,n})\otimes_K W')^{\frd_n}.$
By the density of $\frd$ in $\frd_n$  we have
$$(\cO(U_P^{-,n})\otimes_K W')^\frd=(\cO(U_P^{-,n})\otimes_K W')^{\frd_n}$$
\noindent resp.
$$(\cO_b(U_P^{-,n}) \otimes_K W')^\frd=(\cO_b(U_P^{-,n}) \otimes_K W')^{\frd_n} \;.$$


\noindent Now, the $K$-Banach space $(\cO_b(U_P^{-,n})\otimes_K W')^{\frd_n}$ is the topological dual of 
$U(\fru_P^{-})_n\otimes_K W/\frd_n$ (here we use that $K$ is spherically complete, cf. \cite[9.4]{S1}). Then we proceed with the argumentation as in the proof of Proposition \ref{Duality}. \qed

\medskip

Finally we arrive at the next alternative description of $\Ind^{G}_{P}(W')^\frd$.

\begin{cor}\label{duality_cor}  There is an isomorphism of locally convex $K$-vector spaces compatible with the action of $G_0$:
$$\Ind^{G_0}_{P_0}(W')^\frd = \varinjlim\nolimits_n  \Ind^{G_0}_{P^n}(\cO(U_P^{-,n}) \otimes_K W')^\frd \cong \Big(\varprojlim\nolimits_n
\Ind^{G_0}_{P^n}( U(\fru_P^-)_n\otimes_K W/\frd_n) \Big)' \; .$$
\end{cor}

\Pf Lemma \ref{Lemma_ind_proj} together with the identity (\ref{rigid_analytic}) give rise to an isomorphism
$\Ind^{G_0}_{P_0}(W')^\frd = \varinjlim\nolimits_n  \Ind^{G_0}_{P^n}(\cO(U_P^{-,n}) \otimes_K W')^\frd.$ With the same reasoning
as in the proof of the foregoing statement we see that $\Ind^{G_0}_{P^n}(\cO_b(U_P^{-,n}) \otimes_K W')^\frd$
is the topological dual of $\Ind^{G_0}_{P^n}( U(\fru_P^-)_n\otimes_K W/\frd_n)$. Passing to the limit yields the claim.
\qed


\section{The functor $\cF^G_P$ and its properties}

\begin{para} {\it The functor $\cF^G_P$.}
Denote by $\Rep^{\ell a}_K(G)$  the category of locally analytic representations of $G$ on 
barrelled locally convex Hausdorff $K$-vector spaces. In the previous section we defined for any object $M \in \cO^\frp_\alg$ the $G$-representation $\Ind^G_P(W')^\frd$, cf. \ref{basicdfn}. By \ref{Prop_admissible} it is strongly admissible. In particular, its underlying topological vector space is reflexive. Furthermore, the continuous dual space, equipped with the strong topology, is isomorphic to $D(G) \otimes_{D(\frg,P)} M$, as a $D(G)$-module, cf. \ref{basiciso}. This prompts us to consider the functor
$$\cF^G_P: \cO^\frp_{\alg} \lra \Rep^{\ell a}_K(G)$$
defined by
$$\cF^G_P(M) = \left( D(G) \otimes_{D(\frg,P)} M \right)' \,\,.$$
\end{para}

\vskip8pt

\begin{prop}\label{exact} The functor $\cF^G_P$ is exact.
\end{prop}

\noindent \Pf Let $0\rightarrow M_1\rightarrow M_2 \rightarrow M_3 \rightarrow 0$ be an exact sequence in the category $\cO_\alg^\frp.$ Then there are finite-dimensional algebraic generating ${\bf P}$-representations $W_i\subset M_i$, $i=1,2,3$,  such that we have an induced exact sequence
$0 \rightarrow W_1\rightarrow W_2 \rightarrow W_3 \rightarrow 0$. Indeed, let $W_2$ be a ${\bf P}$-submodule which generates $M_2$. Then the image $W_3$ of $W_2$ in $M_3$ generates $M_3$. Consider $X:=\ker(W_2\rightarrow W_3).$ If it does not generate $M_1$, then we  let $Y$ be any generating system and put $W_1=X+Y$ resp. replace $W_2$ by $W_2+Y$.  The resulting sequence satisfies the claim.

We get a commutative diagram with exact rows:
\vspace{1cm}

\hspace{1.5cm}$\begin{array}{ccccccccc} & & 0 &  & 0  & & 0 & & \\
& & \uparrow &  & \uparrow  & &\uparrow & & \\
0 & \rightarrow &  M_1 &  \rightarrow & M_2 & \rightarrow & M_3 & \rightarrow & 0 \\
& & \uparrow &  & \uparrow  & &\uparrow & & \\
0 & \rightarrow &   U(\frg) \otimes_{U(\frp)} W_1 & \rightarrow &  U(\frg) \otimes_{U(\frp)} W_2 & \rightarrow  &  U(\frg) \otimes_{U(\frp)} W_3 & \rightarrow &  0    \\
& & \uparrow & & \uparrow &  & \uparrow & & \\
0 & \rightarrow  & \frd_1 & \rightarrow  & \frd_2 & \rightarrow & \frd_3 & \rightarrow &  0 \\
& & \uparrow &  & \uparrow  & &\uparrow & & \\
& & 0 &  & 0  & & 0 & &
\end{array}$

\medskip
\noindent Note, that exactness of the middle row follows from the exactness of the functor $U(\frg)\otimes_{U(\frp)} -.$ 
The exactness of the last row is a consequence of the snake resp. $3 \times 3$ lemma. If we apply our functor $\cF=\cF^G_P$ to the above diagram we get a new diagram:
\vspace{0.7cm}

\hspace{1.5cm}$\begin{array}{ccccccccc} & & 0 &  & 0  & & 0 & & \\
& & \downarrow &  & \downarrow  & &\downarrow & & \\
0 & \leftarrow &  \cF(M_1) &  \leftarrow &  \cF(M_2) & \leftarrow & \cF(M_3) & \leftarrow & 0 \\
& & \downarrow &  & \downarrow  & &\downarrow & & \\
0 & \leftarrow &   \Ind^{G}_P(W_1') & \leftarrow &  \Ind^{G}_P(W_2') & \leftarrow  &  \Ind^{G}_P(W_3') & \leftarrow &  0    \\
& & \downarrow & & \downarrow &  & \downarrow & & \\
0 & \leftarrow  &  \cF(\frd_1) & \leftarrow  & \cF(\frd_2) & \leftarrow & \cF(\frd_3) & \leftarrow &  0 \\
& & \downarrow &  & \downarrow  & &\downarrow & & \\
& & 0 &  & 0  & & 0 & &
\end{array}$

\medskip

\noindent The middle row is obviously exact. The first row coincides  by definition with the sequence
$$ 0 \leftarrow   \Ind^{G}_P(W_1')^{\frd_1}  \leftarrow  \Ind^{G}_P(W_2')^{\frd_2}  \leftarrow   \Ind^{G}_P(W_3')^{\frd_3} \leftarrow   0.$$
The following argumentation shows that it is exact on the right. Indeed, the functor $D(G)\otimes_{U(\frg,P} (\cdot)$ is right
exact and, by coadmissibility, yields strict $K$-linear maps. By the theorem of Hahn-Banach it follows that our functor $\cF^G_P$ is in particular left exact. 
So all we have to show by using the snake lemma applied to the last two rows is to prove that the columns of the diagram  are exact.
Exemplarily, we consider the middle column and abbreviate $\frd=\frd_2$ resp. $W=W_2$. The submodule $\frd$ is again an object of the category $\cO^\frp_\alg$.
Thus we have a short exact sequence
$$0 \rightarrow \bar{\frd} \rightarrow U(\frg) \otimes_{U(\frp)} \bar{W} \rightarrow \frd \rightarrow 0$$
for some finite-dimensional algebraic ${\bf P}$-representation $\bar{W}$.  Hence $\cF(\frd)=\Ind^{G}_P(\bar{W}')^{\bar{\frd}}.$
In order to show the surjectivity of $\Ind^{G}_P(W') \rightarrow \cF(\frd)$, it is enough to show
the injectivity of its dual map since the image of $\Ind^{G}_P(W') \rightarrow \cF(\frd)$ is closed by admissibility.
Applying the description in Corollary \ref{duality_cor} the dual map is given by
$$\varprojlim\nolimits_n \Ind^{G_0}_{P^n}(U(\fru^-)_n \otimes_K \bar{W}/\bar{\frd}_n) \rightarrow \varprojlim\nolimits_n \Ind^{G_0}_{P^n}(U(\fru^-)_n \otimes_K W).$$
Now the functor  $\varprojlim_n$ is left exact on projective systems of $K$-vector spaces. 
Hence it suffices to prove that for each $n\in \bN$ the map 
$g_n: U(\fru^-)_n \otimes_K \bar{W}/\bar{\frd}_n \rightarrow U(\fru^-)_n \otimes_K W$, which is induced by the map $g:U(\fru^-) \otimes_K \bar{W}/\bar{\frd}\cong \frd \hookrightarrow U(\fru^-) \otimes_K W$ by passing to completions with respect 
to the norm \ref{norm}, is injective. The left hand side $U(\fru^-)_n \otimes_K \bar{W}/\bar{\frd}_n$ carries the quotient norm. The map $g_n$ is continuous (this can been seen, e.g. by using the continuity criterion with respect to sequences). Thus its kernel is a closed subspace. On the other hand, by the same reasoning as in \cite[Lemma 3.4.4]{OS} we deduce that the natural action of $\frt$ on $U(\fru^-)$ is continuous with respect to the norm on $U(\fru^-)_n$.  Hence all assumptions are satisfied for applying \cite[Korollar 1.3.12]{F}.
The latter result says that the intersection of ${\rm ker}(g_n)$  with $U(\fru^-) \otimes_K \bar{W}/\bar{\frd}$ is non-trivial, if ${\rm \ker}(g_n)$ is non-trivial.
This is a contradiction since this intersection is nothing else but ${\rm ker}(g)=\{0\}$. The claim follows. \qed

\vskip8pt

\begin{cor}\label{nonzero} If $M \neq 0$ then $\cF^G_P(M) \neq 0$.
\end{cor}

\noindent \Pf Let $M \ra N$ be a simple quotient of $M$. By \ref{exact}, this map induces an injective map $\cF^G_P(N) \ra \cF^G_P(M)$. We can therefore assume that $M$ is simple, and it is therefore a highest weight module, generated by some vector $v^+$ with highest weight $\lambda$. Let $W \sub M$ be a finite-dimensional $\frp$-module which contains $v^+$, and hence generates $M$ as $U(\frg)$-module. The weight space of $U(\frg) \otimes_{U(\frp)} W$ with weight $\lambda$ is then the one-dimensional $K$-vector space generated by $1 \otimes v^+$. Put $\frd = \ker(U(\frg) \otimes_{U(\frp)} W \ra M)$.
Then $1 \otimes v^+$ is not contained in $\frd$, and all weights in $\frd$ are distinct from $\lambda$. In the following we consider $\frd$ as a subspace of $U(\fru_P^-) \otimes_K W$, which is $\frt_K$-stable.

With the notation introduced in \ref{yetanother}, $U(\fru_P^-)_n \otimes_K W$ is a Banach space completion of $U(\fru_P^-) \otimes_K W$. As above, let $\frd_n$ be the closure of $\frd$ in $U(\fru_P^-)_n \otimes_K W$. By the last assertion in \cite[1.3.12]{F}, the weight spaces of $\frd_n$ with respect to the $\frt_K$-action are all contained in $\frd$, and we conclude that $\frd_n$ does not contain $1 \otimes v^+$, and the class of $1 \otimes v^+$ in $\left(U(\fru_P^-)_n \otimes_K W \right)/\frd_n$ is non-zero for all $n$. Therefore, the projective limit
$$\varprojlim\nolimits_n
\Ind^{G_0}_{P^n}( U(\fru_P^-)_n\otimes_K W/\frd_n)$$


\noindent is non-zero. By \ref{duality_cor}, this projective limit is reflexive, and its continuous dual space is $\cF^G_P(M)$, which is thus non-zero as well. \qed

\setcounter{enumi}{0}

\begin{para}\label{extending} {\it Extending the functor $\cF^G_P$.} Let $P \sub G$ be a standard parabolic subgroup. 
Let $V$ be a $K$-vector space, equipped with a smooth admissible representation of the Levi subgroup $L_P \sub P$, and regard 
it via inflation as a representation of $P$. We recall that we always consider on $V$ the finest locally convex $K$-vector space topology, i.e., the locally convex inductive limit of its finite-dimensional subspaces. As such $V$ is of compact type and furnishes a locally analytic $P$-representation.

Let $M$ be an object of $\cO^\frp_\alg$ and write it as a quotient of a generalized Verma module as in \ref{basicsequence}:

\vspace{-0.3cm}
$$0 \ra \frd \ra U(\frg) \otimes_{U(\frp)} W \ra M \ra 0 \,.$$


\noindent As $W$ is a finite-dimensional space, the injective tensor product $W'\otimes_{K,\iota} V$ coincides with the projective tensor product $W'\otimes_{K,\pi} V$. Therefore we simply write $W'\otimes_K V$ for it. It is complete and we have $W'\otimes_K V = \varinjlim_H W'\otimes_K V^H$, where $H$ runs through all compact open subgroups of $P$, holds as locally convex vector spaces. Equipped with the diagonal action of $P$, cf. \ref{lemma_lift_alg}, $W' \otimes_K V$ is a locally analytic representation. We set

\vspace{-0.3cm}
\begin{numequation}\label{explicit}
\cF^G_P(M,W,V) = \Ind^G_P(W' \otimes_K V)^\frd = \{f \in \Ind^G_P(W' \otimes_K V) \midc \forall  \frz \in \frd: \langle \frz, f \rangle_{C^{an}(G,V)} = 0 \} \,,
\end{numequation}
\noindent where the pairing $\langle \cdot , \cdot \rangle_{C^{an}(G,V)}$ is defined as in \ref{pair_C(G,K)}. We are going to show that $\cF^G_P(M,W,V)$ is independent of the chosen $P$-representation $W$, in the sense of \ref{well-dfd} below. Later on we will therefore simplify notation by writing $\cF^G_P(M,V)$ instead of $\cF^G_P(M,W,V)$.
\end{para}

\vskip5pt
\begin{prop}\label{well-dfd} Let $G$, $P$, $M$, and $V$ be as above. Let $W_1 \sub W_2$ be two finite-dimensional $U(\frp)$-submodules of $M$ which generate $M$ as a $U(\frg)$-module. Then the canonical map
$\Ind^G_P(W_2' \otimes_K V) \ra \Ind^G_P(W_1' \otimes_K V)$ of  $G$-representations, induced by the map $W_2' \ra W_1'$, restricts to a topological isomorphism 

\vspace{-0.3cm}
$$\cF^G_P(M,W_2,V) \stackrel{\simeq}{\lra} \cF^G_P(M,W_1,V) \;.$$

\end{prop}

\noindent \Pf  For $i = 1,2$, let $0 \ra \frd_i \ra U(\frg)\otimes_{U(\frp)} W_i \ra M \ra 0$ be as in \ref{basicsequence}. 
It is immediate that under the (injective) map $U(\frg)\otimes_{U(\frp)} W_1 \ra U(\frg)\otimes_{U(\frp)} W_2$ induced by the inclusion $W_1 \sub W_2$ the submodule $\frd_1$ is mapped to $\frd_2$. It follows that the canonical map $\Ind^G_P(W_2' \otimes_K V) \ra \Ind^G_P(W_1' \otimes_K V)$ maps $\cF^G_P(M,W_2,V)$ into $\cF^G_P(M,W_1,V)$.

We fix a locally $L$-analytic section $s: G/P \ra G$ of the projection $G \ra G/P$ and let $\cH = s(G/P) \sub G$ be its image, so that we have an isomorphism of locally $L$-analytic manifolds $G \simeq \cH \times P$ (this isomorphism depends on $s$). Using this isomorphism, we can identify the underlying topological vector space of $\Ind^G_P(W_i' \otimes_K V)$ with
$$C^{an}_L(\cH) \hat{\otimes}_K (W_i' \otimes_K V) = C^{an}_L(\cH) \otimes_K W_i' \otimes_K V \;,$$


\noindent the completed tensor product on the left being the projective tensor product, cf. \cite[formula (56) and Remark 5.4]{K2}. This completed tensor product is equal to the ordinary tensor product, as $W' \otimes_K V$ is an inductive limit of finite-dimensional vector spaces. Using this isomorphism, which is natural in $W_i'$ and $V$, we
identify $f \in \Ind^G_P(W_i' \otimes_K V)$ with $\psi_f = \sum_k \psi_k \otimes \phi_k \otimes v_k \in C^{an}_L(\cH) \otimes_K W_i' \otimes_K V$, where $\psi_k \in C^{an}_L(\cH)$, $\phi_k \in W_i'$, $v_k \in V$. Then, for $f$ to be annihilated by $\frd_i$ translates into a condition on $\psi_f$ which is only a condition 
on $\sum_k \psi_k \otimes \phi_k$. With these identifications we find an isomorphism of $K$-vector spaces, natural in $W_i'$ and $V$,

\vspace{-0.3cm}
$$\Ind^G_P(W_i' \otimes_K V)^{\frd_i} \stackrel{\simeq}{\lra} \Big(C^{an}_L(\cH) \otimes_K W_i'\Big)^{\frd_i} \otimes_K V \;.$$


\noindent We get therefore a commutative diagram

\vspace{-0.3cm}
$$\begin{array}{ccccc}
\Ind^G_P(W_2' \otimes_K V)^{\frd_2} & \stackrel{\simeq}{\lra} & \Big(C^{an}_L(\cH) \otimes_K W_2'\Big)^{\frd_2} \otimes_K V & \stackrel{\simeq}{\longleftarrow} & \Ind^G_P(W_2')^{\frd_2} \otimes_K V\\
\downarrow & & & & \downarrow\\
\Ind^G_P(W_1' \otimes_K V)^{\frd_1} & \stackrel{\simeq}{\lra} & \Big(C^{an}_L(\cH) \otimes_K W_1'\Big)^{\frd_1} \otimes_K V  & \stackrel{\simeq}{\longleftarrow} & \Ind^G_P(W_1')^{\frd_1} \otimes_K V\\
\end{array}$$

\noindent The vertical map on the right is an isomorphism because the map $\Ind^G_P(W_2')^{\frd_2} \ra 
\Ind^G_P(W_1')^{\frd_1}$ is an isomorphism of topological vector spaces, by \ref{basiciso}. The map on the left must hence be an isomorphism too.
\qed

\vskip5pt

\begin{para} Let $M \in \cO^\frp_\alg$ and $V$ be as above. Given two finite-dimensional $U(\frp)$-submodules $W_1, W_2 \sub M$, which both generate $M$ as a $U(\frg)$-module, we let $W_3 \sub M$ be any finite-dimensional $U(\frp)$-submodule containing both $W_1$ and $W_2$
(e.g., $W_3 = W_1 + W_2$). By \ref{well-dfd} the canonical maps

\vspace{-0.3cm}
\begin{numequation}\label{canonical-isos}
\cF^G_P(M,W_2,V) \longleftarrow \cF^G_P(M,W_3,V) \lra \cF^G_P(M,W_1,V)
\end{numequation}
\noindent are both isomorphisms of locally analytic $G$-representations. We have therefore a canonical isomorphism $\cF^G_P(M,W_2,V) \lra \cF^G_P(M,W_1,V)$ which does not depend on the choice of $W_3$. This is the unique isomorphism which is obtained from choosing any $W_3$ as above and inverting the map on the left of the resulting diagram \ref{canonical-isos}. Via these uniquely specified isomorphisms we can henceforth identify all representations $\cF^G_P(M,W,V)$ and write $\cF^G_P(M,V)$ for any one of them.
\end{para}

\vskip5pt

\begin{prop}\label{functor} $\cF^G_P(\cdot, \cdot)$ is a bi-functor 

\vspace{-0.3cm}
$$\cO^\frp_{\alg} \times \Rep^{\infty,{\rm adm}}_K(L_P) \longrightarrow  \Rep_K^{\ell a}(G) \,, \; (M,V) \mapsto \cF^G_P(M,V) \,,$$

\noindent which is contravariant in $M$ and covariant in $V$. Here $\Rep^{\infty,{\rm adm}}_K(L_P)$ is the category of smooth admissible representations of the Levi subgroup $L_P \sub P$  on $K$-vector spaces.
\end{prop}

\noindent \Pf Let $\alpha: M \ra N$ be a morphism in $\cO^\frp_\alg$, and let $\beta: U \ra V$ be a morphism in $\Rep^{\infty,{\rm adm}}_K(L_P)$. Choose a finite-dimensional $U(\frp)$-submodule $W_M \sub M$ which generates $M$ as a $U(\frg)$-module. Then choose a finite-dimensional $U(\frp)$-submodule $W_N \sub N$ which generates $N$ as a $U(\frg)$-module and which contains $\alpha(W_M)$. With the notation as in \ref{basicsequence} we get a commutative diagram

\vspace{-0.3cm}
$$\begin{array}{ccccccccc}
0 & \ra & \frd_M & \ra & U(\frg)\otimes_{U(\frp)} W_M & \ra & M & \ra & 0\\
& & & & \downarrow & & \downarrow & & \\
0 & \ra & \frd_N & \ra & U(\frg)\otimes_{U(\frp)} W_N & \ra & N & \ra & 0
\end{array} $$

\noindent It follows that the map in the middle maps $\frd_M$ into $\frd_N$. This shows that the map 

\vspace{-0.3cm}
$$\Ind^G_P(W_N' \otimes_K U) \lra \Ind^G_P(W_M' \otimes_K V)$$

\noindent  induced by $\alpha' \otimes \beta:  W_N' \otimes_K U \lra W_M' \otimes_K V$ maps $\cF^G_P(N,U)$ into $\cF^G_P(M,V)$.
\qed

\vskip5pt
\begin{prop}\label{general admissibility}
(i) For all $M \in \cO^\frp_\alg$, and for all smooth admissible $L_P$-representations $V$ the $G$-representation $\cF^G_P(M,V)$ is admissible.\vskip5pt

(ii) If $V$ is of finite length, then $\cF^G_P(M,V)$ is strongly admissible for all $M \in \cO^\frp_\alg$. 

\end{prop}

\noindent \Pf (i) By \ref{tensorprod} (i), the representation $\Ind^G_P(W' \otimes_K V)$ is admissible. By \cite[6.4 iii.]{ST2}, the closed subrepresentation $\cF^G_P(M,V) = \Ind^G_P(W' \otimes_K V)^\frd$ is admissible as well.

(ii) If $V$ is of finite length, then $\Ind^G_P(W' \otimes_K V)$ is strongly admissible by \ref{tensorprod} (ii). As a closed subrepresentation of $\Ind^G_P(W' \otimes_K V)$, the representation $\cF^G_P(M,V)$ is strongly admissible, cf. \cite[3.5]{ST1}. \qed

\vskip8pt

\begin{prop}\label{exact_in_both} a) The bi-functor $\cF^G_P$ is exact in both arguments.

\vskip5pt

\noindent b) If $Q \supset P$ is a parabolic subgroup, $\frq = \Lie(Q)$, and $M$ an object of $\cO^\frq_\alg$, then

\vspace{-0.3cm}
$$\cF^G_P(M,V) = \cF^G_Q(M,i^{L_Q}_{L_P(L_Q \cap U_P)}(V)) \,,$$

\noindent where $i^{L_Q}_{L_P(L_Q \cap U_P)}(V)=i^Q_P(V)$  denotes the corresponding induced representation in the category of smooth representations.
\end{prop}

\vskip5pt

\noindent \Pf a) The arguments here are similar to the ones used in the proof of \ref{well-dfd}. We fix a locally $L$-analytic section $s: G/P \ra G$ of the projection $G \ra G/P$ and let $\cH = s(G/P) \sub G$ be its image. Then we obtain (as in the proof of \ref{well-dfd}) an isomorphism of $K$-vector spaces, which is natural in $W'$ and $V$,

\vspace{-0.3cm}
$$\Ind^G_P(W' \otimes_K V)^\frd \cong \Big(C^{an}_L(\cH) \otimes_K W'\Big)^\frd \otimes_K V \;.$$


\noindent This shows that $\cF^G_P$ is exact in the second argument. That it is exact in the first argument is Prop. \ref{exact}.

\vskip8pt

b) Let $M$ be an object of $\cO^\frq_\alg$, and let $W \sub M$ be a finite-dimensional $\frq_K$-submodule which generates 
$M$ as a $U(\frg)$-module. Consider the canonical map  $U(\frg)\otimes_{U(\frq)} W \ra M$, and denote by $\frd$ its kernel. We have the 
tautological exact sequence
\begin{numequation}\label{tautexseq}
0 \lra \frd \lra U(\frg)\otimes_{U(\frq)} W \lra M \lra 0 \;.
\end{numequation}
\noindent Similarly, let ${\bf 1}$ the trivial one-dimensional representation (of $\frq_K$ or $\frp_K$), let $U(\frq) \otimes_{U(\frp)} {\bf 1} \ra {\bf 1}$ be the canonical map, $\frd^\frq_\frp$ its kernel, and consider the corresponding exact sequence
\begin{numequation}\label{indPQ}
0 \lra \frd^\frq_\frp \lra U(\frq) \otimes_{U(\frp)} {\bf 1} \lra {\bf 1} \lra 0 \;.
\end{numequation}
\noindent There is a canonical isomorphism $U(\frq)\otimes_{U(\frp)} W \stackrel{\simeq}{\lra} W \otimes_K (U(\frq) \otimes_{U(\frp)} {\bf 1})$ which sends $1 \otimes w$ to $w \otimes (1 \otimes 1)$ (cf. \cite[Prop. 6.5]{Kn}). This isomorphism, together with \ref{indPQ} shows that there is a canonical exact sequence
\begin{numequation}\label{indPQW}
0 \lra W \otimes_K \frd^\frq_\frp \lra W \otimes_K (U(\frq) \otimes_{U(\frp)} {\bf 1}) \simeq U(\frq)\otimes_{U(\frp)} W \lra W \lra 0 \;.
\end{numequation}
\noindent Inducing \ref{indPQW} from $U(\frq)$ to $U(\frg)$ (which is an exact functor by the PBW theorem) gives
the exact sequence
\begin{numequation}\label{indQG}
0 \lra U(\frg) \otimes_{U(\frq)} ( W \otimes_K \frd^\frq_\frp) \lra  U(\frg) \otimes_{U(\frp)} W \lra U(\frg) \otimes_{U(\frq)} W \lra 0 \;.
\end{numequation}
\noindent If we denote the kernel of the natural map $U(\frg) \otimes_{U(\frp)} W \ra M$
by $\tilde{\frd}$, we have the exact sequence
\begin{numequation}\label{indPQZ}
0 \lra \tilde{\frd} \lra U(\frg)\otimes_{U(\frp)} W \lra M \lra 0 \;,
\end{numequation}
\noindent and from \ref{indQG} and \ref{tautexseq} we deduce the exact sequence
\begin{numequation}\label{kernel}
0 \lra U(\frg) \otimes_{U(\frq)} ( W \otimes_K \frd^\frq_\frp) \lra \tilde{\frd} \lra \frd \lra 0 \;.
\end{numequation}
\noindent Furthermore, there is a canonical isomorphism
\begin{numequation}\label{transofind}
\Ind^G_P(W' \otimes_K V) \simeq \Ind^G_Q\Big(\Ind^Q_P (W' \otimes_K V)\Big) \;.
\end{numequation}
\noindent This follows from \cite[5.3 and 2.6]{K2}. Let $(\phi_1,\ldots, \phi_d)$ be a fixed basis of $W'$. Then, for every $q \in Q$, the family $(q^{-1}.\phi_1, \ldots, q^{-1}.\phi_d)$ 
is a basis of $W'$, and given $f \in \Ind^Q_P(W' \otimes_K V)$, we can write 
$$f(q) = (q^{-1}.\phi_1) \otimes \psi_1(q) +  \ldots + (q^{-1}.\phi_d) \otimes \psi_d(q)$$ 
with uniquely determined $\psi_i(q) \in V$. It is straightforward to check that $\psi_i \in \Ind^Q_P(V)$, and that

\vspace{-0.3cm}
$$\Ind^Q_P(W' \otimes_K V) \stackrel{\simeq}{\lra} W' \otimes_K \Ind^Q_P(V) \;, \;\, f \mapsto \phi_1 \otimes \psi_1 + \ldots + \phi_d \otimes \psi_d \;,$$


\noindent is an isomorphism of locally analytic $Q$-representations; the inverse map sends the element $\phi \otimes \psi \in W' \otimes_K \Ind^Q_P(V)$ to $[g \mapsto (g^{-1}.\phi) \otimes \psi(g)] \in \Ind^Q_P(W' \otimes_K V)$. From now on we identify $\Ind^G_P(W' \otimes_K V)^{\tilde{\frd}}$ with $\Ind^G_Q\Big(W' \otimes_K \Ind^Q_P(V)\Big)^{\tilde{\frd}}$.
Suppose now that $f: G \ra W' \otimes_K \Ind^Q_P(V)$ lies in $\Ind^G_Q\Big(W' \otimes_K \Ind^Q_P(V)\Big)^{\tilde{\frd}}$.
Because $\tilde{\frd}$ contains $W \otimes_K \frd^\frq_\frp$ (cf. \ref{kernel}), it follows that $f$ takes values in

\vspace{-0.3cm}
$$W' \otimes_K \Ind^Q_P(V)^{\frd^\frq_\frp} = W' \otimes_K i^Q_P(V) \;.$$


\noindent This, together with \ref{kernel}, shows that

\vspace{-0.3cm}
$$\cF^G_P(M,V) = \Ind^G_Q\Big(W' \otimes_K \Ind^Q_P(V)\Big)^{\tilde{\frd}} =
\Ind^G_Q\Big(W' \otimes_K i^Q_P(V)\Big)^\frd = \cF^G_Q\Big(M,i^{L_Q}_{L_P(L_Q \cap U_P)}(V)\Big) \;.$$
\qed

\vskip8pt

\begin{rmk}\label{Remark}
We remark that also the description in Cor. \ref{duality_cor} has a counterpart in the 'enriched' version, i.e., including a smooth representation $V$.
Indeed, let $V$ be as before a smooth admissible $L_P$-representation. By letting act $U_P$ trivially on $V$,  it has the structure of a $P$-representation. Then we get a topological isomorphism
\begin{eqnarray*}
\Ind^{G_0}_{P_{0}}(W'\otimes_K V)^\frd &  \cong & \Big(\varprojlim\nolimits_n{\rm Ind}^{G_0}_{P^n}\big( (U(\fru_P^{-})_n \otimes_K W) \hat{\otimes}_K V'/\frd_n\big)\Big)'  .
\end{eqnarray*}
Indeed, for the proof we write again
$W'\otimes_K V =  \varinjlim\nolimits_H W' \otimes_K V^H.$ Here $H$ runs through the set of all normal compact  open subgroups of $(L_P)_0.$ Each subspace $W'\otimes_K V^H$ is stable by the action of $P_0.$
Moreover, we have $(W'\otimes_K V^H)'\cong W\otimes_K (V^H)'$ as $V^H$ is finite-dimensional (cf. \cite[Prop. 20.13]{S1}).
By the very definition of $C^{an}(G,W'\otimes V)$ \cite{ST1}  and again by the  finite-dimensionality of $V^H$ we see  by using a variant of Prop. \ref{DualityII} that
\begin{eqnarray*}
\Ind^{G_0}_{P_{0}}(W'\otimes_K V)^\frd & = &\varinjlim\nolimits_n \varinjlim\nolimits_H  {\rm Ind}^{G_0}_{P^{n}} (\cO(U_P^{-,n}) \otimes_K W' \otimes_K V^H)^\frd \\  &  = & \Big(\varprojlim\nolimits_n\varprojlim\nolimits_H {\rm Ind}^{G_0}_{P^{n}}( U(\fru_P^{-})_n \otimes_K W \otimes_K (V^H)'/\frd_n)\Big)' \\
&=&   \Big(\varprojlim\nolimits_n{\rm Ind}^{G_0}_{P^n}\big(((U(\fru_P^{-})_n \otimes_K W) \hat{\otimes}_K V'/\frd_n\big)\Big)'.
\end{eqnarray*}
\end{rmk}

\bigskip

\setcounter{enumi}{0}

\begin{para}\label{locanBGG}{\it Locally analytic {\rm BGG}-resolutions.} Recall that
for a character $\lambda \in X^\ast({\bf T})$, we denote by $M(\lambda)=U(\frg) \otimes_{U(\frb)} K_\lambda$ the corresponding Verma module and by $L(\lambda)$ its simple quotient.
Let $\Delta$ be the set of simple roots, $\Phi$
the set of  roots and $\Phi^+$ the set of positive roots of ${\bf G}$ with respect to our data  ${\bf T\subset B}$. Let
$$X_+=\{\lambda \in  X^\ast({\bf T}) \mid  \forall \alpha \in \Delta: (\lambda,\alpha^\vee )\geq 0 \}$$
\noindent be the set of dominant weights in $X^\ast({\bf T}).$
If $\lambda \in X_+$, then $L(\lambda)$ is finite-dimensional and  comes from an irreducible algebraic $G$-representation. In this situation, we also write $V(\lambda)$ for $L(\lambda).$  Finally, let $W$ be the Weyl group of ${\bf G}$ and denote by $w_0\in W$ its longest element with respect to the Bruhat order.

Denote by $\cdot$ the dot action of $W$ on $X^\ast({\bf T})$  given by
$$w\cdot \lambda= w(\lambda+\rho)-\rho \;,$$
where $\rho=\frac{1}{2}\sum_{\alpha \in \Phi^+} \alpha$.
For a dominant weight $\lambda\in X_+$, the {\rm BGG}-resolution of the finite-dimensional $G$-module $V(\lambda)$ has the following shape

\begin{numequation}\label{BGG}
0 \rightarrow  M(w_0\cdot \lambda) \rightarrow \bigoplus_{w\in W \atop \ell(w)=\ell(w_0)-1} M(w\cdot \lambda) \rightarrow \dots \rightarrow \bigoplus_{w\in W \atop \ell(w)=1} M(w\cdot \lambda) \rightarrow M(\lambda) \rightarrow
V(\lambda)\rightarrow 0.
\end{numequation}

\noindent We refer to \cite{Ku} for the definition of the differentials in this complex.
By applying our exact functor $\cF^G_B$, we get a resolution

\addtocounter{enumi}{1}
\begin{multline}
0 \leftarrow \cF^G_{B}( M(w_0\cdot \lambda)) \leftarrow \bigoplus_{w\in W\atop \ell(w)=\ell(w_0)-1} \cF^G_B(M(w\cdot \lambda))
\leftarrow \dots \\
 \dots \leftarrow \bigoplus_{w\in W \atop \ell(w)=1} \cF^G_B(M(w\cdot \lambda)) \leftarrow \cF^G_B (M(\lambda))
\leftarrow
\cF^G_B(V(\lambda))\leftarrow 0 \;,
\end{multline}
\noindent which by Prop. \ref{exact_in_both} coincides with

\addtocounter{enumi}{1}

\begin{multline}\label{lcBGG}
0 \leftarrow \Ind^{G}_{B}((w_0\cdot \lambda)^{-1}) \leftarrow \bigoplus_{w\in W\atop \ell(w)=\ell(w_0)-1} \Ind^{G}_B((w\cdot \lambda)^{-1})\leftarrow \dots \\
\dots \leftarrow \bigoplus_{w\in W \atop \ell(w)=1} \Ind^{G}_B((w\cdot \lambda)^{-1}) \leftarrow  \Ind^{G}_B(\lambda^{-1}) \leftarrow V(\lambda)\otimes i^G_B(K) \leftarrow 0 \;.
\end{multline}

There is also a parabolic version of the {\rm BGG}-resolution due to Lepowsky \cite{Le}.
Let $P=P_I\subset G, \,I\subset \Delta$ be a standard parabolic subgroup and let $W_I\subset W$ be the subgroup generated by the 
reflections in $I$. Let
$$X^+_I=\{\lambda \in X^\ast({\bf T}) \mid \forall \alpha \in I: (\lambda,\alpha) \geq 0\}$$
be the set of $L_I$-dominant weights.
In particular $X^+ \subset X_I^+$. Let $^IW=W_I\backslash W$ which we identify with the representatives  of shortest length in $W$.
Let $^Iw_0$ be the element of maximal length in $^IW$. If $\lambda$ is in $X_+$ and $w\in {}^IW$ then $w\cdot \lambda \in X_I^+$, cf. \cite[p. 502]{Le}. The $I$-parabolic {\rm BGG}-resolution of $V(\lambda)$ is given by a sequence
$$0 \rightarrow  M_I({}^Iw_0\cdot \lambda) \rightarrow \bigoplus_{w\in{}^IW\atop \ell(w)=\ell({}^Iw_0)-1} M_I(w\cdot \lambda) \rightarrow \dots
\rightarrow \bigoplus_{w\in{}^IW \atop \ell(w)=1} M_I(w\cdot \lambda) \rightarrow M_I(\lambda) \rightarrow
V(\lambda)\rightarrow 0 \;.$$
Again, we refer to \cite{Ku} for the definition of the differentials in this complex. By applying our exact functor $\cF^G_P$, we get a resolution
\addtocounter{enumi}{1}
\begin{multline}
0 \leftarrow \cF^G_P( M_I({}^Iw_0\cdot \lambda)) \leftarrow \bigoplus_{w\in {}^IW\atop \ell(w)=\ell({}^Iw_0)-1} \cF^G_P(M_I(w\cdot \lambda))
\leftarrow \dots \\
\dots \leftarrow \bigoplus_{w\in{} ^IW \atop \ell(w)=1} \cF^G_P(M_I(w\cdot \lambda)) \leftarrow \cF^G_P (M_I(\lambda))
\leftarrow
\cF^G_P(V(\lambda))\leftarrow 0 
\end{multline}
which coincides by Prop. \ref{exact_in_both} with
\addtocounter{enumi}{1}
\begin{multline}
0 \leftarrow \Ind^{G}_P(V_I({}^Iw_0\cdot \lambda)') \leftarrow \bigoplus_{w\in {}^IW\atop \ell(w)=\ell({}^Iw_0)-1} \Ind^G_P(V_I(w\cdot \lambda)')\leftarrow \dots \\
\dots \leftarrow \bigoplus_{w\in{}^IW \atop \ell(w)=1} \Ind^G_P(V_I(w\cdot \lambda)')\leftarrow  \Ind^G_P(V_I(\lambda)') \leftarrow V(\lambda)\otimes i^G_P(K) \leftarrow 0.
\end{multline}
\end{para}

\begin{exam} Let $\bG ={\rm SL}_2$, $\bT \sub \bG$ the diagonal split torus, and $\lambda \in X^\ast(\bT)$ the trivial character. Thus $L(\lambda) = K$ is the trivial representation. Put $\lambda' = s \cdot \lambda = s(\rho)-\rho = -2\rho \in X^\ast(\bT)$, where $s \in W$ is the non-trivial element. The {\rm BGG}-resolution of the trivial representation is the short exact sequence
$$0\rightarrow M(\lambda') \rightarrow M(\lambda) \rightarrow L(\lambda) \rightarrow 0.$$
By applying the functor $\cF^G_B$ to this sequence, we obtain an exact sequence
$$0\rightarrow i^G_B(K) \rightarrow \Ind^{G}_B(K) \rightarrow \Ind^{G}_B((\lambda')^{-1}) \rightarrow 0 $$
The last map coincides with the derivative map, cf. \cite[\S 6]{ST1}, \cite{Mo2}.
\end{exam}

\section{Irreducibility results}\label{irredresults}

The results in this section are valid under the following assumption on the residue characteristic $p$ of $L$: \vskip8pt

\begin{assumption}\label{hyp} If the root system $\Phi = \Phi(\frg,\frt)$ has irreducible components of type $B$, $C$ or $F_4$, we assume $p > 2$, and if $\Phi$ has irreducible components of type $G_2$, we assume that $p > 3$.
\end{assumption}

In section \ref{HWM} we prove certain results on highest weight modules under the same assumption, and these results are used in this section, which is why this restriction is imposed here. It might be possible that a more refined analysis in section \ref{HWM} shows that one can obtain results which are actually sufficient for our purposes here, without assuming \ref{hyp}. \vskip8pt

\begin{dfn}\label{maximal} Let $M$ be an object of the category $\cO$. We call a standard parabolic subalgebra $\frp$ {\it maximal for $M$} if $M \in \cO^\frp$ and if $M \notin \cO^\frq$ for all parabolic subalgebras $\frq$ strictly containing $\frp$.
\end{dfn}

It follows from \cite[sec. 9.4]{H2} that for every object $M$ of $\cO$ there is unique standard parabolic subalgebra $\frp$ which is maximal for $M$.
The same definition applies for objects in the subcategory $\cO_{\alg}$ in which case we also say that the standard parabolic subgroup $P$ corresponding to $\frp$ is {\it maximal for $M$}. In that case $M$ lies in $\cO^\frp_\alg$, by \ref{T-alg implies L-alg}. \vskip8pt

\begin{thm}\label{irredG_0} Suppose \ref{hyp} holds. Let $M \in \cO_\alg^\frp$ be simple and assume that $P$ is maximal for $M$. Then

\vskip8pt

(i) $D(G_0) \otimes_{D(\frg,P_0)}M$ is simple as a $D(G_0)$-module.

\vskip5pt

(ii) $\cF^G_P(M) = \Ind^G_P(W')^\frd$ is topologically irreducible as a $G_0$-representation.

\vskip5pt

(iii) $\cF^G_P(M) = \Ind^G_P(W')^\frd$ is topologically irreducible as a $G$-representation.
\end{thm}

\noindent \Pf By the remark after \ref{equivalence}, part (i) implies part (ii), and part (ii) implies obviously part (iii). Assertion (i) is a special of Theorem \ref{irredH} below if we take $H = G_0$. \qed

\vskip12pt

\begin{para}\label{generalizations} {\it Generalization to open normal subgroups.} It will be useful later 
(in the proof of \ref{irredgeneral}), to have a similar criterion for the irreducibility of subrepresentations of $\cF^G_P(M)|_H$, 
where $H$ is an arbitrary open normal subgroup of $G_0$. For each $g \in G_0,$ the subspace

$$\left\{f \in \Ind^{G_0}_{P_0}(W')^\frd \midc \supp(f) \sub H g P_0\right\}$$

\vskip8pt

\noindent of $\Ind_{P_0}^{G_0}(W')^\frd$ is closed and stable under the action of $H$, and we therefore have a decomposition of $H$-representations

$$\Ind_{P_0}^{G_0}(W')^\frd|_H \; = \; \bigoplus_{g \in H\bksl G_0/P_0}
\left\{f \in \Ind^{G_0}_{P_0}(W')^\frd \midc \supp(f) \sub H g P_0\right\} \;.$$

\vskip8pt

\noindent (The $H$-representation $\Ind_{P_0}^{G_0}(W')^\frd|_H$ is local in the sense that a function $f$, when restricted to a double coset $HgP_0$ and extended by zero, call this function $(f|_{HgP_0})_!$, still lies in that space. This is obviously the case for the induced representation $\Ind_{P_0}^{G_0}(W')$. The condition to be annihilated by the differential operators $\frz \cdot_r$ with in $\frz \in \frd$ is a condition on the germ of $f$ at any element in $G_0$. Hence the function $(f|_{HgP_0})_!$ belongs to $\Ind_{P_0}^{G_0}(W')^\frd$ as well.)

We define the representation $\Ind^{HP_0}_{P_0}(W')^\frd$ as in \ref{basicdfn}. Extending functions on $HP_0$ by zero gives an isomorphism of $H$-representations

$$\Ind^{HP_0}_{P_0}(W')^\frd \cong \left\{f \in \Ind^{G_0}_{P_0}(W')^\frd \midc \supp(f) \sub H P_0\right\} \;.$$

\vskip8pt

\noindent For $g \in G_0$, we denote by $\Ind^{HP_0}_{P_0}(W')^{\frd,g}$ the space $\Ind^{HP_0}_{P_0}(W')^\frd$ equipped with the $H$-action defined by

$$(h\cdot_g f)(x) = f(g^{-1}h^{-1}g x) \;.$$


\noindent The map

$$\begin{array}{ccc} \Ind^{HP_0}_{P_0}(W')^{\frd,g} & \lra & \left\{f \in \Ind^{G_0}_{P_0}(W')^\frd \midc \supp(f) \sub H g P_0\right\} \;, \\
& &\\
f & \longmapsto & x \mapsto \left\{\begin{array}{ccl} f(g^{-1}x) & , & x \in H g P_0\\
0 & , & \mbox{else}   \end{array} \right. \end{array}$$

\vskip8pt

\noindent is then an isomorphism of $H$-representations. In the theorem below we use the following notation. 
For a $D(H)$-module $N$ and $g \in G_0,$ we denote by $\delta_g \star N$ the space $N$, equipped with the structure of a $D(H)$-module given by $\delta \cdot_g n = (\delta_{g^{-1}} \delta \delta_g)n$, where $n \in N$, $\delta \in D(H)$, and the product $\delta_{g^{-1}} \delta \delta_g$, is an element of $D(G_0)$, and is an element of $D(H)$.
\end{para}

\setcounter{enumi}{0}

\begin{thm}\label{irredH} Suppose \ref{hyp} holds. Let $M \in \cO_\alg^\frp$ be simple and assume that $\frp$ is maximal for $M$. Let $H$ be an open normal subgroup of $G_0$, and let $g, g_1, g_2 \in G_0$. Then

\vskip8pt

(i) The dual space of $\Ind^{HP_0}_{P_0}(W')^{\frd,g}$ is isomorphic to $\delta_g \star \left(D(HP_0) \otimes_{D(\frg,P_0)} M\right)$ as $D(H)$-module.

(ii) The $D(H)$-module $\delta_g \star \left(D(HP_0) \otimes_{D(\frg,P_0)} M\right)$ is simple.

(iii) The $D(H)$-modules $\delta_{g_1} \star \left(D(HP_0) \otimes_{D(\frg,P_0)} M\right)$ and $\delta_{g_2} \star \left(D(HP_0) \otimes_{D(\frg,P_0)} M\right)$ are isomorphic if and only if $g_1 HP_0 = g_2 HP_0$.

(iv) The $H$-representation $\Ind^{HP_0}_{P_0}(W')^{\frd,g}$ is topologically irreducible.

(v) The topological  $H$-representations $\Ind^{HP_0}_{P_0}(W')^{\frd, g_1}$ and $\Ind^{HP_0}_{P_0}(W')^{\frd, g_2}$ are isomorphic if and only if $g_1 HP_0 = g_2 HP_0$.
\end{thm}

\noindent \Pf (i) The proofs of \ref{basiciso} and \ref{Prop_admissible} only make use of the fact that $G_0$
is open (hence its Lie algebra is equal to $\frg$) and that it contains $P_0$, and the corresponding statements carry over for any group with these properties. This shows statement (i) for the case when $g = 1$. From this the general case follows immediately.

(iv) This follows from (ii), by \ref{equivalence}.

(v) This follows from (iii), by \ref{equivalence}.

We will now start with the proofs of (ii) and (iii). The first step is to reduce to the case of a suitable subgroup $H_0 \sub H$ which is normal in $G_0$ and uniform pro-$p$. Having this accomplished, we rename $H_0$ to $H$. The second step is to pass to modules over the Banach algebra $D_r(H)$, cf. \cite[sec. 4]{ST2}. The third step consists in studying the restriction of these modules to the subring $U_r(\frg) = U_r(\frg,H)$, which is the completion of $U(\frg)$ with respect to the norm $\|\cdot\|_r$ on $D_r(H)$.

{\it Step 1: reduction to $H_0$.} The standard parabolic subgroup $\bP \sub \bG$ has a smooth model $\bP_0 \sub \bG_0$, and the unipotent radical $\bU_\bP^-$ of the parabolic subgroup opposite to $\bP$ has a smooth model $\bU_{\bP,0} \sub \bG_0$ as well.

Let $\Lie(\bG_0)$, $\Lie(\bP_0)$ and $\Lie(\bU_{\bP,0})$ be the Lie algebras of these (smooth) group schemes. These are $O_L$-lattices in $\frg = \Lie(G)$, $\frp = \Lie(P)$ and $\fru_P^- = \Lie(U_P^-)$, respectively. Moreover, $\Lie(\bG_0)$, $\Lie(\bP_0)$ and $\Lie(\bU_{\bP,0})$ are $\Zp$-Lie algebras, and we have

$$\Lie(\bG_0) = \Lie(\bU^-_{\bP,0}) \oplus \Lie(\bP_0) \;.$$

\vskip8pt

\noindent For $m_0 \ge 1$ ($m_0 \ge 2$ if $p=2$) the $O_L$-lattices $p^{m_0} \Lie(\bG_0)$, $p^{m_0} \Lie(\bP_0)$ and $p^{m_0} \Lie(\bU_{\bP,0})$ are powerful $\Zp$-Lie algebras, cf. \cite[sec. 9.4]{DDMS}, and hence
$\exp_G: \frg \dashrightarrow G$ converges on these $O_L$-lattices. Therefore,

\begin{numequation}\label{H_0}
H_0 = \exp_G\Big(p^{m_0} \Lie(\bG_0)\Big) \;, \;\; H_0^+ = \exp_G\Big(p^{m_0} \Lie(\bP_0)\Big) \;, \;\; H_0^- = \exp_G\Big(p^{m_0} \Lie(\bU^-_{\bP,0})\Big)
\end{numequation}


\noindent are uniform pro-$p$ groups. The adjoint action of $G_0 = \bG_0(O_L)$ leaves $\Lie(\bG_0)$ invariant, and hence $p^{m_0} \Lie(\bG_0)$. This shows that $H_0$ is normal in $G_0$, and analogous arguments show that $H_0^+$ and $H_0^-$ are normal in $P_0$ and $U_{P,0}^-$, respectively. It follows from the existence of 'coordinates of the second kind' that $H_0 = H_0^-H_0^+$ and $H_0 \cap P_0 = H_0^+$ and $H_0 \cap U^-_{P,0} = H_0^-$, cf. \cite[2.2.4 (ii)]{OS}.

Moreover, for large enough $m_0$ the group $H_0$ is contained $H$, and $D(H) = \bigoplus_{h \in H/H_0} \delta_h D(H_0)$. This implies that, as $D(H_0)$-modules,

\begin{numequation}\label{reduction1}
\delta_g \star \left(D(HP_0) \otimes_{D(\frg,P_0)} M\right) = \bigoplus_{h \in HP_0/H_0P_0} \delta_{gh} \star \left(D(H_0P_0) \otimes_{D(\frg,P_0)} M\right) \;.
\end{numequation}


\noindent Let us now assume that statements (ii) and (iii) hold for $H_0$. Then the direct sum on the right of \ref{reduction1} consists of simple $D(H_0)$-modules which are pairwise non-isomorphic. As they are permuted by the action of $H$, it follows that the left side is simple as $D(H)$-module.
Furthermore, if $\delta_{g_1} \star \left(D(HP_0) \otimes_{D(\frg,P_0)} M\right)$ is isomorphic to $\delta_{g_2} \star \left(D(HP_0) \otimes_{D(\frg,P_0)} M\right)$ as $D(H)$-module, then also as $D(H_0)$-module. But then we must have $g_1 h_1 H_0P_0 = g_2 h_2 H_0 P_0$ for some $h_1, h_2 \in H$, in particular $g_1 H P_0 = g_2 H P_0$.

Thus we have shown that statements (ii) and (iii) for $D(H)$ follow from the corresponding statements for $D(H_0)$.

To simplify notation, we will from now on assume that $H$, and its subgroups $H^+$ and $H^-$, are defined as in \ref{H_0}, and are in particular uniform pro-$p$ groups.

{\it Step 2: passage to $D_r(H)$.} We put

\vspace{-0.3cm}
$$\kappa = \left\{\begin{array}{lcl} 1 & , & p>2 \\
2 & , & p=2 \end{array} \right.$$


\noindent (We point out that in \cite{OS}, which we are going to use in the following, $\kappa$ is denoted by $\vep_p$.) Let $r$ always denote a real number in $(0,1) \cap p^\bbQ$ with the property:
\begin{numequation}\label{r_and_s}
\hskip-20pt \mbox{there is $m \in \bbZ_{\ge 0}$ such that $s = r^{p^m}$ satisfies $\frac{1}{p} < s$ and $s^{\kappa} < p^{-1/(p-1)}$ .}
\end{numequation}

\vspace{-0.3cm}
\noindent For the definition of the canonical $p$-valuation on a uniform pro-$p$ group we refer to \cite[2.2.3]{OS}.
We let $\|\cdot\|_r$ denote the norm on $D(H)$ associated to $r$ and the canonical $p$-valuation, cf. \cite[2.2.6]{OS} (where this norm is denoted by $\bar{q}_r$). $D_r(H)$ denotes the corresponding Banach space completion. This is a noetherian Banach algebra, and $D(H) = \varprojlim_{r <1} D_r(H)$ (cf. \cite{ST2} for more information). If $\widetilde{H} \sub G_0$ is a compact open subgroup which contains $H$, then
\begin{numequation}\label{dcomp}
D(\widetilde{H}) = \bigoplus_{g \in \widetilde{H}/H} \delta_g D(H)
\end{numequation}


\noindent and we let $\|\cdot\|_r$ be the maximum norm on $D(\widetilde{H})$ given by

\begin{numequation}\label{maxnorm}
\|\sum_{g \in \widetilde{H}/H} \delta_g \lambda_g \|_r = \max\{\|\lambda_g \|_r \midc g \in \widetilde{H}/H \} \;.
\end{numequation}


\noindent Then the completion $D_r(\widetilde{H})$ of $D(\widetilde{H})$ with respect to $\|\cdot\|_r$ has an analogous decomposition:

$$D_r(\widetilde{H}) = \bigoplus_{g \in \widetilde{H}/H} \delta_g D_r(H)$$


\noindent and therefore $D_r(\widetilde{H}) =  D_r(H) \otimes_{D(H)} D(\widetilde{H})$. We are going to use the preceding discussion in the case when $\widetilde{H} = HP_0$. Put $\bM = D(HP_0) \otimes_{D(\frg,P_0)} M$. To show that $\bM$ is a simple $D(H)$-module it suffices, by \cite[3.9]{ST2}, to show that

\vspace{-0.3cm}
$$\bM_r := D_r(H) \otimes_{D(H)} \bM = D_r(HP_0) \otimes_{D(\frg,P_0)} M$$


\noindent is a simple $D_r(H)$-module for a sequence of $r$'s converging to $1$. By \ref{basiciso} and \ref{nonzero}

\vspace{-0.3cm}
$$\cF^{G_0}_{P_0}(M)' = D(G_0) \otimes_{D(HP_0)} \bM$$


\noindent is non-zero, and thus $\bM = \varprojlim_{r<1} \bM_r$ is non-zero. Hence we must have $\bM_r \neq 0$ for $r$ sufficiently close to $1$. As the image of $M$ in $\bM_r$ generates $\bM_r$ as $D_r(HP_0)$-module, and because $M$ is simple, the canonical map $M \ra \bM_r$ must be injective when $\bM_r \neq 0$. From now on we assume that, in addition to \ref{r_and_s}, $r$ is such that $\bM_r \neq 0$, and we consider $M$ as being contained in $\bM_r$.

{\it Step 3: passage to $U_r(\frg)$.} Let $U_r(\frg) = U_r(\frg,H)$ be the topological closure of $U(\frg)$ in $D_r(H)$. It is important for our approach to have a useful description of $U_r(\frg)$. We will freely use the following fact, which follows from \cite[5.6]{Sch}:
\vskip-12pt

\begin{numequation}\label{density1}
\hskip-50pt \mbox{$U(\frg)$ is dense in $D_r(H)$ if $r^{\kappa} < p^{-\frac{1}{p-1}}$. In this case $U_r(\frg) = D_r(H)$.}
\end{numequation}

\vskip-12pt

\noindent Let $(P_m(H))_{m \ge 1}$ be the lower $p$-series of $H$, cf. \cite[1.15]{DDMS}. Note that $P_1(H) = H$. For $m \ge 0$ put $H^m : = P_{m+1}(H)$ so that $H^0 = H$. We refer to \cite[sec. 4.5, sec. 9.4]{DDMS} for the notion of a $\Zp$-Lie algebra of a uniform pro-$p$ group. It follows from \cite[3.6]{DDMS} that $H^m$ is a uniform pro-$p$ group with $\Zp$-Lie algebra equal to $p^m\Lie_\Zp(H)$. Let $s = r^{p^m}$ be as in \ref{r_and_s}. Denote by $\|\cdot\|_s^{(m)}$ the norm on $D(H^m)$ associated to $s$ and the canonical $p$-valuation on $H^m$. Then, by \cite[6.2, 6.4]{Sch}, the restriction of $\|\cdot\|_r$ on $D(H)$ to $D(H^m)$ is equivalent to $\|\cdot\|_s^{(m)}$, and $D_r(H)$ is a finite and free $D_s(H^m)$-module a basis of which is given by any set of coset representatives for $H/H^m$. By \ref{density1} we can conclude \vskip-12pt

\begin{numequation}\label{density2}
\mbox{If $s = r^{p^m}$ is as in \ref{r_and_s}, then $U(\frg)$ is dense in $D_s(H^m)$, hence $U_r(\frg,H) = D_s(H^m)$.}
\end{numequation}

\vskip-12pt

\noindent In particular, $U_r(\frg) \cap H = H^m$ is an open subgroup of $H$. Let

\vspace{-0.3cm}
$$\frm_r := U_r(\frg)M$$


\noindent be the $U_r(\frg)$-submodule of $\bM_r$ generated by $M$. Because we assume $\bM_r \neq 0$ we also have $\frm_r \neq 0$, and $M$ is contained in $\frm_r$.

Likewise, we equip $D(H^+)$ ($D(H^-)$, resp.) with the norm $\|\cdot\|_r$ associated to the canonical $p$-valuation on $H^+$ ($H^-$, resp.), and give $D(P_0)$ ($D(U^-_{P,0})$, resp.) the maximum norm as above (cf. \ref{dcomp} and \ref{maxnorm}). Again, we denote by $D_r(P_0)$ ($D_r(U^-_{P,0})$, resp.) the corresponding completions. As is easy to see, these norms on $D(P_0)$ ($D(U^-_{P,0})$, resp.) are equal to the restriction of the norms on $D(HP_0)$ ($D(U^-_{P,0}H)$, resp.) coming from the norm $\|\cdot\|_r$ on $D(H)$ by the recipe explained above.

Let $U_r(\frp)$ be the closure of $U(\frp)$ in $D_r(H^+)$. Then $U_r(\frp)$ is the completion of $U(\frp)$ with respect to the norm associated to the canonical $p$-valuation on $H^+$. Because the canonical $p$-valuation of an element of $h \in H$ ($h \in H^+$, resp.) can be read off from $\log_G(h) \in p^{m_0}\Lie(\bG_0)$ ($\log_G(h) \in p^{m_0}\Lie(\bP_0)$, resp.), the canonical $p$-valuation on $H^+$ is the restriction of the canonical $p$-valuation on $H$. Hence $U_r(\frp)$ is also the closure of $U(\frp)$ in $U_r(\frg)$.

It follows from \ref{density2} (applied to $H^+$ and $\frp = \Lie_\Zp(H^+)$) that $D_r(P_0)$ is generated as a module over $U_r(\frp)$ by finitely many Dirac distributions $\delta_{g_1}, \ldots, \delta_{g_k}$, with $g_i \in P_0$. Because $P_0$ acts on $M$, the $U_r(\frg)$-module $\frm_r$ is actually a module over the subring

\begin{numequation}\label{D_r_frg_P_0}
D_r(\frg,P_0) = U_r(\frg)D_r(P_0) = \sum_{i=1}^k U_r(\frg) \cdot \delta_{g_i}
\end{numequation}

\noindent generated by $U_r(\frg)$ and $D_r(P_0)$ inside $D_r(G_0)$. Moreover, $\frm_r$ is a finitely generated $U_r(\frg)$-module (in fact, generated by a single vector of highest weight), hence carries a canonical $U_r(\frg)$-module topology, as $U_r(\frg)$ is a noetherian Banach algebra (cf. \cite[1.4.2]{K1} which applies here because we assume $r \in (\frac{1}{p},1) \cap p^\bbQ$).

The module $M$ is clearly dense in $\frm_r$ with respect to this topology. It follows from \cite[1.3.12]{F}
or \cite[3.4.8]{OS} that $\frm_r$ is a simple $U_r(\frg)$-module, and in particular a simple $D_r(\frg,P_0)$-module.

\setcounter{enumi}{0}

\begin{sublemma}\label{P_{0,r}} Let $r$ and $s$ be as in \ref{r_and_s}. Put $P_{0,r} = HP_0 \cap D_r(\frg,P_0)$. Then

(i) $P_{0,r} = H^mP_0 = H^{-,m}P_0$, where $H^{-,m} = P_m(H^-)$.

(ii) $D_r(HP_0) =  \bigoplus_{g \in HP_0/P_{0,r}} \delta_g D_r(\frg,P_0)$.
\end{sublemma}

\noindent \Pf By \ref{density2} (and the line following) we have $U_r(\frg) \cap H = H^m$, and thus $P_{0,r} \supset H^mP_0$. 
On the other hand, we have

\vspace{-0.3cm}
\begin{numequation}\label{D_r_H_P_0_I}
D_r(HP_0) = \bigoplus_{g \in HP_0/H} \delta_g D_r(H)
\end{numequation}
\noindent and

\vspace{-0.3cm}
\begin{numequation}\label{decompU_r}
D_r(H) = \bigoplus_{h \in H/H^m} \delta_h U_r(\frg) \;.
\end{numequation}

\noindent \ref{D_r_H_P_0_I} and \ref{decompU_r} taken together give

\vspace{-0.3cm}
\begin{numequation}\label{D_r_H_P_0_II}
D_r(HP_0) = \bigoplus_{g \in HP_0/H^m} \delta_g U_r(\frg)  = \bigoplus_{h \in H^-/H^{-,m}} \bigoplus_{g \in P_0/H^{+,m}} \delta_h \delta_g U_r(\frg) \;.
\end{numequation}

\noindent \ref{D_r_H_P_0_II} and \ref{D_r_frg_P_0} taken together give

\vspace{-0.3cm}
\begin{numequation}\label{decompU_r2}
D_r(\frg,P_0) = \bigoplus_{g \in P_0/H^{+,m}} \delta_g U_r(\frg) \;.
\end{numequation}

\noindent If now $h \in H$ is such that $\delta_h \in D_r(\frg,P_0)$, then \ref{decompU_r} and \ref{decompU_r2} together show that $\delta_h = \delta_g \delta$ with $g \in P_0 \cap H$ and $\delta \in U_r(\frg)$. Therefore $\delta  = \delta_{g^{-1}h} \in H \cap U_r(\frg) = H^m$, and so $h \in H^mP_0$. This shows (i). Because the inclusion $H^- \sub H$ induces a bijection $H^-/H^{-,m} \ra (HP_0)/(H^mP_0)$, assertion (ii) now follows from (i) and \ref{D_r_H_P_0_II} and \ref{decompU_r2}. \qed

\vskip8pt

By \ref{P_{0,r}} (ii) we have  $\frm_r = D_r(\frg,P_0) \otimes_{D(\frg,P_0)} M \sub D_r(HP_0) \otimes_{D(\frg,P_0)} M = \bM_r$. Using \ref{P_{0,r}} (ii) we can conclude

\vspace{-0.3cm}
\begin{numequation}\label{decomp}
\bM_r = D_r(HP_0) \otimes_{D(\frg,P_0)} M = D_r(HP_0) \otimes_{D_r(\frg,P_0)} \frm_r = \bigoplus_{g \in HP_0/P_{0,r}} \delta_g \frm_r \,.
\end{numequation}

\noindent Note that, as $U_r(\frg)$-modules, the submodule $\delta_g \frm_r$ on the right hand side is the same as $\delta_g \star \frm_r$, where we use the notation as introduced before \ref{irredH} (for $U_r(\frg)$ instead of $D(H)$). By Theorem \ref{U_r_modules} below the modules $\delta_g \star \frm_r$ are all simple, and there is no isomorphism of $U_r(\frg)$-modules $\delta_{g_1} \star \frm_r \cong \delta_{g_2} \star \frm_r$ if $g_1P_{0,r} \neq g_2P_{0,r}$. This shows that $\bM_r$ is a simple $D_r(H)$-module.

Now assume that $\delta_{g_1} \star \bM$ and $\delta_{g_2} \star \bM$ are isomorphic as $D(H)$-modules. Then, after tensoring with $D_r(H)$ we find that $\delta_{g_1} \star \bM_r$ and $\delta_{g_2} \star \bM_r$ are isomorphic as $D_r(H)$-modules, hence as $U_r(\frg)$-modules. Then \ref{decomp}, together with \ref{U_r_modules}, implies that $g_1h_1 P_{0,r} = g_2 h_2 P_{0,r}$ for some $h_1, h_2 \in H$, and hence $g_1HP_0 = g_2HP_0$.

We have therefore reduced statements (ii) and (iii) of \ref{irredH} to the assertions of the theorem below. \qed

\setcounter{enumi}{0}

\begin{thm}\label{U_r_modules} Assume \ref{hyp} holds. Let $M \in \cO^\frp_\alg$ be as in \ref{irredH}. Let $H = \exp_G\left(p^{m_0}\Lie(\bG_0)\right)$ be uniform pro-$p$. Assume that $r \in (\frac{1}{p},1) \cap p^{\bbQ}$ and $s$ are as in \ref{r_and_s}. Assume that

\vspace{-0.3cm}
$$\frm_r = D_r(\frg,P_0) \otimes_{D(\frg,P_0)} M \neq 0 \;.$$


\noindent (This will be the case if $r$ is close enough to 1.) Then, for every $g \in G_0$ the $U_r(\frg)$-module $\delta_g \star \frm_r$ is simple. For any $g_1, g_2 \in G_0$ with $g_1P_{0,r} \neq g_2P_{0,r}$ the $U_r(\frg)$-modules $\delta_{g_1} \star \frm_r$ and $\delta_{g_2} \star \frm_r$ are not isomorphic.
\end{thm}

\noindent \Pf We continue to use the notation introduced in the proof of \ref{irredH}. We have already seen that $\frm_r$ is simple as $U_r(\frg)$-module, and this implies that $\delta_g \star \frm_r$ is simple as well.

It is trivial to check that for $x, y \in G_0$ one has $\delta_x \star (\delta_y \star \frm_r) = \delta_{xy} \star \frm_r$, and if $x \in P_{0,r}$, then $\delta_x \star \frm_r \ra \frm_r$, $n \mapsto \delta_x \cdot n$, is an isomorphism of $U_r(\frg)$-modules.

Now let $\phi: \delta_{g_1} \star \frm_r \stackrel{\simeq}{\lra} \delta_{g_2} \star \frm_r$ be an isomorphism of $U_r(\frg)$-modules. A straightforward calculation shows that $\phi$ is also a $U_r(\frg)$-module isomorphism $\delta_{g_2^{-1}g_1} \star \frm_r \stackrel{\simeq}{\lra} \frm_r$, so that we may assume $g_2 = 1$. Let $I_0 \sub G_0$ be the Iwahori subgroup whose image in $\bG(O_L/(\pi_L))$ is $\bB(O_L/(\pi_L))$.\footnote{We denote the Iwahori subgroup by $I_0$ because we use $I$ for the subset of simple roots corresponding to $\frp$.} Using the Bruhat decomposition

\vspace{-0.3cm}
$$G_0 = \coprod_{w \in W/W_P} I_0wP_0$$


\noindent we may write $g = g_1= h^{-1}wp_1$ with $h \in I_0$, $w \in W$ and $p_1 \in P_0$.
By the Iwahori decomposition

\vspace{-0.3cm}
$$I_0 = (I_0 \cap {U_{P,0}^-}) \cdot (I_0 \cap P_0)$$


\noindent we can write $h = up_2$ with $p_2 \in I_0 \cap P_0$ and $u \in I_0 \cap U_{P,0}^-$. The same reasoning as before then shows that an isomorphism $\delta_g \star \frm_r \stackrel{\simeq}{\lra} \frm_r$ induces an isomorphism of $U_r(\frg)$-modules which we again denote by $\phi$:

\vspace{-0.3cm}
$$\phi: \delta_w \star \frm_r \stackrel{\simeq}{\lra} \delta_u \star \frm_r \,.$$


\noindent {\it Step 1.} We show first that this can only happen if $w \in W_P$.  Let $\lambda \in X({\bf T})^\ast$ be the highest weight of $M$, i.e. $M = L(\lambda)$, and $I = \{\alpha \in \Delta \midc \langle \lambda, \alpha^\vee \rangle \in \bbZ_{\ge 0} \}$, cf. Example \ref{Example_Verma}.

\vskip8pt

Then by Cor. \ref{locfinitesubalgebra} and the maximality condition with respect to $\frp$, the parabolic subalgebra $\frp$ is $\frp_I$, where the root system of the Levi subalgebra of $\frp_I$ has $I$ as a basis of simple roots. Suppose $w$ is not contained in $W_I = W_P$. Then there is a positive root  $\beta \in \Phi^+\setminus\Phi^+_I$ such that $w^{-1}\beta < 0$, hence $w^{-1}(-\beta) > 0$, cf. \cite[2.3]{Car}. Consider a non-zero element element $y \in \frg_{-\beta}$, and let $v^+ \in M$ be a weight vector of weight $\lambda$ (uniquely determined up to a non-zero scalar). Then we have the following identity in $\delta_w \star \frm_r$,

\vspace{-0.3cm}
$$y \cdot_w v^+ = \Ad(w^{-1})(y) \cdot v^+ = 0$$


\noindent as $\Ad(w^{-1})(y) \in \frg_{-w^{-1}\beta}$ annihilates $v^+$. We have $\phi(v^+) = v$ for some nonzero $v \in \frm_r$. And therefore

\vspace{-0.3cm}
$$0 = \phi(y \cdot_w v^+) = y \cdot_u \phi(v^+) = y \cdot_u v = \Ad(u^{-1})(y) \cdot v \;.$$


\noindent On the other hand, $y' := \Ad(u^{-1})(y)$ is an element of $\fru_P^-$ and as such acts injectively on $M$, cf. Cor. \ref{injective}. We show that $y'$ actually acts injectively on $\frm_r$, too. Indeed, let $v\in \frm_r$ with $y'\cdot v=0.$ Write $v$ as a convergent sum $v = \sum_{\mu \in \Lambda(v)} v_{\lambda-\mu}$ where $\Lambda(v)$ is a (non-empty) subset of $\bbZ_{\ge 0} \alpha_1 \oplus \ldots \oplus \bbZ_{\ge 0} \alpha_\ell$ and $v_{\lambda-\mu} \in M_{\lambda-\mu} \setminus \{0\}$ is a vector of weight $\lambda - \mu$ (cf. the fourth assertion in \cite[1.3.12]{F}). Here $\Delta = \{\alpha_1, \ldots, \alpha_\ell\}$ is the set of simple roots. Write $y' = \sum_{\gamma \in B} y_\gamma$, where $B$ is a non-empty subset of $\Phi^+\setminus \Phi^+_I$ and $y_\gamma$ is a non-zero element of $\frg_{-\gamma} \subset \fru_P^-$. Then we have

\vspace{-0.3cm}
$$0 = y' \cdot v = \sum_{\nu \in \bbZ_{\ge 0} \alpha_1 \oplus \ldots \oplus \bbZ_{\ge 0} \alpha_\ell} \left(\sum_{\mu \in \Lambda(v), \gamma \in B, \nu = \mu+\gamma} y_\gamma \cdot v_{\lambda-\mu}\right) \;,$$


\noindent where

$$\sum_{\mu \in \Lambda(v), \gamma \in B, \nu = \mu+\gamma} y_\gamma \cdot v_{\lambda-\mu}$$


\noindent lies in $M_{\lambda-\nu}$. Because of the uniqueness of the expansions of elements of $\frm_r$ as convergent series of weight vectors, cf. \cite[1.3.12]{F}, we can conclude that $\sum_{\mu \in \Lambda(v), \gamma \in B, \nu = \mu+\gamma} y_\gamma \cdot v_{\lambda-\mu}$ vanishes.

Define on $\bbQ \alpha_1 \oplus \ldots \oplus \bbQ \alpha_\ell$ the lexicographic ordering as in the proof of Prop. \ref{notlocnilp}. Choose $\gamma^+ \in B$ and $\mu^+ \in \Lambda(v)$, both minimal with respect to this ordering. With $\nu^+ = \gamma^+ + \mu^+$ we then have

\vspace{-0.3cm}
$$0 = \sum_{\mu \in \Lambda(v), \gamma \in B, \nu^+ = \mu+\gamma} y_\gamma \cdot v_{\lambda-\mu} = y_{\gamma^+} \cdot v_{\lambda-\mu^+} \;.$$

\noindent This contradicts Cor. \ref{injective}, and we can thus conclude that $w$ must be an element of $W_P$.

\medskip
From now on we may even assume that $w=1$ since $W_P \subset P_0 \sub P_{0,r}$.

\vskip8pt

\noindent {\it Step 2.} Now let $u \in U^-_{P,0}$ and let $\phi: \frm_r \stackrel{\simeq}{\lra} \delta_u \star \frm_r$ be an isomorphism of $U_r(\frg)$-modules. We suppose that $u$ is not contained in $P_{0,r}$ and are seeking to derive a contradiction. Let $v^+ \in M_\lambda \setminus \{0\}$ be a vector of highest weight as above. Put $\phi(v^+) = v$ with some non-zero $v \in \frm_r$. Then we have for any $x \in \frt$, the following identity in $\delta_u \star \frm_r$:

\vspace{-0.3cm}
$$\lambda(x)v = \phi(x \cdot v^+) = x \cdot_u \phi(v^+) = \Ad(u^{-1})(x) \cdot v \;.$$


\noindent We have thus for all $x \in \frt$, the following identity in $\frm_r$

\vspace{-0.3cm}
\begin{numequation}\label{identity}
\lambda(x)v = \Ad(u^{-1})(x) \cdot v \hskip5pt .
\end{numequation}

\vspace{-0.3cm}
\noindent Note that this equation only involves the action of elements of $\fru_P^- \oplus \frt$, because $\Ad(u^{-1})(x)$ is in $\fru_P^- \oplus \frt$. Next we embed $\frm_r$ into the ''formal completion'' of $M$, i.e.,

\vspace{-0.3cm}
$$\hat{M} = \prod_{\mu} M_\mu$$


\noindent by mapping the weight spaces $M_\mu \sub \frm_r$ to the corresponding space $M_\mu \sub \hat{M}$ (cf. the fourth assertion in \cite[1.3.12]{F}).
Then $\hat{M}$ is in an obvious way a module for $U(\fru_P^- \oplus \frt)$. This module structure extends to a representation of $U^-_P $ as follows. Since $U_P^-$ is a unipotent group, the exponential $\exp_{U_P^-}: \fru_P^- \dashrightarrow U_P^-$ is actually defined on the whole Lie algebra, so that we have a bijective map $\exp_{U_P^-}: \fru_P^- \ra U_P^-$. Let $\log_{U_P^-}: U_P^- \ra \fru_P^-$ be the inverse map. Then we can define for $u \in U_P^-$ and $v = \sum_\mu v_\mu \in \hat{M}$:

$$\delta_u \cdot v = \sum_\mu \sum_{n \ge 0} \frac{1}{n!} \log_{U_P^-}(u)^n \cdot v_\mu \,.$$


\noindent Note that this sum is well-defined in $\hat{M}$, because $\log_{U_P^-}(u)$ is in $\fru_P^-$, and there are hence only finitely many terms of given weight in this sum. (We continue to write the action of an element $u \in U_P^-$ on $\hat{M}$ by $\delta_u$.) Moreover, it gives an action of $U_P^-$ on $\hat{M}$ because

$$\begin{array}{rcl}
\exp(\log_{U_P^-}(u_1)) \cdot \exp(\log_{U_P^-}(u_2)) &
= & \exp\big(\cH(\log_{U_P^-}(u_1),\log_{U_P^-}(u_2))\big) \\
\\
& = & \exp(\log_{U_P^-}(u_1u_2)) \,,
\end{array}$$

\vskip5pt

\noindent where $\cH(X,Y) = \log(\exp(X)\exp(Y))$ is the Baker-Campbell-Hausdorff series (which converges on $\fru_P^-$, as $\fru_P^-$ is nilpotent). It is then immediate that this action is compatible with the action of $U(\fru_P^- \oplus \frt)$. The identity \ref{identity} implies then the following identity in $\hat{M}$

\vspace{-0.3cm}
$$\delta_{u^{-1}} \cdot( x \cdot (\delta_u \cdot v)) = \Ad(u^{-1})(x) \cdot v = \lambda(x)v \hskip5pt ,$$


\noindent for all $x \in \frt$, and thus, multiplying both sides of the previous equation with $\delta_u$,

\vspace{-0.3cm}
$$x \cdot (\delta_u \cdot v) = \lambda(x) \delta_u \cdot v \hskip5pt .$$

\noindent Hence $\delta_u \cdot v \in \hat{M}$ is a weight vector of weight $\lambda$ and must therefore be equal to a non-zero scalar multiple of $v^+$. After scaling $v$ appropriately we have $\delta_u \cdot v = v^+$ or
$v = \delta_{u^{-1}} \cdot v^+$ with

\begin{numequation}\label{series}
\delta_{u^{-1}} \cdot v^+ = \sum_{n \ge 0} \frac{1}{n!} (-\log_{U_P^-}(u))^n \cdot v^+ =: \Sigma \hskip5pt .
\end{numequation}

\noindent Our goal is to show that the series $\Sigma$, which is an element of $\hat{M}$, does in fact not lie in the image of $\frm_r$ in $\hat{M}$, if $u \notin P_{0,r}$, thus achieving a contradiction.

{\it Step 3.} Let $\frg_\bbZ$ be a $\bbZ$-form of $\frg$, i.e. $\frg_\bbZ \otimes_\bbZ L = \frg$. We fix a Chevalley basis $(x_\gamma,y_\gamma,h_\alpha \midc \gamma \in \Phi^+, \alpha \in \Delta)$ of $[\frg_\bbZ,\frg_\bbZ]$. We have $x_\gamma \in \frg_\gamma$, $y_\gamma \in \frg_{-\gamma}$, and $h_\alpha = [x_\alpha,y_\alpha] \in \frt$, for $\alpha \in \Delta$. Then

$$\Lie_\Zp(H^-) = p^{m_0}\Lie(\bU^-_{\bP,0}) = \bigoplus_{\beta \in \Phi^+ \setminus \Phi^+_I} O_L y_\beta^{(0)} \;,$$


\noindent where $y_\beta^{(0)} = p^{m_0}y_\beta$. Moreover, the $\Zp$-Lie algebra of $H^{-,m}$ is

\vspace{-0.3cm}
$$\Lie_\Zp(H^{-,m}) = p^m\Lie(H^-) = \bigoplus_{\beta \in \Phi^+ \setminus \Phi^+_I} O_L y_\beta^{(m)} \,,$$


\noindent where $y_\beta^{(m)} = p^{m_0+m}y_\beta$.
We assume that $r$ and $s$ are as in \ref{r_and_s}. Then \ref{density2} shows that $U_r(\fru_P^-) = U_r(\fru_P^-,H^-) = D_s(H^{-,m})$. Elements in $U_r(\fru_P^-)$ thus have a description as power series in $(y^{(m)}_\beta)_{\beta \in \Phi^+ \setminus \Phi^+_I}$:

\vspace{-0.3cm}
$$U_r(\fru_P^-) = \left\{\sum_{n = (n_\beta)} d_n  (y^{(m)})^n \midc \lim_{|n| \ra \infty} |d_n|s^{\kappa|n|} = 0 \right\} \,,$$


\noindent where $(y^{(m)})^n$ is the product of the $(y^{(m)}_\beta)^{n_\beta}$ over all $\beta \in \Phi^+ \setminus \Phi^+_I$, taken in some fixed order.
Let $\|\cdot\|_s^{(m)}$ be the norm on $D_s(H^{-,m})$ induced by the canonical $p$-valuation on $H^{-,m}$. Then we have for any generator $y^{(m)}_\beta$

\vspace{-0.3cm}
\begin{numequation}\label{norm_generator}
\| y^{(m)}_\beta\|_s^{(m)} = s^\kappa \;,
\end{numequation}

\noindent as follows from \cite[5.3, 5.8]{Sch}. We recall that by \ref{P_{0,r}} we have $P_{0,r} = H^{-,m}P_0$. Write

$$\log_{U_P^-}(u) = \sum_{\beta \in B(u)} z_\beta \;,$$

\noindent with a non-empty set $B(u) \sub \Phi^+ \setminus \Phi^+_I$ and non-zero elements $z_\beta \in \frg_{-\beta}$. Put

$$B^+(u) = \{\beta \in B(u) \midc z_\beta \notin O_L y^{(m)}_\beta \} \;.$$

\noindent This is a non-empty subset of $B(u)$ since $u \notin P_{0,r}$. Put $B'(u) = B(u) \setminus B^+(u)$,

$$z^+ = \sum_{\beta \in B^+(u)} z_\beta \hskip10pt \mbox{ and } \hskip10pt z' = \sum_{\beta \in B'(u)} z_\beta = \log_{U_P^-}(u) - z^+ \;.$$

\noindent Then $z' \in \Lie_\Zp(H^{-,m})$, and thus $\exp(z') \in H^{-,m}$. The element $u_1 = u\exp(z')^{-1} = u\exp(-z')$ does not lie in $H^{-,m}$, and $\delta_u \star \frm_r \cong \delta_{u_1} \star \frm_r$. We may hence replace $u$ by $u_1 = u \exp(-z')$. Then we compute $\log_{U^-}(u_1) = \log_{U^-}(u \exp(-z'))$ by means of the Baker-Campbell-Hausdorff formula. Because of the commutators appearing in this formula we have

\begin{numequation}\label{log}
\log_{U^-}(u_1) \in z^+ + \sum \left( \, \mbox{iterated commutators of } \frg_{-B^+(u)} \mbox{ with } \frg_{-B'(u)} \, \right) \hskip5pt ,
\end{numequation}

\noindent where $\frg_{-B^+(u)} = \bigoplus_{\beta \in B^+(u)} \frg_{-\beta}$ and analogous for $\frg_{-B'(u)}$.
Recall that the {\it height} $ht(\beta)$ of the root $\beta = \sum_{\alpha \in \Delta}n_\alpha \alpha$ is defined to be the sum $\sum_{\alpha \in \Delta}n_{\alpha}$.
Put $ht'(u) = \min\{ht(\beta) \midc \beta \in B'(u)\}$.  It follows from the preceding formula \ref{log} that
if the right summand does not vanish we have  $ht'(u_1) > ht'(u)$.

The process of passing from $u$ to $u_1$ can be repeated finitely many times, but will finally produce an element $u_N \in U^-_{P,0} \setminus H^{-,m}$ which has the property that $B'(u_N) = \emptyset$ (and hence $u_{N+1} = u_N$). We will denote this element again by $u$.

\vskip8pt

Next we chose an extremal element $\beta^+$ among the $\beta \in B(u) = B^+(u)$, i.e. one of the minimal generators of the cone $\sum_{\beta \in B(u)} \bbR_{\ge 0}\beta$
inside $\bigoplus_{\alpha \in \Delta} \bbR \alpha$. Then we have

$$\bbR_{>0}\beta^+ \cap \sum_{\beta \in B(u), \beta \neq \beta^+} \bbR_{\ge 0}\beta \hskip5pt = \hskip5pt \emptyset \,.$$

\noindent This means in particular that no positive multiple of $\beta^+$ can be written as a linear combination \linebreak $\sum_{\beta \in B(u), \beta \neq \beta^+} c_\beta \beta$ with non-negative integers $c_\beta \in \bbZ_{\ge 0}$. It follows that if $n$ is a positive integer and

\begin{numequation}\label{implication}
\begin{array}{l}
n\beta^+ = \gamma_1 + \ldots + \gamma_m \hskip5pt \mbox{ with not necessarily distinct } \hskip5pt \gamma_i \in B(u) \\
\\
\Rightarrow \hskip8pt \left[m=n \hskip5pt \mbox{ and } \hskip5pt \gamma_1 = \ldots = \gamma_n = \beta^+ \right] \hskip3pt .
\end{array}
\end{numequation}

\vskip8pt

After these intermediate considerations we return to the series

$$\Sigma = \sum_{n \ge 0} \frac{1}{n!} (-\log_{U_P^-}(u))^n \cdot v^+ = \sum_{n \ge 0} \frac{(-1)^n}{n!} (z_{\beta^+} + z_{\beta_2} + \ldots + z_{\beta_k})^n \cdot v^+$$

\noindent where $B(u) = \{\beta^+, \beta_2, \ldots, \beta_k \}$. It follows from \ref{implication} and Cor. \ref{injective}  that if we write $\Sigma$ as a (formal) sum of weight vectors, there is for any $n \in \bbZ_{\ge 0}$ a unique non-zero weight vector in this series which is of weight $\lambda - n \beta^+$, and this element is $\frac{(-1)^n}{n!}z_{\beta^+}^n \cdot v^+$.

\vskip8pt

{\it Step 4.} The last part of the proof is to show that the formal sum of weight vectors

\vspace{-0.3cm}
\begin{numequation}\label{expseries}
\sum_{n \ge 0} \frac{(-1)^n}{n!}z_{\beta^+}^n \cdot v^+
\end{numequation}

\vspace{-0.3cm}
\noindent cannot be the corresponding sum of weight components of an element of $\frm_r$, when considered as an element of $\hat{M}$ and written as a sum of weight vectors.

Write $\Phi^+ = \{\beta_1, \ldots, \beta_t\}$ such that $\Phi_I^+ = \{\beta_{\tau+1}, \ldots, \beta_t\}$. (Note: the roots $\beta_i$ are not necessarily the same as in Step 3.) Choose finitely many weight vectors $v_j \in M$ of weight $\lambda - \mu_j$, $j=1,\ldots,k$, generating $M$ as $U(\fru^-_P)$-module. We can and will assume that $v_j$ is of the form $\left(\prod_{\eta=\tau+1}^t y_{\beta_k}^{\nu_{j,\eta}}\right) \cdot v^+$.

Put $y_i = y_{\beta_i}$ and $y^{(m)}_i = y^{(m)}_{\beta_i}$ for $i=1, \ldots, \tau$. To show that \ref{expseries} does not converge to an element in $\frm_r$, we write

\vspace{-0.3cm}
\begin{numequation}\label{basicrelation}
\frac{y_{\beta^+}^n}{n!}.v^+ = \sum_{1 \le j \le k} \sum_{\nu \in \cI_{n,j}} c_{\nu,j}^{(n)} y_1^{\nu_1} \cdot \ldots \cdot y_\tau^{\nu_\tau}.v_j \,,
\end{numequation}

\noindent where $\cI_{n,j} \sub \bbZ_{\ge 0}^\tau$ consists of those $\nu = (\nu_1,\ldots,\nu_\tau)$ such that $\mu_j + \nu_1\beta_1 + \ldots + \nu_\tau\beta_\tau = n\beta^+$. It follows from the Prop. \ref{both_conditions}
that there is at least one index $(\nu,j) = (\nu_1, \ldots, \nu_\tau, j)$ with the property that

\begin{numequation}\label{estimate_c}
|c_{\nu,j}^{(n)}|_K \ge \left|\frac{1}{n!}\right|_K  \hskip10pt \mbox{and} \hskip10pt \nu_1 + \ldots + \nu_\tau + \sum_{k=\tau+1}^t \nu_{j,k} \ge n \,.
\end{numequation}

\noindent We will use this inequality to estimate the $||\cdot||_r$ - norm of any lift of $\frac{1}{n!}z_{\beta^+}^n.v^+$ to

$$\bigoplus_{1 \le j \le k} U_r(\fru_P^-) \otimes_K Kv_j\;.$$

\noindent Here the $\|\cdot\|_r$ - norm on this free $U_r(\fru_P^-)$-module is the supremum norm of the $||\cdot||_r$ - norms on each 
summand, defined by

\vspace{-0.3cm}
$$\|\delta \otimes v_j\|_r := \|\delta\|_r \;,$$


\noindent for $\delta \in U_r(\fru_P^-)$. Because $B'(u)=\emptyset$ (cf. Step 3), we have $z_{\beta^+} = p^{-a}y^{(m)}_{\beta^+}$ for some $a>0$. We then get from \ref{basicrelation}

\vspace{-0.3cm}
$$\frac{z_{\beta^+}^n}{n!}.v^+ =  p^{-na} \frac{(y^{(m)}_{\beta^+})^n}{n!} \cdot v^+ = \sum_{1 \le j \le k} \sum_{\nu \in \cI_{n,j}} c_{\nu,j}^{(n)} p^{-na+(m_0+m)(n- \nu_1-\ldots-\nu_\tau)} (y^{(m)}_1)^{\nu_1} \cdot \ldots \cdot (y^{(m)}_\tau)^{\nu_\tau}.v_j\,.$$

\vskip8pt

\noindent It follows from \cite[6.2, 6.4]{Sch} that the restriction of the norm $\|\cdot\|_r$ on $D(H,K)$ to $D(H^m,K)$ is equivalent to the norm $\|\cdot\|^{(m)}_s$ on $D(H^m,K)$ (cf. also \cite[7.4]{Sch}, where it is stated that the induced topologies are equivalent). Therefore, we are now going to estimate the $\|\cdot\|^{(m)}_s$ - norm of the term

$$c_{\nu,j}^{(n)} p^{-na+(m_0+m)(n- \nu_1-\ldots-\nu_r)} (y^{(m)}_1)^{\nu_1} \cdot \ldots \cdot (y^{(m)}_r)^{\nu_r} \;,$$

\vskip8pt

\noindent where $(\nu,j)$ is such that \ref{estimate_c} holds. By \ref{norm_generator}, the $\|\cdot\|^{(m)}_s$-norm of this term is greater or equal to

$$\begin{array}{cl}
& \left|\frac{1}{n!}\right|_K p^{na + (m_0+m)(\nu_1+\ldots+\nu_r-n)} s^{\nu_1+\ldots+\nu_r} \\
 & \\
= & \left|\frac{1}{n!}\right|_K p^{na- n + (m_0+m-1)(\nu_1+\ldots+\nu_r-n)} (ps)^{\nu_1+\ldots+\nu_r} \,.
\end{array}$$

\noindent Note that \ref{estimate_c} implies that $n - (\nu_1 + \ldots + \nu_r)$ is bounded from above by some constant $C_1$. Hence we get

\vspace{-0.3cm}
$$p^{na- n + (m_0+m-1)(\nu_1+\ldots+\nu_r-n)} \ge p^{n(a-1)} \cdot p^{-(m_0+m-1)C_1}$$

\noindent is bounded away from $0$ (recall that $a \in \bbZ_{>0}$). And because $s> \frac{1}{p}$ we have $ps >1$, and the term $(ps)^{\nu_1+\ldots+\nu_r}$ is unbounded as $n \ra \infty$ ($\nu_1 + \ldots + \nu_\tau \ge n - C_1$). Altogether we get that any lift of $\frac{1}{n!}z_{\beta^+}^n.v^+$ to
$\bigoplus_{1 \le j \le t} U_r(\fru^-) \otimes_K Kv_i$ has an $\|\cdot\|_r$ - norm which exceeds $C_2(ps)^n$, where $C_2>0$ is some constant (cf. the relation between the norms $\|\cdot\|_r$ and $\|\cdot\|^{(m)}_s$ mentioned above).

The sum $\sum_{n \ge 0} \frac{1}{n!}z_{\beta^+}^n.v^+$, which is an element of $\hat{M}$ is therefore not contained in the image of $\frm_r$ (under the injection $\frm_r \hra \hat{M}$). And as we have seen before, this in turn proves that $\sum_{n \ge 0} \frac{(-\log(u))^n}{n!}.v^+ = \delta_{u^{-1}}.v^+ \in \hat{M}$ is not in the image of $\frm_r$. \qed

\begin{thm}\label{irredgeneral} Assume that Assumption \ref{hyp} holds. Let $M \in \cO_\alg$ be simple and assume that $\frp$ is maximal 
for $M$ (cf. \ref{maximal}). Let $V$ be a smooth and irreducible $L_P$-representation. Then $\cF^G_P(M,V) = \Ind^G_P(W'\otimes_K V)^\frd$ is topologically irreducible as a $G$-representation.
\end{thm}

\noindent \Pf We start by giving an outline of the proof. Put $X = (\cF^G_P(M,V),\pi)$. Then, to any $G$-subrepresentation $U \subset X$ there 
is a smooth representation $U_{sm}$ and a natural $G$-morphism $\Ind^G_P(W')^\frd \otimes_K U_{sm} \ra U$. Here $U_{sm}$ is a subrepresentation of $X_{sm}$, and we will show that $X_{sm}$ is isomorphic to the smoothly induced representation $\ind^G_P(V)$. We will prove that if $U$ is closed and non-zero, 
then $U_{sm}$ is non-zero as well, and the composite map

\vspace{-0.3cm}
$$\Ind^G_P(W')^\frd \otimes_K U_{sm} \, \hra \, \Ind^G_P(W' )^\frd \otimes_K \ind^G_P(V) \,
\lra \, X$$


\noindent is surjective. As this map factors through $U$ we necessarily have $U=X$.

\vskip10pt

\noindent {\it Step 1: the smooth representation $U_{sm}$.}
For a $G$-subrepresentation $U \sub X$, we put

\vspace{-0.3cm}
$$U_{sm} = \varinjlim_H \Hom_H(\Ind_P^G(W')^\frd|_H, U|_H) \,,$$


\noindent where the limit is taken over all compact open subgroups $H$ of $G$, and the homomorphisms are continuous homomorphisms of 
topological vector spaces. Of course, it suffices to take the limit over a set of compact open subgroups which is a neighborhood basis 
of $1 \in G$. There is an action of $G$ defined on $U_{sm}$ as follows: for $\phi \in U_{sm}$,
$g \in G$ put

\vspace{-0.3cm}
$$(g\phi)(f) = \pi(g)(\phi(\pi'(g^{-1})(f))) \,,$$


\noindent where $\pi'$ is the representation on $\Ind^G_P(W')^\frd.$ Note that $U_{sm}$ is by construction a smooth representation. Taking for $U$ the whole space $X$ gives us a smooth $G$-representation 
$X_{sm}$.
Note that there is always a continuous map of $G$-representations

\vspace{-0.3cm}
$$\Phi_U: \Ind_P^G(W')^\frd \otimes_K U_{sm} \lra U \,,
\; f \otimes \phi \mapsto \phi(f) \,.$$


\noindent In order to analyze $U_{sm}$ we will need to understand the restriction of $\Ind_P^G(W')^\frd$ to compact-open subgroups $H$.

\noindent {\it Step 2: restricting $\Ind_P^G(W')^\frd$ to compact open subgroups.} Let $H \sub G_0$ be an open normal subgroup. We have

\vspace{-0.3cm}
\begin{numequation}\label{generaldecomp}
\Ind_P^G(W')^\frd|_H \; = \; \Ind_{P_0}^{G_0}(W')^\frd|_H \; = \; \bigoplus_{\gamma \in G_0/HP_0}
\Ind^H_{H \cap \gamma P_0 \gamma^{-1}}((W')^\gamma)^\frd \,.
\end{numequation}

\noindent The notation on the right hand side has the following meaning. The underlying $K$-vector spaces of $(W')^\gamma$ and $W'$ are the same. Further, if $(\rho',W')$ is the given representation of then $((\rho')^\gamma, (W')^\gamma)$ is the representation of $H \cap \gamma P_0 \gamma^{-1}$ defined by

\vspace{-0.3cm}
$$(\rho')^\gamma(h)= \rho'(\gamma^{-1}h\gamma) \;.$$


\noindent It is now straightforward to check that, as $H$-representations, with the notation introduced before \ref{irredH},

\vspace{-0.3cm}
$$\Ind^H_{H \cap \gamma P_0 \gamma^{-1}}((W')^\gamma)^\frd \; \simeq \; \delta_\gamma \star \Ind^{HP_0}_{P_0}(W')^\frd \;.$$


\noindent By \ref{irredH} (ii, iii), all $H$-representations on the right hand side of \ref{generaldecomp} are topologically irreducible, and between any two such representations $\Ind^H_{H \cap \gamma_1 P_0 \gamma_1^{-1}}(W)$ and $\Ind^H_{H \cap \gamma_2 P_0 \gamma_2^{-1}}(W)$ with $\gamma_1HP_0 \neq \gamma_2HP_0$, there is no non-zero continuous $H$-equivariant morphism. Similarly, for the representation $\Ind_P^G(W' \otimes_K V)^\frd$ we have

\vspace{-0.3cm}
$$\Ind_P^G(W' \otimes_K V)^\frd|_H \; = \; \Ind_{P_0}^{G_0}(W' \otimes_K V )^\frd|_H \; = \; \bigoplus_{\gamma \in G_0/HP_0}
\Ind^H_{H \cap \gamma P_0 \gamma^{-1}}((W')^\gamma \otimes_K V^\gamma)^\frd \,.$$


\noindent Denote by $V^{H \cap P_0} \sub V$ the subspace on which $H \cap P_0$ acts trivially. Then

\vspace{-0.3cm}
$$\Ind^H_{H \cap \gamma P_0 \gamma^{-1}}((W')^\gamma \otimes_K (V^{H \cap P_0})^\gamma)^\frd$$


\noindent is an $H$-subrepresentation of $\Ind^H_{H \cap \gamma P_0 \gamma^{-1}}((W')^\gamma \otimes_K V^\gamma)^\frd$. And

\vspace{-0.3cm}
$$\Ind^H_{H \cap \gamma P_0 \gamma^{-1}}((W')^\gamma \otimes_K (V^{H \cap P_0})^\gamma)^\frd$$


\noindent is canonically isomorphic to

\vspace{-0.3cm}
$$\left(\Ind^H_{H \cap \gamma P_0 \gamma^{-1}}((W')^\gamma)\right)^\frd \otimes_K V^{H \cap P_0} \,.$$


\noindent Let

\vspace{-0.3cm}
$$\phi: \Ind_{P_0}^{G_0}(W')^\frd|_H \ra \Ind_{P_0}^{G_0}(W' \otimes_K V )^\frd|_H$$


\noindent be a continuous homomorphism of $H$-representations. For each $\gamma,$ let $f_\gamma$ be a non-zero vector in

$$\Ind^H_{H \cap \gamma P_0 \gamma^{-1}}((W')^\gamma)^\frd \,.$$

\noindent Because of the irreducibility of the representations $\Ind^H_{H \cap \gamma P_0 \gamma^{-1}}((W')^\gamma)^\frd$, the map $\phi$ is uniquely determined by the images $\phi(f_\gamma) \in \Ind_{P_0}^{G_0}(W' \otimes_K V )^\frd$. By the very definition of the representation $\Ind_{P_0}^{G_0}(W' \otimes_K V )^\frd$, the function $\phi(f_\gamma): G_0 \ra W' \otimes_K V$ is rigid-analytic on the components of some covering of $G_0$ by 'polydiscs' $\Delta_i$, and takes values in a $BH$-subspace
of $W' \otimes_K V$. But any such $BH$-subspace is finite-dimensional, hence contained in a subspace of the form $W' \otimes_K V^{H' \cap P_0}$ for some (small enough) open subgroup $H' \sub H$ which is normal in $G_0$. Shrinking $H'$ we may further suppose that each $\Delta_i$ is $H'$-stable (by multiplication from the left), and that each $\phi(f_\gamma)$ takes values in $W' \otimes_K V^{H' \cap P_0}$. For any $h \in H$, we have

$$\phi(h.f_\gamma)(g) = [h.\phi(f_\gamma)](g) = \phi(f_\gamma)(h^{-1}g) \in W' \otimes_K V^{H'\cap P_0} \,.$$

\vskip8pt

\noindent Because the subspace of $\Ind^H_{H \cap \gamma P_0 \gamma^{-1}}((W')^\gamma)^\frd$ generated by the functions $h.f_\gamma$, $h \in H$, is dense in $ \Ind^H_{H \cap \gamma P_0 \gamma^{-1}}((W')^\gamma)^\frd$, we have $\phi(f)(g) \in W' \otimes_K V^{H' \cap P_0}$ for all $f \in  \Ind^H_{H \cap \gamma P_0 \gamma^{-1}}((W')^\gamma)^\frd$. It follows that $\phi$ induces a continuous map of $H'$-representations

$$\phi: \Ind_{P_0}^{G_0}(W')^\frd|_{H'} \ra \Ind_{P_0}^{G_0}(W' \otimes_K V^{H' \cap P_0} )^\frd|_{H'} = \left[\Ind^{G_0}_{P_0} (W')^\frd \right]|_{H'} \otimes_K V^{H' \cap P_0}$$

\vskip8pt

\noindent By \ref{irredH} (ii, iii), the irreducible $H'$-representation $\Ind^{H'}_{H' \cap \gamma P_0 \gamma^{-1}}((W')^\gamma)^\frd$ contained in the left hand side is mapped by $\phi$ to the
corresponding summand $\left[\Ind^{H'}_{H' \cap \gamma P_0 \gamma^{-1}}((W')^\gamma)^\frd\right] \otimes_K V^{H' \cap P_0}$ on the right hand side. The map $\phi$, restricted to this summand, is thus of the form $f \mapsto f \otimes v$ for some $v \in V^{H' \cap P_0}$.
This shows that any element in $X_{sm}$ is ''locally'' (i.e. on the subrepresentations $\Ind^{H'}_{H' \cap \gamma P_0 \gamma^{-1}}(W^\gamma)^\frd$) given by vectors in $V^{H' \cap P_0}$ (for sufficiently small subgroups $H'$). It is now an easy matter to verify that the 
$G$-action on $X_{sm}$ identifies $X_{sm}$ with the smoothly induced representation $\ind_P^G(V)$. Thus we have shown the

\begin{sublemma} The canonical map $\ind_P^G(V) \ra X_{sm}$, $\varphi \mapsto [f \mapsto (g \mapsto f(g) \otimes \varphi(g))]$, is an isomorphism.
\end{sublemma}

\noindent {\it Step 3: $U_{sm}$ is non-zero for non-zero $U$.} Let $U \sub X$ be a non-zero closed $G$-invariant subspace. Let $f \in U$ be a non-zero element. Then, as we have seen before, $f$ takes values in a vector space $W' \otimes_K V^{H \cap P_0}$ for a sufficiently small compact open subgroup $H$. Therefore, $f$ is contained in

$$\bigoplus_{\gamma \in G_0/HP_0}
\Ind^H_{H \cap \gamma P_0 \gamma^{-1}}\left((W')^\gamma \otimes_K (V^{H \cap P_0})^\gamma \right)^\frd = \bigoplus_{\gamma \in G_0/HP_0}
\Ind^H_{H \cap \gamma P_0 \gamma^{-1}}\left((W')^\gamma\right)^\frd \otimes_K V^{H \cap P_0}$$

\noindent It follows from this that $U$ contains a non-zero $H$-invariant subspace of

$$\bigoplus_{\gamma \in G_0/HP_0}
\Ind^H_{H \cap \gamma P_0 \gamma^{-1}}\left((W')^\gamma \right)^\frd \otimes_K V^{H\cap P_0} \,.$$

\noindent Because the representations $\Ind^H_{H \cap \gamma P_0 \gamma^{-1}}(W^\gamma)$ are irreducible, and because there is no non-zero continuous intertwiner between any two such (with $H\gamma_1 P_0 \neq H \gamma_2 P_0$), any such subspace is isomorphic to

\vspace{-0.3cm}
$$\bigoplus_{\gamma \in G_0/HP_0}
\Ind^H_{H \cap \gamma P_0 \gamma^{-1}}\left((W')^\gamma \right)^\frd \otimes_K V_\gamma \,,$$

\noindent with a subspace $V_\gamma$ of $V^{H \cap P_0}$. Thus, any non-zero vector in some $V_\gamma$ gives rise to a non-zero $H$-equivariant homomorphism from $\Ind_{P_0}^{G_0}(W')^\frd|_H$ to $U$. 
Therefore, $U_{sm} \neq \{0\}$.

\vskip8pt

\noindent {\it Step 4: $\Ind^G_P(W')^\frd \otimes_K U_{sm}$ surjects onto $X$.} Let $\phi_0 \in U_{sm}$ be a non-zero element which we identify with an element of the smoothly induced representation $\ind^G_P(V)$. Without loss of generality we may assume $\phi_0(1) \neq 0$. Consider the $P$-morphism 
$\Pi: U_{sm} \ra V$ given by $\phi \mapsto \phi(1)$. As $\phi_0(1) \neq 0$ we deduce that the image of $\Pi$ is a non-zero subrepresentation of $V$, and is hence equal to $V$ (because $V$ was supposed to be irreducible).
Therefore, for any $v \in V$ there is some $\phi \in U_{sm}$ with $\phi(1) = v$. As $U_{sm}$ is a $G$-representation we find that for any $g \in G$ and any $v \in V$ there is some 
$\phi \in U_{sm}$ with $\phi(g) = v$. Then, on a neighborhood $N$ of $g$, the map $\phi$ is constant with value $v$ on this neighborhood.
Let $S \sub G_0$ be a compact locally analytic subset such that $G = S \cdot P$ and hence $G_0 = S \cdot P_0$. (That is, every $g \in G$ is the product $s \cdot p$ with uniquely determined $s \in S$ and $p \in P$). Then there is a compact open subset $S' \sub S$ and a compact open neighborhood of the identity $P' \sub P_0$ such that $S' \cdot P' \subset N$. Let $f: G_0 \ra W'$ be a function whose support is contained in $S' \cdot P_0$. Then the function $S \ra W' \otimes_K V$, $s \mapsto f(s) \otimes \phi(s)$, is equal to the function $S \ra W' \otimes_K V$, $s \mapsto f(s) \otimes v$.

\vskip5pt

Let $f \in \Ind^G_P(W' \otimes_K V)^\frd$ be any element. Then $f$ is uniquely determined by its restriction to $S$.
As we have seen above, the set $f(G_0)$ is contained in a finite-dimensional vector space 
$W' \otimes_K V_0$ with $V_0 \sub V$. Let $v_1, \ldots, v_r$ be a basis of $V_0$. Write 
$f(s) = f_1(s) \otimes v_1 +\ldots + f_r(s) \otimes v_r$ with locally analytic $W'$-valued functions $f_i$ on $S$. 
Extend each function $f_i$ to $G$ by $f_i(s \cdot p) = \rho'(p^{-1})(f_i(s))$, where $\rho'$ is the representation of $P$ on $W'$. 
Then $f_i \in \Ind^G_P(W')$ for all $i$. In fact, examining the action of the differential operators in $\frd$ shows that 
$f_i \in \Ind^G_P(W')^\frd$ for all $i$. For any $s \in S$ choose some $\phi_{i,s} \in U_{sm}$ such that $\phi_{i,s}(x) = v_i$ for all $x$ in a compact open neighborhood $N_{i,s} \sub S$ of $s$. As $S$ is compact, finitely many of the $N_{i,s}$ will cover $S$. We can then choose a (finite) disjoint refinement $(N_{i,j})_j$ of the finite covering. Then we restrict $f_i$ to each of these $N_{i,j} \cdot P$, and extend it by $0$ to $(S
  \setminus N_{i,j}) \cdot P$. Denote the function thus obtained by $f_{i,j}$. Again, all $f_{i,j}$ lie in $\Ind^G_P(W')^\frd$, and $x \mapsto f_{i,j}(x) \otimes \phi_{i,s(i,j)}(x) = f_{i,j}(x) \otimes v_i$, for a suitably chosen $s(i,j) \in S$. Obviously, we have: $f = \sum_{i,j} f_{i,j} \otimes \phi_{i,s(i,j)}$. Indeed, by construction,
both $f|_S$ and the restriction of the sum to $S$ coincide. And both are functions in $\Ind^G_P(W' \otimes_K V)^\frd$; therefore, they are equal.
This shows that $f$ can be written as a sum of simple tensors of functions in $\Ind^G_P(W')^\frd$ and functions in $U_{sm}$ hence, the 
map $\Ind^G_P(W')^\frd \otimes_K U_{sm} \ra \Ind^G_P(W' \otimes_K V)^\frd = X$ is surjective. \qed

\vskip18pt

\begin{exam}\label{Example}
a) Let $M$ be a finite-dimensional irreducible algebraic ${\bf G}$-module, so that in the above theorem we may choose ${\bf P=G}$. Then we deduce that $\cF^G_G(M,V)=M \otimes_K V$ is topologically irreducible for any smooth irreducible representation $V$ of $G$. This result was already proved by D. Prasad \cite[app.]{ST4}, and our proof is modeled after his approach.

b) Let $M= U(\frg) \otimes_{U(\frp)} W \in \cO^\frp_\alg$ be an irreducible generalized Verma module for some irreducible $\bL_\bP$-representation $W$. Then by Cor. \ref{locfinitesubalgebra} the parabolic subalgebra $\frp$ is maximal in the above sense. It follows  that $\cF^G_P(M)=\Ind^G_P(W')$ is topologically irreducible. This result was proved in \cite{OS} for an arbitrary irreducible finite-dimensional locally analytic $L_P$-representation $W.$
\end{exam}

\section{Composition series}\label{Composition series}

\begin{para} Before we formulate our main result in this section, we recall the definition of (smooth) generalized Steinberg representations, cf. \cite{BW}, \cite{Ca}.
Let $P$ be a standard parabolic subgroup of $G$ and denote by $i^G_P=\ind^G_P(1)=C^\infty(G/P,K)$ the smooth induced representation of locally constant functions on $G/P.$ The generalized Steinberg representation to $P$ is the quotient

$$v^G_P= i^G_P/\sum_{P\subsetneq Q \subset G}i^G_Q \,.$$

\noindent This is an irreducible $G$-representation and all irreducible subquotients of $i^G_P$ are of the form $v^G_Q$ with $P\subset Q$. Each such representation occurs with multiplicity one, cf. loc.cit.
\end{para}

\setcounter{enumi}{0}

\begin{para} {\it How to obtain a composition series for $\cF^G_P(M,V)$.} We are going to describe a method for finding a composition series of the representation $\cF^G_P(M,V)$. Let $M$ be an object of $\cO^\frp_\alg$ and $V$ an object of $\Rep^{\infty,a}_K(L_P)$ of finite length. Let $(0) = V_0 \subsetneq V_1 \subsetneq \ldots \subsetneq V_r = V$ be a composition series of smooth $L_P$-representations for $V$. By the exactness of $\cF^G_P$ in the second argument, cf. Prop. \ref{exact_in_both}, we get a chain of inclusions

\vspace{-0.3cm}
$$(0) = \cF^G_P(M,V_0) \subsetneq  \cF^G_P(M,V_1) \subsetneq \ldots \subsetneq  \cF^G_P(M,V_r) =  \cF^G_P(M,V) \,.$$

\noindent Further the quotient $\cF^G_P(M,V_{i+1})/ \cF^G_P(M,V_i)$ is isomorphic to $\cF^G_P(M,V_{i+1}/V_i)$. We may hence assume that $V$ is irreducible.

Let $M$ be an object of the category $\cO_\alg^\frp.$ Then it has a
Jordan-H\"older series

\vspace{-0.3cm}
$$M=M^0 \supsetneq M^1 \supsetneq M^2 \supsetneq \ldots \supsetneq M^r
\supsetneq M^{r+1}=(0) $$

\noindent in the same category. Thus we may apply our functor $\cF^G_P$ to the
pair $(M,V)$. We get a sequence of surjections

\vspace{-0.3cm}
$$\cF^G_P(M,V) = \cF^G_P(M^0,V) \stackrel{p_0}{\lra} \cF^G_P(M^1,V)
\stackrel{p_1}{\lra} \cF^G_P(M^2,V) \stackrel{p_2}{\lra} \ldots
\stackrel{p_{r-1}}{\lra} \cF^G_P(M^r,V) \stackrel{p_r}{\lra} (0) \,.$$


\noindent For any integer $i$ with $0\leq i\leq r,$ we put

\vspace{-0.3cm}
$$q_i:=p_i\circ p_{i-1} \circ \cdots \circ p_1 \circ p_0 \,.$$

\noindent and

\vspace{-0.3cm}
$$\cF^i:=\ker(q_i)$$

\noindent which is a closed subrepresentation of $\cF^G_P(M,V)$. We obtain a filtration

\vspace{-0.3cm}
\begin{numequation}\label{filtration_neu}
\cF^{-1}=(0)\subset \cF^0 \subset \cdots \subset \cF^{r-1} \subset
\cF^{r}=\cF^G_P(M,V)
\end{numequation}
\noindent by closed subspaces with

\vspace{-0.3cm}
$$\cF^i/\cF^{i-1} \cong \cF^G_P(M^i/M^{i+1},V) \,.$$


Let $Q_i = L_i \cdot U_i \supset P$ be the standard parabolic subgroup
which is maximal for $M^i/M^{i+1}$. Then

$$\cF^G_P(M^i/M^{i+1},V) = \cF^G_{Q_i}(M^i/M^{i+1},i^{L_i}_{L_P\cdot
L_i\cap P}(V)) \,.$$

\noindent Then we can choose a Jordan-H\"older series for $i^{L_i}_{L_P\cdot
L_i\cap P}(V)$ and we obtain a Jordan-H\"older series for

\vspace{-0.3cm}
$$\cF^G_{Q_i}(M^i/M^{i+1}, i^{L_i}_{L_P\cdot L_i\cap P}(V)).$$

\noindent By refining the filtration \ref{filtration_neu} we get thus a
Jordan-H\"older series of $\cF^G_P(M)$.

In the case where $V$ is actually the
trivial representation the irreducible subquotients of the smooth
representation $i^{L_i}_{L_P\cdot L_i\cap P}=i^{Q_i}_P$ are the generalized
Steinberg representations $v^{Q_i}_Q$ with $Q_i\supset Q\supset P$, recalled above. Thus the
irreducible subquotients of $\cF^G_P(M,1) = \cF^G_P(M)$ are of the form
$\cF^G_{Q_i}(M^i/M^{i+1},v^{Q_i}_Q),\; i=1,\ldots,r$, and where $Q_i\supset Q\supset P$.

\end{para}

\vskip8pt

\begin{rmk}
Independently of our work, B. Schraen \cite{Schr} determined the Jordan-H\"older series for certain subquotients of parabolically induced locally analytic representations for the group ${\rm GL}_3(\bbQ_p)$, including the locally analytic Steinberg representation.
\end{rmk}


\section{Applications to equivariant vector bundles on Drinfeld's half space}

\setcounter{enumi}{1}

Let $\cX$ be Drinfeld's half space of dimension $d\geq1$ over $K$. This is a rigid-analytic variety over $K$ given by the complement of all  $K$-rational hyperplanes in projective space $\bbP_K^d,$ i.e.,

$$\cX = \bbP_K^d\setminus \bigcup \nolimits_{H \subsetneqq K^{d+1}}\mathbb \bbP(H),$$
where $H$ runs over the set of $K$-rational hyperplanes in $K^{d+1}.$

\noindent There is  natural action of $G = \GL_{d+1}(K)$ on $\cX$ induced by the algebraic
action $m: \bG \times \bbP^d_K \rightarrow \bbP^d_K$ of $\bG = \GL_{d+1/K}$  defined by

$$g\cdot [q_0:\cdots :q_d]:=m(g,[q_0:\cdots :q_d]):= [q_0:\cdots :q_d]g^{-1} \,.$$

\noindent It is known that $\cX$ is a Stein space \cite[\S 1, Prop. 4]{ScStI}. Moreover, for every homogeneous vector bundle $\cE$  on $\bbP_K^d$ the space of sections $\cE(\cX)=H^0(\cX,\cE)$ has the structure of a $K$-Fr\'echet space equipped with a continuous $G$-action. Its topological dual
$\cE(\cX)'$ is a locally analytic $G$-representation, cf. \cite{ST3,O}.

In \cite{O} one of us studied the structure of $\cE(\cX)$ for homogeneous  vector bundles  $\cE$ on $\bbP^d_K$. The theorem below generalizes the result of Schneider and Teitelbaum \cite{ST3} for the canonical bundle
$\cE = \Omega^d$ resp. Pohlkamp \cite{Po} for the structure sheaf $\cE = \cO$.

\begin{thm}\label{TheoremVB}
Let $\cE$ be a homogeneous vector bundle on $\bbP^d_K.$ There is a $G$-equivariant filtration by closed $K$-subspaces of
$\cE(\cX)$
\begin{numequation}\label{filtration}
 \cE(\cX)=\cE(\cX)^0 \supset \cE(\cX)^{1} \supset \cdots \supset \cE(\cX)^{d-1} \supset \cE(\cX)^d =  H^0(\bbP^d,\cE),
\end{numequation}
\noindent such that for $j=0,\ldots,d-1,$ there are extensions  of locally analytic $G$-representations

\vspace{-0.3cm}
\begin{numequation}\label{extensions}
 0 \rightarrow v^{G}_{{P_{(j+1,1,\ldots,1)}}}(H^{d-j}(\bbP^d_K,\cE)') \rightarrow (\cE(\cX)^{j}/\cE(X)^{j+1})' \rightarrow  \Ind^{G}_{P_{(j+1,d-j)}}(U_j')^{\frd_j} \rightarrow 0.
\end{numequation}
\end{thm}

Here, for a decomposition $(n_1,\ldots,n_s)$ of $d+1$, the symbol $P_{(n_1,\ldots,n_s)}$ denotes the corresponding lower standard parabolic subgroup of $G$. The module $v^{G}_{{P_{(j+1,1,\ldots,1)}}}(H^{d-j}(\bbP^d_K,\cE)')$ is a generalized Steinberg representation with coefficients in the finite-dimensional  algebraic $G$-module $H^{d-j}(\bbP^d_K,\cE)'$.

\noindent The $P_{(j+1,d-j)}$-representation $U_j'$ is a tensor product $N_{j}' \otimes_K v^{\GL_{d-j}}_{\GL_{d-j}\cap B}$ of an 
algebraic representation $N_{j}'$ and the Steinberg representation $v^{\GL_{d-j}}_{\GL_{d-j}\cap B}$ of the factor 
$\GL_{d-j}$ of the Levi subgroup $L_{(j+1,d-j)} \subset P_{(j+1,d-j)}$. The first one is characterized by the property that it generates the
kernel $\tilde{H}^{d-j}_{\bbP^j_K} (\bbP^d_K, \cE)$ of the natural  homomorphism
$$H^{d-j}_{\bbP^j_K} (\bbP^d_K, \cE) \rightarrow  H^{d-j}(\bbP^d_K,\cE)$$

\noindent as a module with respect to $U(\frg)$. Here $\bbP^j_K=V(X_{j+1},\ldots,X_d)$ is the linear subvariety of $\bbP^d_K$ defined by the vanishing of the coordinates $X_{j+1},\ldots,X_d.$ The module $\frd_j$ is
just the kernel of the induced surjection
$U(\frg) \otimes_{U(\frp_{(j+1,d-j)})} N_j' \to \tilde{H}^{d-j}_{\bbP^j_K} (\bbP^d_K, \cE)$.

\vskip5pt

In the case where $\cE$  arises from an irreducible representation of the Levi subgroup $L_{(1,d)},$  we could make our result more precise, i.e.,  concerning the structure of $N_j.$ Rather than recalling this result in full generality, we restrict our attention from now on  to homogeneous line bundles on $\bbP^d_K$, where we will get even a more precise formula, cf. Prop. \ref{Propositionweight_b}.

\vskip5pt

Let $s\in \bbZ$ and denote by  $\lambda '=(s,\ldots,s) \in \bbZ^d$ the constant integral weight of $\GL_d$.
Let $r=\lambda_0\in \bbZ$ and set
$$\lambda:=(r,s,\ldots,s)\in \bbZ^{d+1} \,.$$

\noindent We denote by $\cL_\lambda$ the homogeneous line bundle on $\bbP^d_K$ such that its fibre in the base point is the irreducible algebraic $L_{(1,d)}$-representation corresponding to $\lambda$. Then we obtain
\begin{numequation}\label{linebundle}
 \cL_\lambda=\cO(r-s)
\end{numequation}

\noindent where the $G$-linearization is given by the tensor product of the natural one on $\cO(r-s) $ with the character $\det^s$.

Put $w_j:=s_j\cdots s_1,$ where $s_i\in W$ is the (standard) simple reflection in the Weyl group $W \cong S_{d+1}$ of $G$.
Recall that $\cdot$ denotes the dot action of $W$ on $X^\ast({\bf T})_{\bbQ}$.
If $ \chi=(\chi_0,\ldots,\chi_d)\in \bbZ^{d+1}$ we get
$$w_i \cdot \chi =   (\chi_1-1,\chi_2-1,\ldots, \chi_i-1,\chi_0 +i , \chi_{i+1},\ldots,\chi_d) \,.$$
Hence for $\chi=\lambda=(r,s,\ldots,s)$, we compute

\vspace{-0.3cm}
\begin{eqnarray*}
w_0 \cdot \lambda & = &  \lambda \\ 
w_1 \cdot \lambda & = &  (s-1,r+1 ,s,\ldots,s) \\
& \vdots & \\
w_i \cdot \lambda & = &  (s-1,s-1,\ldots, s-1,r+i , s,\ldots,s) \\
& \vdots & \\
w_d \cdot \lambda & = &  (s-1 ,s-1,\ldots,s-1,r+d)\,.
\end{eqnarray*}

\vskip8pt

In particular, there is at most one integer $0\leq i_0 \leq d$, such that $w_{i_0}\cdot \lambda$ is dominant with respect to the Borel subgroup $\bB^+$ of upper triangular matrices. In fact, this integer is characterized by the non-vanishing of $H^{i_0}(\bbP^d_K, \cL_\lambda)$ (cf. \cite[Thm. IV']{Bo} resp. \cite[Thm. III.5.1]{Ha}),
which is $i_0=0$ for $r\geq s$ resp. $i_0=d$ for $s \geq r+d+1$.
Otherwise, there is a unique integer $i_0 < d$ with $w_{i_0}\cdot \lambda = w_{i_0+1}\cdot \lambda.$ This is the case for $0\leq i_0=s-r-1<d+1.$ We get
\begin{numequation}\label{domordkl}
w_{i}\cdot \lambda \succ  w_{i+1}\cdot \lambda
\end{numequation}
\noindent for all $i\geq i_0$ (resp. $i > i_0$ if $w_{i_0}\cdot \lambda = w_{i_0+1}\cdot \lambda$), and
\begin{numequation}\label{domordgr}
w_{i}\cdot \lambda \prec w_{i+1}\cdot \lambda
\end{numequation}
\noindent for all $i<i_0$, with respect to the dominance order $\succ$ on $X^\ast({\bf T})_{\bbQ}$. We put

$$\mu_{i,\lambda}:=\left\{
\begin{array}{cc}
w_{i-1}\cdot \lambda  & : i \leq i_0 \\
w_{i} \cdot \lambda & : i > i_0 \,
\end{array} \right\} \,,i=1,\ldots,d.$$

\noindent This is a ${\bf L_{(i,d-i+1)} }$-dominant weight (with respect to the Borel subgroup ${\bf L_{(i,d-i+1)}} \cap {\bf B^+}$).
Let $V_{i,\lambda}$ be the finite-dimensional irreducible ${\bf L_{(i,d-i+1)} }$-module
with highest weight $\mu_{i,\lambda}.$ Let ${\bf P^+_{(i,d-i+1)}}$ be the upper triangular parabolic subgroup to the decomposition $(i,d+1-i)$ and denote by ${\bf U^+_{(i,d-i+1)}}$ its unipotent radical. By considering the trivial action of ${\bf U^+_{(i,d-i+1)}}$ on $V_{i,\lambda}$, we may view it as a ${\bf P^+_{(i,d-i+1)}}$-module.

In the following we identify the linear subvarieties $\bbP^{d-i}_K, \,0\leq i\leq d,$ with the closed subschemes $V(X_0,\ldots,X_{i-1}) \subset \bbP^d_K$ defined by the vanishing  of the first $i$ coordinate functions.
Note that the stabilizer of this subvariety in ${\bf G}$  is just ${\bf P^+_{(i,d+1-i)}}.$
\footnote{We note that for formulating Thm. \ref{TheoremVB} we have used in loc.cit. the identification of $\bbP^{d-i}_K$ with $V(X_{d-i+1},\ldots,X_d)$. Therefore the standard parabolic subgroups used there are lower (block) triangular. Afterwards we used the conjugacy of  $V(X_0,\ldots,X_{i-1})$ and  $V(X_{d-i+1},\ldots,X_d)$ within  $\bbP^d_K$ via the action of ${\bf G}$ on $\bbP^d_K.$}
In loc.cit. we  saw that we can realize $V_{i,\lambda}$ as a submodule of $\tilde{H}^{i}_{\bbP^{d-i}_K}(\bbP^d_K,\cL_\lambda)$. 

In fact, $\tilde{H}^{i}_{\bbP^{d-i}_K}(\bbP^d_K,\cL_\lambda)$ is naturally a module for ${\bf P^+_{(i,d-i+1)}} \ltimes U(\frg)$. This symbolic notation signifies that it is both a $U(\frg)$-module and an algebraic representation of ${\bf P^+_{(i,d-i+1)}}$, and that these two 
structures are compatible in the sense that $h.(\frx.(h^{-1}.v)) = {\rm Ad}(h)(\frx).v$ for all $h \in {\bf P^+_{(i,d-i+1)}}$, $\frx \in \frg$, $v \in \tilde{H}^{i}_{\bbP^{d-i}_K}(\bbP^d_K,\cL_\lambda)$. 

The operation of $\frg$ is induced by the homogeneous line bundle. The second one is induced by functoriality since it is the stabilizer of $\bbP^{d-i}_K$.

\begin{prop}\label{Propositionweight}
For $1\leq i \leq d,$ the ${\bf P^+_{(i,d-i+1)}} \ltimes U(\frg)$-module
$\tilde{H}^{i}_{\bbP^{d-i}_K}(\bbP^d_K,\cL_\lambda)$ coincides with the  module $\bigoplus_{k_0,\ldots, k_{i-1} \leq 0  \atop {k_{i},\ldots, k_d \geq 0  \atop {k_0+\cdots + k_d=0}}} K\cdot X_0^{k_0}X_1^{k_1}\cdots X_d^{k_d} \cdot  V_{i,\lambda} .$
\end{prop}

\noindent \Pf This was proved in loc.cit. Prop. 1.4.2.
\qed

From this latter statement we deduce immediately the following corollary.

\begin{cor}
 For $1\leq i \leq d,$ the $U(\frg)$-module
$\tilde{H}^{i}_{\bbP^{d-i}_K}(\bbP^d_K,\cL_\lambda)$ lies in the category $\cO_\alg^{\frp^+_{(i,d-i+1)}}.$
\end{cor}

\begin{rmk}
It is well-known  that the above objects lie in $\cO,$ although, mostly the case of the full flag variety is considered, cf. e.g. \cite{Ku}, \cite{AL}. Indeed this can be seen  by using the Grothendieck-Cousin complex which was used in loc.cit. to compute  $\tilde{H}^{i}_{\bbP^{d-i}_K}(\bbP^d_K,\cL_\lambda)$.  Here all contributions of this complex are objects in $\cO.$ Hence the cohomology of this complex which computes $\tilde{H}^{i}_{\bbP^{d-i}_K}(\bbP^d_K,\cL_\lambda)$ lies in $\cO$, as well.
\end{rmk}

Recall that we denote for a character $\mu \in X^\ast({\bf T})$ by $L(\mu)$ the irreducible highest weight $U(\frg)$-module of weight $\mu$.

\setcounter{enumi}{0}

\begin{prop}\label{Propositionweight_b}
For $1\leq i \leq d,$ the ${\bf P^+_{(i,d-i+1)}} \ltimes U(\frg)$-module
$\tilde{H}^{i}_{\bbP^{d-i}_K}(\bbP^d_K,\cL_\lambda)$ is isomorphic to $L(\mu_{i,\lambda}).$
\end{prop}

\noindent \Pf In order to show the statement it suffices to check the following conditions.

\smallskip
1) $\tilde{H}^{i}_{\bbP^{d-i}_K}(\bbP^d_K,\cL_\lambda)$ is a highest weight module of weight $\mu_{i,\lambda}$.

\smallskip
2) $\tilde{H}^{i}_{\bbP^{d-i}_K}(\bbP^d_K,\cL_\lambda)$ is irreducible.

\noindent Because of identity \ref{linebundle} it suffices to check the conditions in the case where $s=0.$ Then $\cL_\lambda=\cO(r)$ is the r-th Serre twist with its natural $G$-linearization.
All cohomology groups are $K$-subspaces of $\bigoplus_{k_0,\ldots, k_{d}\in \bbZ} K\cdot X_0^{k_0}X_1^{k_1}\cdots  X_d^{k_d}$.
Here the action of $\frg$ is given as follows. For a root $\alpha=\alpha_{i,j}=\epsilon_i-\epsilon_j\in \Phi,$ let
$$L_\alpha :=L_{(i,j)} \in \frg_\alpha$$ be the standard generator of the
weight space $\frg_\alpha$  in $\frg.$ Then the action  of $\frg$  is determined by
\begin{numequation}\label{Lieaktion}
L_{(i,j)}\cdot X_0^{k_0}X_1^{k_1}\cdots  X_d^{k_d} = k_j\cdot \frac{X_i}{X_j}\cdot X_0^{k_0}X_1^{k_1}\cdots  X_d^{k_d}
\end{numequation}
and
\begin{equation*}
t\cdot X_0^{k_0}X_1^{k_1}\cdots  X_d^{k_d} = (\sum\nolimits_ i k_i t_i)\cdot X_0^{k_0}X_1^{k_1}\cdots  X_d^{k_d} \; \mbox { for }t=(t_0,\ldots,t_d) \in \frt.
\end{equation*}

We prove statements 1) and 2) case by case.

a) Case $r\geq 0$. Then $\mu_{i,\lambda}=w_i\cdot \lambda$ and
$$V_{i,\lambda}=  \bigoplus_{k_i,\ldots, k_{d} \geq 0  \atop {k_i+\cdots + k_d=r+i}} K\cdot X_0^{-1}X_1^{-1}\cdots X_{i-1}^{-1} X_i^{k_i}\cdots X_d^{k_d}.$$ Further  $v_{i,\lambda}=X_0^{-1}X_1^{-1}\cdots X_{i-1}^{-1} X_i^{r+i}$ is a maximal vector in $\tilde{H}^{i}_{\bbP^{d-i}_K}(\bbP^d_K,\cL_\lambda)$ of weight $\mu_{i,\lambda}$.

It is easy to see by formula \ref{Lieaktion} that $v_{i,\lambda}$ generates  $\tilde{H}^{i}_{\bbP^{d-i}_K}(\bbP^d_K,\cL_\lambda)$ as a $U(\frg)$-module, so condition 1) is satisfied. As for condition 2), let $N\subset \tilde{H}^{i}_{\bbP^{d-i}_K}(\bbP^d_K,\cL_\lambda) $ be a submodule, $N\neq(0)$. Then $N$ lies in the category $\cO$ as a submodule, hence $\frt$ acts semi-simply. Let $X_0^{k_0}X_1^{k_1}\cdots  X_d^{k_d}\in N$, where $k_0,\ldots k_{i-1}<0$, $k_i\,\ldots,k_d\geq 0$
and $\sum^d_{j=0} k_j=r.$ Then by repeated multiplication with $L_{(k,l)}$ with $0\leq k\leq i-1$ and $l\geq i$, we can achieve that  - up to scalar - $k_0=\cdots=k_{i-1}=-1.$ By further multiplying with $L_{(k,l)}$ where $k=i$ and $l\geq i+1,\ldots,d$, we see that $v_{i,\lambda}\in N.$ Thus $N=\tilde{H}^{i}_{\bbP^{d-i}_K}(\bbP^d_K,\cL_\lambda)$.

b) Case $r\leq -d-1$. Then $\mu_{i,\lambda}=w_{i-1}\cdot \lambda$ and
$$V_{i,\lambda}=  \bigoplus_{k_0,\ldots, k_{i-1} < 0  \atop {k_0+\cdots + k_{i-1}=r}} K\cdot X_0^{k_0}X_1^{k_1}\cdots X_{i-1}^{k_{i-1}}.$$ Further $X_0^{-1}X_1^{-1}\cdots X_{i-2}^{-1} X_{i-1}^{r+i-1}$ is a maximal vector in $\tilde{H}^{i}_{\bbP^{d-i}_K}(\bbP^d_K,\cL_\lambda)$ of weight $\mu_{i,\lambda}$.

It is easy to see by formula \ref{Lieaktion} that $v_{i,\lambda}$ generates  $\tilde{H}^{i}_{\bbP^{d-i}_K}(\bbP^d_K,\cL_\lambda)$ as a $U(\frg)$-module, so condition 1) is satisfied. As for condition 2), let $N\subset \tilde{H}^{i}_{\bbP^{d-i}_K}(\bbP^d_K,\cL_\lambda) $ be a submodule, $N\neq(0)$.  Let $X_0^{k_0}X_1^{k_1}\cdots  X_d^{k_d}\in N$, where $k_0,\ldots k_{i-1}<0$, $k_i\,\ldots,k_d\geq 0$
and $\sum^{d}_{j=0} k_j=r.$ Then by repeated multiplication with $L_{(k,l)}$ with $0\leq k\leq i-1$ and $l\geq i$, we can achieve that  - up to scalar - $k_i=\cdots=k_{d}=0.$ By further multiplying with $L_{(k,l)}$ where $k < i-1$ and $l=i-1$, we see that $v_{i,\lambda}\in N.$ Thus $N=\tilde{H}^{i}_{\bbP^{d-i}_K}(\bbP^d_K,\cL_\lambda)$.

c) Case $0>r > -d-1$. Then for $i\leq i_0=-r-1$, we have  $\mu_{i,\lambda}=w_{i-1}\cdot \lambda$  and
$$V_{i,\lambda}=  \bigoplus_{k_0,\ldots, k_{i-1} < 0  \atop {k_0+\cdots + k_{i-1}=r}} K\cdot X_0^{k_0}X_1^{k_1}\cdots X_{i-1}^{k_{i-1}}.$$ Further $X_0^{-1}X_1^{-1}\cdots X_{i-2}^{-1} X_{i-1}^{r+i-1}$ is a maximal vector in $\tilde{H}^{i}_{\bbP^{d-i}_K}(\bbP^d_K,\cL_\lambda)$ of weight $\mu_{i,\lambda}$.

For $i> i_0=-r-1$, we have  $\mu_{i,\lambda}=w_{i}\cdot \lambda$  and
$$V_{i,\lambda}=  \bigoplus_{k_0,\ldots, k_{i-1} < 0  \atop {k_0+\cdots + k_{i-1}=r}} K\cdot X_0^{k_0}X_1^{k_1}\cdots X_{i-1}^{k_{i-1}}.$$ Further $X_0^{-1}X_1^{-1}\cdots X_{i-2}^{-1} X_{i-1}^{r+i-1}$ is a highest weight vector in $\tilde{H}^{i}_{\bbP^{d-i}_K}(\bbP^d_K,\cL_\lambda)$ of weight $\mu_{i,\lambda}$.

Here the reasoning is a mixture of the previous cases. \qed

\bigskip
Now we use the above result for the computation of $\tilde{H}^i_{\bbP^{d-i}_K}(\bbP^d_K,\cE)$ where $\bbP^{d-i}_K$ is identified with
the closed subscheme $V_i=V(X_{d-i+1},\ldots ,X_d)$ of $\bbP^d_K.$ Consider the block matrix
$$ z_i:=\left( \begin{array}{cc} 0  & I_i \\  I_{d+1-i} & 0 \end{array} \right) \in G,$$
where $I_j\in {\rm GL}_{j}(K)$ denotes the $j\times j$-identity matrix. Then $V(X_0,\ldots X_{i-1})$ is transformed into $V(X_{d-i+1},\ldots ,X_d)$ under the action of $z_i$ on $\bbP^d_K.$ We have
\begin{equation}\label{conjugation}
z_i \cdot {\bf P_{(d-i+1,i)}} \cdot z_i^{-1} = {\bf P^+_{(i,d+1-i)}}
\end{equation}
and on the Levi subgroups the conjugacy map
is given by
$${\bf L_{(d-i+1,i)}} \ni\left( \begin{array}{cc} A  & 0 \\  0 & B \end{array} \right) \mapsto \left( \begin{array}{cc} B & 0 \\  0 & A \end{array} \right) \in {\bf L_{(i,d-i+1)}}.$$
Hence the  ${\bf P_{(d-i+1,i)}}\ltimes U(\frg)$-module  $\tilde{H}^{i}_{V_i}(\bbP^d_K,\cE),$
is  given by $\tilde{H}^i_{\bbP^{d-i}_K}(\bbP^d_K,\cE)$ twisted with the action of $z_i.$
In particular, we can choose $V_{i,\lambda}$ - equipped with its action of ${\bf P_{(d-i+1,i)}}$ via the isomorphism \ref{conjugation} - to be the representation $N_{d-i}$ of Thm. \ref{TheoremVB}. Its highest weight is $z_i^{-1}\cdot \mu_{i,\lambda}$.

\begin{cor} If $\cE=\cL_\lambda$ is a line bundle, then
the contributions $\Ind^{G}_{P_{(j+1,d-j)}}(U_j')^{\frd_j}$ in \ref{extensions} are topologically irreducible, $j=0,\ldots,d-1$.
\end{cor}

\noindent \Pf By definition of our functor $\cF^G_P$ there is for $j=0,\ldots,d-1,$ the identity
$\Ind^{G}_{P_{(j+1,d-j)}}(U_j')^{\frd_j}$ $=\cF^G_{P_{(j+1,d-j)}}(L(z_j^{-1}\cdot\mu_{j,\lambda})).$
By Thm. \ref{irredG_0} all these contributions are topologically irreducible since the occuring standard parabolic subgroups are
(proper)  maximal and the simple modules  $L(z_j^{-1}\cdot\mu_{j,\lambda})$ do not lie in $\cO^{\frg}$ (since they are not finite-dimensional). \qed

Now we consider the remaining ingredients of the filtration \ref{filtration}. For the line bundle $\cL_\lambda$, we compute using \cite[Thm. III.5.1]{Ha}

\vspace{-0.3cm}
\begin{eqnarray*}
H^{\ast}(\bbP^d_K,\cL_\lambda) & = & H^{i_0}(\bbP^d_K,\cL_\lambda) \\ \\ & = & \left\{ \begin{array}{cl}
{\rm Sym}^{r-s}(K^{d+1})\otimes \det^s & \mbox { for  $r-s \geq 0$ and $i_0=0$} \medskip \\ \medskip
{\rm Sym}^{-d-1-(r-s)}(K^{d+1})'\otimes \det^{s-1} & \mbox{ for $r-s \leq -d-1$ and $i_0=d$}  \\ \smallskip
(0) &  \mbox{ otherwise }
     \end{array}\right. .
\end{eqnarray*}

Hence we see that the filtration step $\cL_\lambda(\cX)^d$ vanishes or it is  irreducible.   Further we conclude
that the contribution  $v^{G}_{{P_{(j+1,1,\ldots,1)}}}(H^{d-j}(\bbP^d_K,\cE)')$ in \ref{extensions} vanishes for $j\neq 0.$
In the case $j=0$ this object coincides with $v^{G}_{B}(K) \otimes ({\rm Sym}^{-d-1-(r-s)}(K^{d+1})\otimes \det^{1-s})$
which is irreducible by \cite[app.]{ST4} resp. Thm. \ref{irredgeneral}.
Consequently, the filtration \ref{filtration} is essentially a Jordan-H\"older series of $H^0(\cX,\cL_\lambda)$, i.e.,
by refining the filtration in the case $r-s\leq d-1$ in the naive way we get a real Jordan-H\"older series.


\section{Appendix: some properties of highest weight modules in $\cO$}\label{HWM}

This section is about relations in a simple $U(\frg)$-module $M \in \cO^\frp_\alg$ of the form

$$y_\gamma^n \cdot v^+ = \sum_{\nu \in \cI_n} c_\nu y_{\beta_1}^{\nu_1} \cdot \ldots \cdot y_{\beta_t}^{\nu_t} \cdot v^+ \,,$$


\noindent where

- $\Phi^+ = \{\beta_1, \ldots, \beta_t\}$,

- $\cI_n$ consists of all $t$-tuples $(\nu_1, \ldots, \nu_t) \in \bbZ_{\ge 0}^t$ satisfying $\nu_1\beta_1 + \ldots +\nu_t \beta_t = n\gamma$,

- the elements $y_\gamma \in \frg_{-\gamma}$, $y_{\beta_i} \in \frg_{-\beta_i}$ are part of a Chevalley basis,

- $v^+$ is a highest weight vector for $M$,

- the standard parabolic subalgebra $\frp = \frp_I$ is maximal for $M$,

-  $\gamma \in \Phi^+ \setminus \Phi_I^+$,

- the coefficients $c_\nu$ are in $K$.

\vskip8pt

\noindent Our aim (cf. \ref{both_conditions}) is to show that there is at least one $\nu$ with $|c_\nu| \ge 1$ and $\nu_1 + \ldots + \nu_t \ge n$.

In particular, this result implies that $y_\gamma$ acts injectively on $M$ (take $c_\nu = 0$ for all $\nu$). In fact, we prove first this statement about the injectivity of the action of $y_\gamma$ (in \ref{injective}) and later use it to prove the more general result about the absolute value of (at least one) $c_\nu$.

That $y_\gamma$ acts injectively on $M$ (assuming $\gamma \notin \Phi_I$) is also contained in a paper by I. Dimitrov, O. Mathieu and 
I. Penkov, as was kindly pointed out to us by V. Mazorchuk, cf. \cite[3.6]{DMP}. Our proof is completely different from the proof given 
in loc.cit.

We begin with a standard commutator relation valid in any associative unital algebra $A$. For $x \in A$ we let $\ad(x): A \ra A$ be defined by $\ad(x)(z) = [x,z] = xz-zx$.
In the following we also write $[x^{(i)},z]$ for $\ad(x)^i(z)$, the value on $z$ of the $i$-th iteration of $\ad(x)$.

\begin{lemma}\label{generallemma} Let $x,z_1, \ldots, z_n \in A$. For all $k \in \bbZ_{\ge 0}$ one has

$$x^k \cdot z_1z_2 \ldots z_n = \sum_{
i_1 + \ldots + i_{n+1} = k} {k \choose i_1  \ldots  i_{n+1}} [x^{(i_1)},z_1]\cdot \ldots \cdot [x^{(i_n)},z_n]x^{i_{n+1}}$$


\noindent where, as usual,

$${k \choose i_1  \ldots  i_{n+1}} = \frac{k!}{(i_1!) \cdot \ldots \cdot (i_{n+1}!)}$$


\noindent Similarly,

$$[x^{(k)}, z_1z_2 \ldots z_n] = \sum_{
i_1 + \ldots + i_n = k} {k \choose i_1  \ldots  i_n} [x^{(i_1)},z_1]\cdot \ldots \cdot [x^{(i_n)},z_n] \;.$$
\end{lemma}

\noindent \Pf This is straightforwardly proved by induction on $k$. \qed

\vskip18pt

\begin{lemma}\label{lemma1} Let $x \in \frg$, let $M$ be a
$U(\frg)$-module and $v \in M$.

\vskip8pt

\noindent (i) If $x$ acts locally finitely on $v$ (i.e., the $K$-vector space generated by $(x^i.v)_{i \ge 0}$ is finite-dimensional), then $x$ acts locally finitely on $U(\frg).v$.

\vskip8pt

\noindent (ii) If $x.v = 0$ and $[x,[x,y]] = 0$ for some $y \in \frg$, then

$$x^ny^n.v = n! [x,y]^n.v \;.$$

\end{lemma}

\noindent \Pf (i) It suffices to show that $x$ acts locally finitely on any element of the form $z_1 \ldots z_n.v$ 
(for arbitrary $n$ and arbitrary elements $z_1, \ldots, z_n \in \frg$).  Any element of $U(\frg)$ of the 
form $[x^{(i_1)},z_1]\cdot \ldots \cdot [x^{(i_n)},z_n]$ is contained in $\frg \cdot \ldots \cdot \frg$ ($n$ factors), and $\frg \cdot \ldots \cdot \frg$ is a finite-dimensional $K$-subspace of $U(\frg)$. As the elements $x^i.v$, $i \ge 0$, are all contained in a finite-dimensional vector space, it follows from the formula 
in Lemma \ref{generallemma} that all elements $x^i \cdot z_1 \ldots z_n.v$ are contained in a finite-dimensional $K$-vector space.

\vskip8pt

(ii) In the formula of Lemma \ref{generallemma} we let $z_i = y$, $1 \le i \le n$, and get that the term

\vspace{-0.3cm}
$$[x^{(i_1)},y]\cdot \ldots \cdot [x^{(i_n)},y]x^{i_{n+1}}.v$$


\noindent vanishes as soon as $i_{n+1} > 0$ or some $i_j > 1$ for $1 \le j \le n$. Therefore, the only non-zero term corresponds to $(i_1,\ldots, i_{n+1}) = (1, \ldots, 1, 0)$. \qed

\vskip12pt

As before, we let $\Phi = \Phi(\frg,\frt)$ be the root system and

$$\Phi^+ = \{\beta_1, \ldots, \beta_t\} \hskip10pt \mbox{and} \hskip10pt \Delta = \{\alpha_1, \ldots, \alpha_\ell\} \sub \Phi^+$$


\noindent be a set of positive roots and a corresponding basis.

\begin{lemma}\label{lemma2} Assume $\gamma \in \Phi^+$ is not simple and write  $\gamma = \alpha + \beta$ with $\alpha \in \Delta$ and $\beta \in \Phi^+$. Suppose that $i\beta - j \alpha$ is in $\Phi^+$ for some $i,j \in \bbZ_{>0}$.
\vskip8pt

\noindent (i) Then $(i\beta - j\alpha) - \gamma  =  (i-1)\beta -(j+1)\alpha$ is {\bf either} a positive root {\bf or} not in $\Phi \cup \{0\}$. Therefore: {\bf either} $[\frg_{i\beta - j\alpha}, \frg_{-\gamma}]$ is equal to a root space $\frg_\chi$ with $\chi = i'\beta - j'\alpha \in \Phi^+$ for some $i',j' \in \bbZ_{>0}$ {\bf or} $[\frg_{i\beta - j\alpha}, \frg_{-\gamma}] = 0$.
\vskip8pt

\noindent (ii) Moreover, $(i\beta - j\alpha) - \alpha  =  i\beta -(j+1)\alpha$ is {\bf either} a positive root {\bf or} not in $\Phi \cup \{0\}$. Therefore: {\bf either} $[\frg_{i\beta - j\alpha}, \frg_{-\alpha}]$ is equal to a root space $\frg_\chi$ with $\chi = i'\beta - j'\alpha \in \Phi^+$ for some $i',j' \in \bbZ_{>0}$ {\bf or} $[\frg_{i\beta - j\alpha}, \frg_{-\alpha}] = 0$.
\vskip8pt

\noindent (iii) Let $M$ be a $U(\frg)$-module and $v \in M$ be annihilated by the radical $\fru$ of $\frb$. Let $x \in \frg_\beta$ and $y \in \frg_{-\gamma}$. Then, for any sequence of non-negative integers $i_1, \ldots, i_n$ we have

\vspace{-0.3cm}
$$[x^{(i_1)},y]\cdot \ldots \cdot [x^{(i_n)},y].v = 0$$


\noindent if there is at least one $i_j > 1$.
\end{lemma}

\noindent \Pf Assertions (i) and (ii) follow from the fact that $\beta$ is not a multiple of $\alpha$ (the root system $\Phi$ is reduced), and $\beta$ must then contain a simple root $\alpha' \neq \alpha$, i.e. $\beta - \alpha' \in \sum_{\tau \in \Delta} \bbZ_{\ge 0} \tau$.

\vskip8pt

(iii) We may clearly assume that $i_1 > 1$ and that all $i_j \in \{0,1\}$ for $j>1$.  We will show by induction on $n$ that $[x^{(i_1)},y]\cdot \ldots \cdot [x^{(i_n)},y]$ is contained in

\vspace{-0.3cm}
$$\sum_{
\begin{array}{c}
i>0, j>0 \\
i\beta - j\alpha \in \Phi^+
\end{array}} U(\frg)\frg_{i\beta-j\alpha}$$


\noindent To begin with, notice that $[x^{(i_1)},y]$ is in $\frg_{i_1\beta-\gamma} = \frg_{(i_1-1)\beta-\alpha}$, and as $i_1-1>0$ this space is either zero or a root space $\frg_\chi$ with $\chi \in \Phi^+$. The assertion is hence true for $n = 1$. Now let $n>1$, $i_n \in \{0,1\}$, and fix an element $z \in \frg_{i\beta-j\alpha}$ with $i>0$, $j >0$, and $i\beta - j\alpha \in \Phi^+$. Consider the product $z[x^{(i_n)},y]$.

\vskip8pt

(a) If $i_n = 0$ then $[x^{(i_n)},y] = y \in \frg_{-\gamma}$ and $z[x^{(i_n)},y] = zy = [z,y] + yz$. If $i\beta-j\alpha - \gamma  = (i-1)\beta - (j+1)\alpha$ is a positive root, we are done, because then $[z,y]$ is in $\frg_{i'\alpha-j'\beta}$ with $i'\alpha-j'\beta \in \Phi^+$. Otherwise, by (i), $i\beta-j\alpha - \gamma$ is not in $\Phi \cup \{0\}$, and then $[z,y] = 0$.

\vskip8pt

(b) If $i_n = 1$ then $[x^{(i_n)},y] = [x,y] \in \frg_{-\alpha}$ and $z[x^{(i_n)},y] = z[x,y] = [z,[x,y]] + [x,y]z$. If $i\beta-j\alpha - \alpha  = i\beta - (j+1)\alpha$ is a positive root, we are done, because then $[z,[x,y]]$ is in $\frg_{i'\alpha-j'\beta}$ with $i'\alpha-j'\beta \in \Phi^+$. Otherwise, by (ii), $i\beta-j\alpha - \alpha$ is not in $\Phi \cup \{0\}$, and then $[z,[x,y]] = 0$. \qed

\begin{para}\label{lexorder} In the remainder of this section we will use the following lexicographic ordering on \linebreak
$\bbZ\alpha_1 \oplus \ldots \oplus \bbZ\alpha_\ell \;$:

$$\sum_{i=1}^\ell n_i \alpha_i > \sum_{i=1}^\ell n_i' \alpha_i \hskip10pt \iff \hskip10pt \exists k \ge 1: n_i = n_i' \mbox{ for } 1 \le i \le k-1 \mbox{ and } n_k > n_k' \;.$$

\vskip8pt
\end{para}

\begin{prop}\label{notlocnilp} Let $\frp = \frp_I$ for some $I \sub \Delta$. Suppose $M \in \cO^\frp$ is a highest weight module with highest weight $\lambda$ and

\vspace{-0.3cm}
$$I = \{\alpha \in \Delta \midc \langle \lambda, \alpha^\vee \rangle \in \bbZ_{\ge 0} \} \;.$$


\noindent Then no non-zero element of $\fru_\frp^-$ acts locally finitely on $M$.
\end{prop}

\noindent \Pf Let $v^+$ be a weight vector with weight $\lambda$. Let $y \in \fru^-_\frp$ be a non-zero element which acts locally finitely on $M$. Write $y = \sum_{\gamma \in \Phi^+ \setminus \Phi^+_I} y_\gamma$ with elements $y_\gamma \in \frg_{-\gamma}$. Put $B = \{\gamma \in \Phi^+ \setminus \Phi^+_I \midc y_\gamma \neq 0 \}$ (this set is non-empty) and choose $\gamma^+ \in B$ to be maximal among the elements in $B$ for the lexicographic ordering \ref{lexorder}. Write

\vspace{-0.3cm}
$$y^n.v^+ = \sum_{1 \le i_1, \ldots, i_n \le \tau} y_{\gamma_{i_1}}\cdot \ldots \cdot y_{\gamma_{i_n}}.v^+ \,.$$


\noindent where $\Phi^+ \setminus \Phi^+_I = \{\gamma_1, \ldots, \gamma_\tau\}$. Using the total ordering \ref{lexorder} it is easily seen that among the elements $y_{\gamma_{i_1}}\cdot \ldots y_{\gamma_{i_n}}.v^+$ only $y_{\gamma^+}^n.v^+$ can have weight $\lambda-n\gamma^+$. But as all $y^i.v^+$, $i \ge 0$, are contained in a finite-dimensional subspace, we must have $y_{\gamma^+}^n.v^+ = 0$ for some $n \in \bbZ_{>0}$. It follows from Lemma \ref{lemma1} that $y_{\gamma^+}$ acts then locally finitely on $M$.


We can therefore assume that $y = y_\gamma \in \frg_{-\gamma}\setminus \{0\}$ is contained in a root space, where $\gamma \in \Phi^+ \setminus \Phi_I$. Write $\gamma = \sum_{\alpha \in \Delta} c_\alpha \alpha$ (with non-negative integers $c_\alpha$). We show by induction on the height of $\gamma$, $ht(\gamma) = \sum_{\alpha \in \Delta} c_\alpha$, that $y_\gamma$ can not act locally finitely. By Lemma \ref{lemma1} this is equivalent to the statement that $y_\gamma^n.v^+ \neq 0$ for all positive integers $n$. (Note that the vectors $y_\gamma^n \cdot v^+$, if non-zero, have pairwise distinct weights $\lambda - n \cdot \gamma$, hence are linearly independent.) $y_\gamma^n.v^+$. If $ht(\gamma) = 1$, then $\gamma$ is an element of $\Delta \setminus I$. Rescaling $y_\gamma$ we can choose $x_\gamma \in \frg_\gamma$ such that $[x_\gamma,y_\gamma] = h_\gamma$ and $[h_\gamma,x_\gamma] = 2x_\gamma$ and $[h_\gamma,y_\gamma] = -2y_\gamma$. A well-known formula (which is easy to prove by in
 duction) gives

\vspace{-0.3cm}
\begin{numequation}\label{heightone}
x_\gamma^ny_\gamma^n.v^+ = n!\prod_{i=0}^{n-1}(\lambda(h_\gamma) - i).v^+ = n!\prod_{i=0}^{n-1}(\langle \lambda, \gamma^\vee \rangle - i).v^+ \,.
\end{numequation}


\noindent As $I = \{\alpha \in \Delta \midc \langle \lambda, \alpha^\vee \rangle \in \bbZ_{\ge 0} \}$, it follows that $\langle \lambda, \gamma^\vee \rangle \notin \bbZ_{\ge 0}$ and the term on the right of \ref{heightone} does not vanish. In particular, $y_\gamma^n.v^+ \neq 0$ for all $n \ge 0$.


Now suppose $ht(\gamma)>1$. Then we can write $\gamma = \alpha + \beta$ with $\alpha \in \Delta$ and $\beta \in \Phi^+$. Clearly, not both $\alpha$ and $\beta$ can be contained in $\Phi_I$. We distinguish two cases.


(a) Suppose $\beta - \alpha$ is not in $\Phi$. As $\beta \neq \alpha$ (the root system $\Phi$ is reduced) we have $[\frg_\alpha,\frg_{-\beta}] = [\frg_{-\alpha},\frg_\beta] = \{0\}$. Then, if $\alpha \notin I$, we let $x_\beta$ be a non-zero element of $\frg_\beta$ and have by Lemma \ref{lemma1}:

\vspace{-0.3cm}
$$x_\beta^ny_\gamma^n.v^+ = n![x_\beta,y_\gamma]^n.v^+ \,.$$


\noindent But as $[x_\beta,y_\gamma]$ is a non-zero element of $\frg_{-\alpha}$ we can conclude by induction that $[x_\beta,y_\gamma]^n.v^+ \neq 0$ for all $n \ge 0$. And thus $y_\gamma^n.v^+ \neq 0$ for all $n \ge 0$.

 If, on the other hand, $\alpha \in I$, then $\beta \notin \Phi_I$. Let $x_\alpha$ be a non-zero element of $\frg_\alpha$. Then we have by Lemma \ref{lemma1}:

 \vspace{-0.3cm}
$$x_\alpha^ny_\gamma^n.v^+ = n![x_\alpha,y_\gamma]^n.v^+ \,.$$


\noindent And as $[x_\alpha,y_\gamma]$ is a non-zero element of $\frg_{-\beta}$ we can again conclude by induction.


(b) Suppose $\beta - \alpha$ is in $\Phi$. Then it must be in $\Phi^+$, and we have $\gamma - k\alpha \in \Phi^+$ for $0 \le k \le k_0$ (with $k_0 \le 3$, cf. \cite[0.2]{H1}), and $\gamma - k\alpha \notin \Phi \cup \{0\}$ for $k > k_0$. This implies $[x_\alpha^{(i)},y_\gamma] = 0$ for $i>k_0$. By Lemma \ref{generallemma} we have

$$x_\alpha^{nk_0}y_\gamma^n.v^+ = \sum_{
i_1 + \ldots + i_{n+1} = nk_0} {nk_0 \choose i_1  \ldots  i_n \, i_{n+1}} [x_\alpha^{(i_1)},y_\gamma]\cdot \ldots \cdot [x_\alpha^{(i_n)},y_\gamma]x_\alpha^{i_{n+1}}.v^+ $$


\noindent and as $x_\alpha$ annihilates $v^+$ this reduces to

$$\sum_{i_1 + \ldots + i_n = nk_0} {nk_0 \choose i_1  \ldots  i_n} [x_\alpha^{(i_1)},y_\gamma]\cdot \ldots \cdot [x_\alpha^{(i_n)},y_\gamma].v^+
$$


\noindent By what we have just observed, the corresponding term vanishes if there is one $i_j>k_0$. Therefore, only the term with all $i_j = k_0$ contributes, and this sum is hence equal to

$${nk_0 \choose k_0  \ldots  k_0} [x_\alpha^{(k_0)},y_\gamma]^n.v^+ = \frac{(nk_0)!}{(k_0!)^n} \; [x_\alpha^{(k_0)},y_\gamma]^n.v^+$$

\vskip8pt

\noindent If $\gamma - k_0\alpha$ is not in $\Phi_I$ we are done, because $[x_\alpha^{(k_0)},y_\gamma]$ is a non-zero element of $\frg_{-(\gamma-k_0\alpha)}$. Otherwise we necessarily have $\alpha \notin I$. In this case, if we choose some $x_\beta \in \frg_\beta \setminus \{0\}$, we have by Lemma \ref{generallemma} and Lemma \ref{lemma2} (iii)

\vspace{-0.3cm}
$$x_\beta^ny_\gamma^n.v^+ = n![x_\beta,y_\gamma]^n.v^+ \,,$$


\noindent and $[x_\beta,y_\gamma]$ is a non-zero element of $\frg_{-\alpha}$. And, as we are now in the case of height one, we can thus conclude again. \qed

\begin{cor}\label{injective} Let $\frp = \frp_I$ for some $I \sub \Delta$. Suppose $M \in \cO^\frp$ is a simple module of highest weight $\lambda$ and

\vspace{-0.3cm}
$$I = \{\alpha \in \Delta \midc \langle \lambda, \alpha^\vee \rangle \in \bbZ_{\ge 0} \} \;.$$


\noindent Then the action of any non-zero element of $\fru_\frp^-$ on $M$ is injective.
\end{cor}


\noindent \Pf By Lemma \ref{lemma1} (i), the set $N$ of elements $v \in M$ on which a fixed element $\frx \in \fru_\frp^- \setminus \{0\}$ acts locally finitely is a $U(\frg)$-submodule. Clearly $N$ contains $\ker(M \stackrel{\frx \, \cdot}{\lra} M)$. Because $M$ is assumed to be simple, and as $\frx$ is not acting locally finitely on $M$ by Prop. \ref{notlocnilp}, this submodule must be the zero module. \qed


\begin{cor}\label{locfinitesubalgebra} Let $M \in \cO$ be a highest weight module. Then the set of elements in $\frg$ which act locally finitely on $M$ is a standard parabolic Lie subalgebra of $\frg$. If $M$ has highest weight $\lambda$, then this standard parabolic subalgebra is $\frp_I$ where

\vspace{-0.3cm}
$$I = \{\alpha \in \Delta \midc \langle \lambda, \alpha^\vee \rangle \in \bbZ_{\ge 0} \} \;.$$


\noindent $\frp_I$ is maximal for $M$ in the sense of \ref{maximal}.
\end{cor}


\noindent \Pf Let $v^+$ be a weight vector of weight $\lambda$ and define $I$ as above. Then $\lambda$ is in

\vspace{-0.3cm}
$$\Lambda_I^+ = \{\mu \in \frt^* \midc \forall \alpha \in I: \langle \mu , \alpha^\vee \rangle \in \bbZ_{\ge 0} \}$$


\noindent (cf. \cite[sec. 9.2]{H1} for the notation) and $M$ is in $\cO^\frp$ with $\frp = \frp_I$, cf. \cite[Thm. in sec. 9.4]{H1}. Suppose $z \in \frg \setminus \frp$ acts locally finitely on $M$. Then there is $n \in \bbZ_{>0}$ and $c_1, \ldots, c_n \in K$ such that

\vspace{-0.3cm}
\begin{numequation}\label{finiteness}
z^n.v^+ + c_1 z^{n-1}.v^+ + \ldots + c_{n-1}z.v^+ + c_n.v^+ = 0 \,.
\end{numequation}


\vspace{-0.3cm}
\noindent Write $z = x + \sum_{\gamma \in \Phi^+ \setminus \Phi_I} y_\gamma$ with elements $y_\gamma \in \frg_{-\gamma}$ and $x \in \frp$. As $z \notin \frp$ there is at least one $\gamma \in \Phi^+ \setminus \Phi_I$ with $y_\gamma \neq 0$. Put $B = \{\gamma \in \Phi^+ \setminus \Phi_I \midc y_\gamma \neq 0 \}$ (this set is non-empty) and choose $\beta \in B$ to be maximal among the elements in $B$ for the lexicographic ordering \ref{lexorder}. By expanding $z^n$ as a sum of products of $x$'s and $y_\gamma$'s it is then easily seen that only $y_\beta^n.v^+$ can have weight $\lambda-n\beta$. From equation \ref{finiteness} we deduce that $y_\beta^n.v^+ = 0$. This however contradicts Prop. \ref{notlocnilp}. \qed

\vskip8pt

\noindent {\it On certain relations in highest weight modules.} For the rest of this section we fix a Chevalley basis $(x_\beta,y_\beta,h_\alpha \midc \beta \in \Phi^+, \alpha \in \Delta)$ of $\frg' = [\frg,\frg]$, cf. \cite[Thm. in sec. 25.2]{H2}. Here we have $x_\beta \in \frg_\beta$ and $y_\beta \in \frg_{-\beta}$. We then let $\frg'_\bbZ$ be the $\bbZ$-span of these basis elements. $\frg'_\bbZ$ is a Lie algebra over $\bbZ$. For later use we quote from \cite[sec. 25]{H2} the following facts.

\vskip8pt

\begin{prop}\label{Chevalley} Let $\beta$, $\beta'$ be linearly independent roots. Let $z_\beta$ be $x_\beta$, if $\beta$ is positive, and let $z_\beta$ be $y_{-\beta}$, if $\beta$ is negative. Define $z_{\beta'}$ and $z_{\beta+\beta'}$ in the same manner, if $\beta + \beta'$ is a root.
Let $\beta'-r\beta, \beta'-(r-1)\beta, \ldots, \beta'+q\beta$ ($r \ge 0$, $q \ge 0$) be the $\beta$-string through $\beta'$. Then: \vskip5pt

(i) $r-q = \langle \beta, (\beta')^\vee \rangle$.

(ii) $[z_\beta,z_{\beta'}] = \pm(r+1)z_{\beta+\beta'}$ if $\beta + \beta'$ is a root.
\end{prop}

\noindent \Pf This is contained in \cite[Prop. in sec. 25.1, Thm. in sec. 25.2]{H2}. \qed

\begin{para}\label{intro} Let $M$ be a simple highest weight module in the category $\cO^\frp_\alg$, and assume that $\frp = \frp_I$ is maximal for $M$. Denote by $v^+$ a vector of highest weight $\lambda$. We will be studying relations

\vspace{-0.3cm}
$$y_\gamma^n.v^+ = \sum_{\nu \in \cI_n} c_\nu y_{\beta_1}^{\nu_1} \cdot \ldots \cdot y_{\beta_t}^{\nu_t}.v^+ \;,$$


\noindent where $\cI_n$ consists of all tuples $\nu = (\nu_1, \ldots, \nu_t) \in \bbZ_{\ge 0}^t$ satisfying $\nu_1\beta_1 + \ldots + \nu_t\beta_t = n\gamma$, and $c_\nu$ are coefficients in $K$.  As before, $\Phi^+ = \{\beta_1, \ldots, \beta_t\}$. Our aim is to show that there is at least one $\nu \in \cI_n$ having both of the following properties:

\vspace{-0.3cm}
$$\begin{array}{ll}
i. & \nu_1 + \ldots + \nu_t \ge n \,, \\
& \\
ii. & |c_\nu|_K \ge 1 \hskip4pt .
\end{array}$$


\noindent We start by showing the existence of some $\nu \in \cI_n$ with the second property.
\end{para}

\setcounter{enumi}{0}

\begin{lemma}\label{estimate} Let $M
\in \cO^\frp_\alg$ and $\frp = \frp_I$ be as above \ref{intro}. Suppose the residue characteristic of $K$ does not divide any of the non-zero numbers among $\langle \beta, \alpha^\vee\rangle$, $\alpha, \beta \in \Phi$, $\alpha \neq \pm \beta$. Let $\gamma \in \Phi^+ \setminus \Phi^+_I$. Denote by $\cI_n$ the set of all $\nu \in \bbZ_{\ge 0}^r$ such that $\nu_1\beta_1 + \ldots + \nu_r\beta_r = n\gamma$. Then, for any $n \in \bbZ_{\ge 0}$ and any expression

\vspace{-0.3cm}
\begin{numequation}\label{relation}
y_{\gamma}^n.v^+ = \sum_{\nu \in \cI_n} c_\nu y_1^{\nu_1} \cdot \ldots \cdot y_t^{\nu_t}.v^+ \hskip4pt ,
\end{numequation}


\noindent (where $y_i = y_{\beta_i}$, $i = 1, \ldots, t$) there is at least one index
$\nu \in \cI_n$ such that $|c_\nu|_K \ge 1$.
\end{lemma}

\vskip12pt

\noindent \Pf We start by recalling that for any $i \ge 0$ and $\beta \in \Phi^+$ the endomorphism of $\frg$:

$$\frac{1}{i!}[x_\beta^{(i)}, \cdot ] = \frac{1}{i!}\ad(x_\beta)^i$$

\vskip8pt

\noindent preserves the $\bbZ$-form $\frg'_\bbZ$ of $\frg'$, cf. \cite[Prop. in sec. 25.5]{H2}. Denote by $U(\frg_\bbZ')$ the enveloping algebra of $\frg'_\bbZ$, i.e., the quotient of the tensor algebra $T_\bbZ(\frg'_\bbZ)$ by the two-sided ideal generated by all elements of the form $xy-yx-[x,y]$ with $x,y \in \frg'_\bbZ$. (We point out that $U(\frg'_\bbZ)$ is, of course, {\it not} the Kostant $\bbZ$-form of the enveloping algebra, as in \cite[sec. 26.4]{H2}.) It follows from \ref{generallemma} that $\frac{1}{i!}\ad(x_\beta)^i$ also preserves $U(\frg'_\bbZ)$.

The proof proceeds by induction on $ht(\gamma)$. Suppose $ht(\gamma) = 1$ and let $\beta_{i} = \gamma$. Then the set $\cI_n$ consists of a single element $\nu$ which is the $t$-tuple that has the entry $n$ in the $i^{\mbox{\tiny th}}$ place and zeros elsewhere. The right hand side of \ref{relation} is thus $c_\nu y_{\gamma}^n.v^+$. By Cor. \ref{injective} the element $y_{\gamma}$ acts injectively on $M$, and we therefore get $c_\nu = 1$.

\vskip8pt

Now we assume that $ht(\gamma) > 1$. Write $\gamma = \alpha + \beta$ with a simple root $\alpha \in \Delta$ and a positive root $\beta$.
Clearly, not both $\alpha$ and $\beta$ can be contained in $\Phi_I$. We distinguish two cases.

\vskip8pt

(a) Suppose $\beta - \alpha$ is not in $\Phi$. As $\beta \neq \alpha$ (the root system $\Phi$ is reduced) we 
have $[\frg_\alpha,\frg_{-\beta}] = [\frg_{-\alpha},\frg_\beta] = \{0\}$. Then, if $\alpha \notin I$, we consider $x_\beta$, the element 
of the Chevalley basis which generates $\frg_\beta$. We have by Lemma \ref{lemma1}:

\vspace{-0.3cm}
$$x_\beta^ny_\gamma^n.v^+ = n![x_\beta,y_\gamma]^n.v^+ \;.$$


\noindent Consider the $\beta$-string through $-\gamma$:  $-\gamma-r\beta, \ldots, -\gamma + q\beta$. Then  $-\alpha-(r+1)\beta, \ldots, -\alpha+(q-1)\beta$ is the $\beta$-string through $-\alpha$, and because we assume here that $-\alpha+\beta \notin \Phi$, we deduce that $q=1$. By \ref{Chevalley} we then conclude $r+1 = \langle -\alpha, \beta^\vee \rangle$ and $[x_\beta,y_\gamma] = \pm \langle \alpha, \beta^\vee \rangle y_\alpha$. Hence

\vspace{-0.3cm}
$$x_\beta^ny_\gamma^n.v^+ = n!( \pm \langle \alpha, \beta^\vee \rangle)^n  y_\alpha^n.v^+ \;.$$


\noindent As $\langle \alpha, \beta^\vee \rangle = -(r+1) \neq 0$ and because of our assumption on the residue characteristic of $K$, the integer $\langle \alpha, \beta^\vee \rangle$ is invertible in $O_K$.

On the other hand, equation \ref{relation} gives

\vspace{-0.3cm}
$$x_\beta^n.y_\gamma^n.v^+ = \sum_{\nu \in \cI_n} c_\nu x_\beta^n(y_1^{\nu_1} \cdot \ldots \cdot y_t^{\nu_t}).v^+ = \sum_{\nu \in \cI_n} c_\nu \ad(x_\beta)^n(y_1^{\nu_1} \cdot \ldots \cdot y_t^{\nu_t}).v^+ \; .$$


\noindent But as $\ad(x_\beta)^n(y_1^{\nu_1} \cdot \ldots \cdot y_t^{\nu_t})$ is in $n!U(\frg'_\bbZ)$, and because $h_\tau.v^+ = \langle \lambda, \tau^\vee \rangle v^+ \in \bbZ v^+$ (for all $\tau \in \Delta$), we find that $x_\beta^n.y_\gamma^n.v^+$
is of the form

\vspace{-0.3cm}
$$n! \sum_{\nu' \in \cI'_n} c'_{\nu'} y_1^{\nu_1'} \cdot \ldots \cdot y_t^{\nu_t'}.v^+ \; ,$$


\noindent where $\cI'_n$ consists of all $\nu' \in \bbZ_{\ge 0}^t$ such that $\nu_1'\beta_1 + \ldots + \nu_t'\beta_t = n\alpha = n(\gamma - \beta)$, and
the numbers $c'_{\nu'}$ are linear combinations of the $c_\nu$ with integral coefficients. We therefore get

\vspace{-0.3cm}
$$y_\alpha^n.v^+ = \frac{1}{(\pm \langle \alpha, \beta^\vee \rangle)^n} \sum_{\nu' \in \cI'_n} c'_{\nu'} y_1^{\nu_1'} \cdot \ldots \cdot y_t^{\nu_t'}.v^+ \; .$$


\noindent The induction hypothesis shows that at least one of the coefficients $c'_{\nu'}$ must be of absolute value at least $1$, and this implies that at least one of the coefficients $c_\nu$ is of absolute value at least $1$.

\vskip8pt

Now suppose $\alpha \in I$. Then $\beta \notin \Phi_I$ and we consider $x_\alpha$, the member of the Chevalley basis which generates $\frg_\alpha$. By Lemma \ref{lemma1}:

\vspace{-0.3cm}
$$x_\alpha^ny_\gamma^n.v^+ = n![x_\alpha,y_\gamma]^n.v^+ \;.$$


\noindent The same arguments as above give $[x_\alpha,y_\gamma] = \pm \langle \beta, \alpha^\vee \rangle y_\beta$, and $\langle \beta, \alpha^\vee \rangle$ is a unit in $O_K$. As before, we then multiply the right hand side of \ref{relation} with $x_\alpha^n$ and find that

$$y_\beta^n.v^+ = \frac{1}{(\pm \langle \beta, \alpha^\vee \rangle)^n} \sum_{\nu' \in \cI'_n} c'_{\nu'} y_1^{\nu_1'} \cdot \ldots \cdot y_t^{\nu_t'}.v^+ \; ,$$

\vskip8pt

\noindent where the numbers $c'_{\nu'}$ are linear combinations of the $c_\nu$ with integral coefficients. And we conclude again by induction.

\vskip8pt

(b) Suppose $\beta - \alpha$ is in $\Phi$. Then it must be in $\Phi^+$, and we have $\gamma - k\alpha \in \Phi^+$ for $0 \le k \le k_0$ (with $k_0 \le 3$, cf. \cite[0.2]{H1}), and $\gamma - k\alpha \notin \Phi \cup \{0\}$ for $k > k_0$. This implies $[x_\alpha^{(i)},y_\gamma] = 0$ for $i>k_0$. By Lemma \ref{generallemma} we have

\vspace{-0.3cm}
$$x_\alpha^{nk_0}y_\gamma^n.v^+ = \sum_{
i_1 + \ldots + i_{n+1} = nk_0} {nk_0 \choose i_1  \ldots  i_n \, i_{n+1}} [x_\alpha^{(i_1)},y_\gamma]\cdot \ldots \cdot [x_\alpha^{(i_n)},y_\gamma]x_\alpha^{i_{n+1}}.v^+ $$


\noindent and as $x$ annihilates $v^+$ this reduces to

\vspace{-0.3cm}
$$\sum_{i_1 + \ldots + i_n = nk_0} {nk_0 \choose i_1  \ldots  i_n} [x_\alpha^{(i_1)},y_\gamma]\cdot \ldots \cdot [x_\alpha^{(i_n)},y_\gamma].v^+
$$


\noindent By what we have just observed, the corresponding term vanishes if there is one $i_j>k_0$. Therefore, only the term with all $i_j = k_0$ contributes, and this sum is hence equal to

$${nk_0 \choose k_0  \ldots  k_0 } [x_\alpha^{(k_0)},y_\gamma]^n.v^+ = \frac{(nk_0)!}{(k_0!)^n} \; [x_\alpha^{(k_0)},y_\gamma]^n.v^+$$

\vskip8pt

\begin{sublemma} $[x_\alpha^{(k_0)},y_\gamma] = k_0! \cdot c \cdot y_{\gamma - k_0\alpha}$ with an integer $c$ which is a unit in $O_K$.
\end{sublemma}

\noindent \Pf Suppose $k_0 = 2$. Let $\gamma - 2\alpha, \ldots, \gamma +q\alpha$ be the $\alpha$-string through $\gamma$. 
As such a string consists of at most four roots (cf. \cite[0.2]{H1}) we have $q \le 1$. The $\alpha$-string through $-\gamma$ then 
begins with $-\gamma-q\alpha$. Assume first that $q=0$. By \ref{Chevalley} we have $[x_\alpha,y_\gamma] = \pm y_{\gamma-\alpha}$ and 
then $[x_\alpha,y_{\gamma-\alpha}] = \pm 2y_{\gamma-2\alpha}$, so indeed $[x_\alpha^{(2)},y_\gamma] = \pm 2! \cdot y_{\gamma - 2\alpha}$. 
If $q=1$ then there is a string of length four, and $\Phi$ must contain an irreducible component of type $G_2$. By \ref{hyp} we 
then assume that $3$ is invertible in $O_K$. Moreover, we have $[x_\alpha,y_\gamma] = \pm 2 y_{\gamma-\alpha}$ and 
then $[x_\alpha,y_{\gamma-\alpha}] = \pm 3y_{\gamma-2\alpha}$, so that indeed $[x_\alpha^{(2)},y_\gamma] = \pm 2! \cdot 3 \cdot y_{\gamma - 2\alpha}$ with $3 \in O_K^*$.

Suppose $k_0=3$. Then $\gamma - 3\alpha, \gamma-2\alpha, \gamma -\alpha, \gamma$ is the $\alpha$-string through $\gamma$ and $-\gamma, -\gamma+\alpha, -\gamma +2\alpha, -\gamma+3\alpha$ is the $\alpha$-string through $-\gamma$. Using \ref{Chevalley} we compute $[x_\alpha^{(3)},y_\gamma] = \pm 3! y_{\gamma-3\alpha}$. \qed

We conclude that

\vspace{-0.3cm}
\begin{numequation}\label{conclusion}
x_\alpha^{nk_0}y_\gamma^n.v^+ = (nk_0)! \cdot c^n \cdot (y_{\gamma - k_0 \alpha})^n . v^+ \hskip8pt \mbox{with} \hskip8pt c \in O_K^*\;.
\end{numequation}


\noindent If $\gamma - k_0\alpha$ is not in $\Phi_I$ we are done, by the same reasoning as in part (a). Otherwise we necessarily have $\alpha \notin I$. In this case we have by Lemma \ref{generallemma} and Lemma \ref{lemma2} (iii)

\vspace{-0.3cm}
$$x_\beta^ny_\gamma^n.v^+ = n![x_\beta,y_\gamma]^n.v^+ \,.$$


\noindent Let $-\gamma-r\beta, \ldots, -\gamma+q\beta$ be the  $\beta$-string through $-\gamma$.  
We have $q \ge 2$ and thus $r \le 1$. 
If $r=0$ then $[x_\beta,y_\gamma] = \pm y_\alpha$. If $r=1$ then $[x_\beta,y_\gamma] =\pm 2 y_{\alpha}$. 
But in this case the $\beta$-string through $-\gamma$ consists of four roots, and $\Phi$ must have an irreducible component of type $G_2$. 
By \ref{hyp} we have $p>3$, and thus $2 \in O_K^*$. Therefore, we always have $[x_\beta,y_\gamma] = c y_\alpha$ with $c \in O_K^*$. 
With the same arguments as in part (a) we can thus deduce the claim. \qed

\vskip12pt

\setcounter{enumi}{0}

\begin{lemma}\label{ABCD}
Suppose that none of the irreducible components of $\Phi$ is of type $G_2$. Let $\gamma \in \Phi^+$. Consider a relation
\begin{numequation}\label{relation2}
n \gamma = \nu_1\beta_1 + \ldots + \nu_t\beta_t
\end{numequation}
\noindent with non-negative integers $n, \nu_1, \ldots, \nu_t \in \bbZ_{\ge 0}$. Then $n \le \nu_1 + \ldots + \nu_t$.
\end{lemma}

\vskip5pt

\noindent \Pf We may assume that $\Phi$ is irreducible and that $n$ is positive (there is nothing to show when $n=0$). It is known that for irreducible reduced roots systems other than $G_2$ the square of the ratio between the lengths of any two roots $\beta, \gamma$ is among $\{\frac{1}{2},1,2\}$. Root systems of type $A$, $D$, $E$ (so-called simply laced root systems) have the property that its roots are all of equal length, whereas for root systems of type $B$, $C$, and $F_4$ the ratio can also be $\frac{1}{2}$ or $2$. The relations

\vspace{-0.3cm}
$$\langle \beta,\gamma \rangle = 2\frac{(\beta,\gamma)}{(\gamma,\gamma)} = 2 \frac{\|\beta\|}{\|\gamma\|} \cos(\theta) \hskip12pt \mbox{ and } \hskip12pt \langle \beta,\gamma \rangle \langle \gamma,\beta \rangle = 4 \cos(\theta)^2$$


\noindent show that if $\frac{\|\beta\|^2}{\|\gamma\|^2} \in \{\frac{1}{2},1,2\}$ then $\langle \beta,\gamma \rangle \le 2$, i.e., $(\beta,\gamma) \le (\gamma,\gamma)$, cf. \cite[9.4]{H2}. Taking the scalar product of both sides of \ref{relation2} with $\gamma$ and dividing by $(\gamma,\gamma)$ we get

\vspace{-0.3cm}
$$n = \nu_1\frac{(\beta_1,\gamma)}{(\gamma,\gamma)} + \ldots + \nu_t\frac{(\beta_t,\gamma)}{(\gamma,\gamma)} \le \nu_1 + \ldots + \nu_t \,.$$


\noindent And this is what we asserted. \qed

\vskip12pt

Now we generalize the preceding lemma so as to assure the existence of some $\nu \in \cI_n$ satisfying both conditions $i.$ and $ii.$ of \ref{intro}.

\setcounter{enumi}{0}

\begin{prop}\label{both_conditions} Let $M \in \cO^\frp_\alg$ and $\frp = \frp_I$ be as in \ref{intro}. Suppose the residue characteristic of $K$ does not divide any of the non-zero numbers among $\langle \beta, \alpha^\vee\rangle$, $\alpha, \beta \in \Phi$, $\alpha \neq \pm \beta$. Let $\gamma \in \Phi^+ \setminus \Phi^+_I$. Denote by $\cI_n$ the set of all $\nu \in \bbZ_{\ge 0}^t$ such that $\nu_1\beta_1 + \ldots + \nu_t\beta_t = n\gamma$. Then, for any $n \in \bbZ_{\ge 0}$ and any expression

\vspace{-0.3cm}
\begin{numequation}\label{relation3}
y_{\gamma}^n.v^+ = \sum_{\nu \in \cI_n} c_\nu y_1^{\nu_1} \cdot \ldots \cdot y_t^{\nu_t}.v^+ \hskip4pt ,
\end{numequation}


\noindent (where $y_i = y_{\beta_i}$, $i = 1, \ldots, t$) there is at least one index
$\nu \in \cI_n$ such that $\nu_1 + \ldots + \nu_t \ge n$ and $|c_\nu|_K \ge 1$.
\end{prop}

\vskip5pt

\noindent \Pf If $\Phi$ does not have an irreducible component of type $G_2$ then \ref{estimate} and \ref{ABCD} prove the assertion of \ref{both_conditions}. Therefore we assume in the following that $\Phi$ is irreducible of type $G_2$.

The basic idea of the proof is as follows. There is nothing to show if $ht(\gamma) = 1$. Now suppose that $ht(\gamma) > 1$. Then there is $\gamma' \in \Phi^+$ and $k_0 \in \bbZ_{>0}$ such that $\gamma - k_0\gamma' \in \Phi^+ \setminus \Phi^+_I$ and

$$\frac{1}{(nk_0)!} x_{\gamma'}^{nk_0} \cdot y_\gamma^n.v^+ = \frac{1}{(k_0!)^n} \cdot [x_{\gamma'}^{(k_0)}, y_\gamma]^n.v^+ = c \cdot (y_{\gamma-k_0\gamma'})^n.v^+ \;,$$

\vskip8pt

\noindent with $c \in O_K^*$, cf. \ref{conclusion}. Considering the right hand side of \ref{relation3}, we aim to show that for any $\nu \in \cI_n$ with $\nu_1 + \ldots + \nu_t < n$ the term

\vspace{-0.3cm}
\begin{numequation}\label{vanishing_term}
x_{\gamma'}^{nk_0} \cdot y_1^{\nu_1} \cdot \ldots \cdot y_t^{\nu_t}.v^+
\end{numequation}

\vspace{-0.3cm}
\noindent in
\begin{numequation}\label{modified_sum}
\frac{1}{(nk_0)!} x_{\gamma'}^{nk_0} \cdot \left(\sum_{\nu \in \cI_n} c_\nu y_1^{\nu_1} \cdot \ldots \cdot y_t^{\nu_t}.v^+\right) = \sum_{\nu \in \cI_n} c_\nu \cdot \frac{1}{(nk_0)!} x_{\gamma'}^{nk_0} \cdot y_1^{\nu_1} \cdot \ldots \cdot y_t^{\nu_t}.v^+
\end{numequation}


\noindent vanishes. This means that the sum on the right hand side of \ref{modified_sum} is equal to

\vspace{-0.3cm}
\begin{numequation}\label{modified_sum2}
\sum_{\nu \in \cJ_n} c_\nu \cdot \frac{1}{(nk_0)!}x_{\gamma'}^{nk_0} \cdot y_1^{\nu_1} \cdot \ldots \cdot y_t^{\nu_t}.v^+ \;,
\end{numequation}


\noindent where $\cJ_n \sub \cI_n$ consists only of those $\nu \in \cI_n$ for which $\nu_1 + \ldots + \nu_t \ge n$. We can then rewrite \ref{modified_sum2} as

\vspace{-0.3cm}
$$\sum_{\nu' \in \cI'_n} c'_{\nu'} \cdot y_1^{\nu_1'} \cdot \ldots \cdot y_t^{\nu_t'}.v^+ \;,$$


\noindent where $\cI'_n$ consists of all $\nu' \in \bbZ_{\ge 0}^t$ such that $\nu_1'\beta_1 + \ldots + \nu_t' \nu_t' = n(\gamma - k_0\gamma')$, and the numbers $c'_{\nu'}$ are linear combinations with integral coefficients of the numbers $c_\nu$, with $\nu \in \cJ_n$. (Recall that the operator $\frac{1}{(nk_0)!}{\rm ad}(x_{\gamma'})^{nk_0}$ preserves $\frg_\bbZ'$.) Applying Lemma \ref{estimate}, we find that there is $\nu' \in \cI_n'$ such that $|c'_{\nu'}|_K \ge 1$. Hence there is at least one $\nu \in \cJ_n$ such that $|c_\nu|_K \ge 1$. Thus proving our assertion.

\vskip8pt

\noindent We need to show that the term \ref{vanishing_term} actually vanishes if $\nu_1 + \ldots + \nu_t < n$.

Recall that $\Phi$ is of type $G_2$ here. Let $\Delta = \{\alpha,\beta\}$ with $\alpha$ being the short and $\beta$ being the long root. We put

\vspace{-0.3cm}
$$\beta_1 = \alpha \,, \hskip8pt \beta_2 = \beta \,, \hskip8pt \beta_3 = \alpha+ \beta \,, \hskip8pt \beta_4 = 2\alpha + \beta \,, \hskip8pt \beta_5 = 3\alpha + \beta \,, \hskip8pt \beta_6 = 3\alpha + 2 \beta \;.$$


Consider the equation $n\gamma = \nu_1\beta_1 + \ldots + \nu_6\beta_6$, i.e.,

\vspace{-0.3cm}
\begin{numequation}\label{sumofpositiveroots}
n\gamma = \nu_1\alpha + \nu_2\beta + \nu_3(\alpha+ \beta) + \nu_4( 2\alpha + \beta) + \nu_5 (3\alpha + \beta) + \nu_6(3\alpha + 2 \beta)
\end{numequation}

\noindent We consider all possible cases for $\gamma$ and $I$.


{\it Case when $\gamma = \alpha$ or $\gamma = \beta$.} There is nothing to show in this case, as we can write a multiple of a simple root in one and only one way as a sum of positive roots.

\vskip8pt

{\it Case when $\gamma = \beta_5$ or $\gamma = \beta_6$.} Comparing the coefficients of $\alpha$, the equation \ref{sumofpositiveroots} implies in this case

$$3n = \nu_1 + \nu_3 + 2\nu_4 + 3\nu_5 + 3\nu_6 \,.$$


On the other hand, assuming $\nu_1 + \ldots + \nu_6 < n$ we get that $3\nu_1 + \ldots + 3\nu_6 < \nu_1 + \nu_3 + 2\nu_4 + 3\nu_5 + 3\nu_6$ which implies $2\nu_1 + 3\nu_2 + 2\nu_3 + \nu_4 < 0$ which is impossible. It remains to discuss the cases when $\gamma = \alpha + \beta$ or $\gamma  = 2\alpha + \beta$.

\vskip8pt

{\it Case when $\gamma = \alpha + \beta$.} In this case equation \ref{sumofpositiveroots} implies

\vspace{-0.3cm}
$$n = \nu_1 + \nu_3 + 2 \nu_4 + 3 \nu_5 + 3 \nu_6 \hskip10pt \mbox{and} \hskip10pt n = \nu_2 + \nu_3 + \nu_4 + \nu_5 + 2 \nu_6$$


Subtracting the equation on the right from the equation on the left and adding $\nu_2$ on both sides gives $\nu_2 = \nu_1 + \nu_4 + 2\nu_5 + \nu_6$.

\vskip8pt

(a) Suppose $I = \emptyset$ or $I = \{\alpha\}$. Then we let $\gamma' = \alpha$. Consider

$$\begin{array}{rcl}
x_\alpha^n \cdot y_1^{\nu_1} \cdot \ldots \cdot y_6^{\nu_6}.v^+ & = &
[x_\alpha^{(n)},y_1^{\nu_1} \cdot \ldots \cdot y_6^{\nu_6}].v^+ \\
& & \\
& = &
\sum_{i_1 + \ldots + i_6 = n} {n \choose i_1  \ldots  i_6} [x_\alpha^{(i_1)},y_1^{\nu_1}] \cdot \ldots \cdot [x_\alpha^{(i_6)},y_6^{\nu_6}].v^+
\end{array}$$


Moreover,

\vspace{-0.3cm}
\begin{numequation}\label{commutators}
[x_\alpha^{(i_j)},y_j^{\nu_j}] = \sum_{k_1 + \ldots + k_{\nu_j} = i_j} {i_j \choose k_1  \ldots  k_{\nu_j}} [x_\alpha^{(k_1)},y_j] \cdot \ldots \cdot [x_\alpha^{(k_{\nu_j})},y_j] \hskip4pt .
\end{numequation}

Therefore, if $\beta_j - \alpha$ is not in $\Phi \cup \{0\}$, the term \ref{commutators} vanishes if $i_j>0$. Hence $i_2 = i_6 = 0$. Moreover, if $i_j > \nu_j$, then there must be at least one $k_t \ge 2$ in the formula \ref{commutators}. However, if $\beta_j - 2\alpha$ is not in $\Phi \cup \{0\}$, the term \ref{commutators} vanishes then. And because $[x_\alpha^{(2)},y_{\alpha+\beta}] \in \frg_{-\beta+\alpha} = 0$, we get $i_3 \le \nu_3$. Similarly we see that $i_1 \le 2\nu_1$. Now suppose $n > \nu_1 + \ldots + \nu_6$ and consider the inequality

\vspace{-0.3cm}
$$n = i_1 + \ldots + i_6 > \nu_1 + \ldots + \nu_6 = 2\nu_1 + \nu_3 + 2\nu_4 + 3\nu_5 + 2\nu_6 \,.$$


\noindent Here we have used that $\nu_2 = \nu_1 + \nu_4 + 2\nu_5 + \nu_6$. As $i_2 = i_6 = 0$, $i_3 \le \nu_3$ and $i_1 \le 2\nu_1$ we get

\vspace{-0.3cm}
$$i_4 + i_5 > (2\nu_1 - i_1) + (\nu_3-i_3) + 2\nu_4 + 3\nu_5 + 2\nu_6 \ge 2\nu_4 + 3\nu_5 \,.$$


\noindent This shows that either $i_4 > 2\nu_4$ or $i_5 > 3 \nu_5$. But in each of these cases the corresponding term $[x_\alpha^{(i_4)},y_4^{\nu_4}]$ or $[x_\alpha^{(i_5)},y_5^{\nu_5}]$ vanishes. This proves our assertion in this case.

\vskip8pt

(b) Suppose $I = \{\beta\}$. Then we let $\gamma' = \beta$. Consider

$$\begin{array}{rcl}
x_\beta^n \cdot y_1^{\nu_1} \cdot \ldots \cdot y_6^{\nu_6}.v^+ & = &
[x_\beta^{(n)},y_1^{\nu_1} \cdot \ldots \cdot y_6^{\nu_6}].v^+ \\
& & \\
& = &
\sum_{i_1 + \ldots + i_6 = n} {n \choose i_1  \ldots  i_6} [x_\beta^{(i_1)},y_1^{\nu_1}] \cdot \ldots \cdot [x_\beta^{(i_6)},y_6^{\nu_6}].v^+
\end{array}$$


\noindent Moreover,

\begin{numequation}\label{commutators2}
[x_\beta^{(i_j)},y_j^{\nu_j}] = \sum_{k_1 + \ldots + k_{\nu_j} = i_j} {i_j \choose k_1  \ldots  k_{\nu_j}} [x_\beta^{(k_1)},y_j] \cdot \ldots \cdot [x_\beta^{(k_{\nu_j})},y_j] \; .
\end{numequation}

\noindent Therefore, if $\beta_j - \beta$ is not in $\Phi \cup \{0\}$, the term \ref{commutators2} vanishes if $i_j>0$. Hence $i_1 = i_4 = i_5 = 0$. Moreover, if $i_j > \nu_j$, then there must be at least one $k_t \ge 2$ in the formula \ref{commutators2}. However, if $\beta_j - 2\beta$ is not in $\Phi \cup \{0\}$, the term \ref{commutators2} vanishes then. So we get $i_3 \le \nu_3$ and $i_6 \le \nu_6$. Similarly we see that $i_2 \le 2\nu_2$. Assuming $n > \nu_1 + \ldots + \nu_6$ we get

\vspace{-0.3cm}
$$n = i_1 + \ldots + i_6 > \nu_1 + \ldots + \nu_6 \,,$$


\noindent and hence

\vspace{-0.3cm}
$$i_2 > \nu_1 + \nu_2 + (\nu_3 - i_3) + \nu_4 + \nu_5 + (\nu_6 - i_6) \ge  \nu_2 \,.$$


The left hand side of \ref{commutators2}, for $j=2$, is of weight $(i_2-\nu_2)\beta$, and $[x_\beta^{(i_2)},y_2^{\nu_2}].v^+$ is therefore $0$ or of weight $\lambda + (i_2-\nu_2)\beta$, which shows
that it must vanish. Assuming, as we may, that we had ordered the positive roots such that $\beta$ comes last (i.e., $\beta = \beta_6$ instead of $\beta = \beta_2$), we see that $x_\beta^n \cdot y_1^{\nu_1} \cdot \ldots \cdot y_6^{\nu_6}.v^+$ vanishes.

\vskip8pt

{\it Case when $\gamma = 2\alpha + \beta$.} In this case equation \ref{sumofpositiveroots} implies

\vspace{-0.3cm}
\begin{numequation}\label{relation31}
2n = \nu_1 + \nu_3 + 2 \nu_4 + 3 \nu_5 + 3 \nu_6 \hskip10pt
\end{numequation}
\begin{numequation}\label{relation32}
n = \nu_2 + \nu_3 + \nu_4 + \nu_5 + 2 \nu_6
\end{numequation}

\noindent We assume that $n > \nu_1 + \nu_2 + \nu_3 + \nu_4 + \nu_5 + \nu_6$. By \ref{relation32}, this implies

\vspace{-0.3cm}
$$\nu_2 + \nu_3 + \nu_4 + \nu_5 + 2 \nu_6 > \nu_1 + \nu_2 + \nu_3 + \nu_4 + \nu_5 + \nu_6 \,,$$


\noindent\noindent and therefore $\nu_6 > \nu_1$. In particular,

\vspace{-0.3cm}
\begin{numequation}\label{nu6}
\nu_6 > 0 \,.
\end{numequation}
\noindent Subtracting \ref{relation32} from \ref{relation31} we get $n = \nu_1 - \nu_2 + \nu_4 + 2\nu_5 + \nu_6$, and using again our assumption that $n > \nu_1 + \nu_2 + \nu_3 + \nu_4 + \nu_5 + \nu_6$ we find

\vspace{-0.3cm}
$$\nu_1 - \nu_2 + \nu_4 + 2\nu_5 + \nu_6 > \nu_1 + \nu_2 + \nu_3 + \nu_4 + \nu_5 + \nu_6 \,,$$


\noindent or, equivalently, $\nu_5 > 2\nu_2 + \nu_3$. In particular

\vspace{-0.3cm}
\begin{numequation}\label{nu5}
\nu_5 > 0 \,.
\end{numequation}
\noindent We distinguish two cases.

(a) Suppose $I = \emptyset$ or $I = \{\alpha\}$. Then we let $\gamma' = \alpha$. We arrange the positive roots in this order: $\beta_6, \beta_1, \beta_2, \beta_3, \beta_4, \beta_5$, and write elements of $U(\fru^-_\frb)$ as sums of monomials of the form $y_6^{\nu_6} \cdot y_1^{\nu_1} \cdot \ldots \cdot y_5^{\nu_5}$. Consider

\begin{numequation}\label{multfromleft3}
\begin{array}{cl}
& x_\alpha^{2n} \cdot y_6^{\nu_6} \cdot y_1^{\nu_1} \cdot \ldots \cdot y_5^{\nu_5}.v^+ \\
& \\
= & [x_\alpha^{(2n)},y_6^{\nu_6} \cdot y_1^{\nu_1} \cdot \ldots \cdot y_5^{\nu_5}].v^+ \\
& \\
= & \sum_{i_1 + \ldots + i_6 = 2n} {2n \choose i_1  \ldots  i_6} [x_\alpha^{(i_6)},y_6^{\nu_6}] \cdot [x_\alpha^{(i_1)},y_1^{\nu_1}] \cdot \ldots \cdot [x_\alpha^{(i_5)},y_5^{\nu_5}].v^+
\end{array}
\end{numequation}


\noindent Moreover,

\begin{numequation}\label{commutators3}
[x_\alpha^{(i_j)},y_j^{\nu_j}] = \sum_{k_1 + \ldots + k_{\nu_j} = i_j} {i_j \choose k_1  \ldots  k_{\nu_j}} [x_\alpha^{(k_1)},y_j] \cdot \ldots \cdot [x_\alpha^{(k_{\nu_j})},y_j] \; .
\end{numequation}


\noindent Since $[x_\alpha,y_6] = 0$, it follows from \ref{commutators3} that we only need to consider tuples $(i_1, \ldots, i_6)$ for which $i_6 = 0$. For those tuples we have $2n = i_1 + i_2 + i_3 + i_4 + i_5$. On the other hand, \ref{relation31} implies

\vspace{-0.3cm}
\begin{numequation}\label{upshot1}
\nu_1 + \nu_3 + 2\nu_4 + 3\nu_5 = 2n - 3\nu_6 < 2n \,,
\end{numequation}


\noindent because $\nu_6 > 0$ (cf. \ref{nu6}). Now consider the term

$$[x_\alpha^{(i_1)},y_1^{\nu_1}] \cdot \ldots \cdot [x_\alpha^{(i_5)},y_5^{\nu_5}].v^+$$


\noindent in the third line of \ref{multfromleft3}. It has weight

\vspace{-0.3cm}
$$\begin{array}{cl}
& (i_1 + i_2 + i_3 + i_4 + i_5)\alpha - \nu_1\beta_1 - \nu_2\beta_2 - \nu_3\beta_3 - \nu_4\beta_4 -\nu_5\beta_5 + \lambda\\
& \\
= & (2n - \nu_1 + \nu_3 + 2 \nu_4 + 3 \nu_5)\alpha + (2n-\nu_2 + \nu_3 + \nu_4 + \nu_5)\beta + \lambda
\end{array}$$


\noindent By \ref{upshot1}, this weight is not of the form $\lambda - (\mbox{sum of positive roots})$, and must therefore vanish. Hence

\vspace{-0.3cm}
$$x_\alpha^{2n} \cdot y_6^{\nu_6} \cdot y_1^{\nu_1} \cdot \ldots \cdot y_5^{\nu_5}.v^+ = 0$$


\noindent for all $\nu$ with $n > \nu_1 + \nu_2 + \nu_3 + \nu_4 + \nu_5 + \nu_6$.

\vskip8pt

(b) Suppose $I = \{\beta\}$. Then we let $\gamma' = \alpha+\beta$. In this case we order the positive roots as follows: $\beta_5, \beta_1, \beta_2, \beta_3, \beta_4, \beta_6$, and write elements of $U(\fru^-_\frb)$ as sums of monomials of the form $y_5^{\nu_5} \cdot y_1^{\nu_1} \cdot \ldots \cdot y_4^{\nu_4} \cdot y_6^{\nu_6}$. Consider

\vspace{-0.3cm}
\begin{numequation}\label{multfromleft4}
\begin{array}{cl}
& x_{\alpha+\beta}^n \cdot y_5^{\nu_5} \cdot y_1^{\nu_1} \cdot \ldots \cdot y_4^{\nu_4} \cdot y_6^{\nu_6}.v^+ =  [x_{\alpha+\beta}^{(n)},y_5^{\nu_5} \cdot y_1^{\nu_1} \cdot \ldots \cdot y_4^{\nu_4} \cdot y_6^{\nu_6}].v^+ \\
& \\
= & \sum_{i_1 + \ldots + i_6 = n} {n \choose i_1  \ldots  i_6} [x_{\alpha+\beta}^{(i_5)},y_5^{\nu_5}]\cdot [x_{\alpha+\beta}^{(i_1)},y_1^{\nu_1}] \cdot \ldots \cdot [x_{\alpha+\beta}^{(i_4)},y_4^{\nu_4}] \cdot [x_{\alpha+\beta}^{(i_6)},y_6^{\nu_6}].v^+
\end{array}
\end{numequation}

\noindent Moreover,

\begin{numequation}\label{commutators4}
[x_{\alpha+\beta}^{(i_j)},y_j^{\nu_j}] = \sum_{k_1 + \ldots + k_{\nu_j} = i_j} {i_j \choose k_1  \ldots  k_{\nu_j}} [x_{\alpha+\beta}^{(k_1)},y_j] \cdot \ldots \cdot [x_{\alpha+\beta}^{(k_{\nu_j})},y_j] \hskip4pt .
\end{numequation}


\noindent Since $[x_{\alpha+\beta},y_5] = 0$, it follows from \ref{commutators4} that we only need to consider tuples $(i_1, \ldots, i_6)$ for which $i_5 = 0$. For those tuples we have $n = i_1 + i_2 + i_3 + i_4 + i_6$. On the other hand, \ref{relation32} implies

\vspace{-0.3cm}
\begin{numequation}\label{upshot2}
\nu_2 + \nu_3 + \nu_4 + 2\nu_6 = n - \nu_5 < n \,,
\end{numequation}


\noindent because $\nu_5 > 0$ (cf. \ref{nu5}). Now consider the term

\vspace{-0.3cm}
$$[x_\alpha^{(i_1)},y_1^{\nu_1}] \cdot \ldots \cdot [x_\alpha^{(i_4)},y_4^{\nu_4}] \cdot [x_\alpha^{(i_6)},y_6^{\nu_6}].v^+$$


\noindent in the last line of \ref{multfromleft4}. It has weight

\vspace{-0.3cm}
$$\begin{array}{cl}
& (i_1 + i_2 + i_3 + i_4 + i_6)(\alpha + \beta) - \nu_1\beta_1 - \nu_2\beta_2 - \nu_3\beta_3 - \nu_4\beta_4 -\nu_6\beta_6 + \lambda\\
& \\
= & (n - \nu_1 + \nu_3 + 2 \nu_4 + 3 \nu_5)\alpha + (n-\nu_2 + \nu_3 + \nu_4 + 2\nu_6)\beta + \lambda
\end{array}$$


\noindent By \ref{upshot2}, this weight is not of the form $\lambda - (\mbox{sum of positive roots})$, and must therefore vanish. Hence

\vspace{-0.3cm}
$$x_{\alpha + \beta}^n \cdot y_5^{\nu_5} \cdot y_1^{\nu_1} \cdot \ldots \cdot y_4^{\nu_4} \cdot y_6^{\nu_6}.v^+ = 0$$


\noindent for all $\nu$ with $n > \nu_1 + \nu_2 + \nu_3 + \nu_4 + \nu_5 + \nu_6$.
This completes the proof. \qed

\bibliographystyle{plain}
\bibliography{JoHo}

\end{document}